\documentclass[10pt, letterpaper, oneside]{amsbook}
\usepackage{amsmath}
\usepackage{mathtools}
\usepackage{amssymb}
\usepackage{amsthm}
\usepackage{mathrsfs}

\usepackage{bbm}
\usepackage{bm}
\usepackage{stmaryrd}
\usepackage[bookmarksopen,bookmarksdepth=2]{hyperref} 
\usepackage[margin=1in]{geometry}
\usepackage{xcolor}
\usepackage{enumitem}
\usepackage{epsfig}
\usepackage{float}

\usepackage{pgfplots}
\usetikzlibrary{patterns,hobby}
\usetikzlibrary{decorations.pathmorphing}
\usetikzlibrary{calc}

\allowdisplaybreaks
\usepackage{xspace}

\usepackage{alan_style}

\newcommand{\hair}{\ifmmode\mskip1mu\else\kern0.08em\fi}

\newcommand{\Var}{\mathrm{Var}}
\newcommand{\E}{\mathbb{E}}
\renewcommand{\P}{\mathbb{P}}

\newcommand{\A}{\mathcal{A}}
\newcommand{\F}{\mathcal{F}}
\newcommand{\FBB}{\mathcal{F}_{\mathrm{BB}}}
\newcommand{\R}{\mathbb{R}}
\newcommand{\N}{\mathbb{N}}
\newcommand{\Cov}{\mathrm{Cov}}
\renewcommand{\tilde}{\widetilde}

\newcommand{\one}{\mathbbm{1}}
\newcommand{\eps}{\varepsilon}
\newcommand{\B}{\mathcal{B}}

\newcommand{\PF}{\mathbb{P}_{\!\mathcal F}}
\newcommand{\EF}{\mathbb{E}_{\!\mathcal F}}
\newcommand{\dip}{d_\mathrm{ip}}

\newtheorem{theorem}{Theorem}[chapter]
\newtheorem{proposition}[theorem]{Proposition}
\newtheorem{lemma}[theorem]{Lemma}
\newtheorem{corollary}[theorem]{Corollary}

\theoremstyle{definition}
\newtheorem{definition}[theorem]{Definition}

\theoremstyle{definition}
\newtheorem{remark}[theorem]{Remark}

\theoremstyle{definition}
\newtheorem*{notation}{Notation}

\makeatletter
\def\subsection{\@startsection{subsection}{2}%
  \z@{.5\linespacing\@plus.7\linespacing}{-.5em}%
  {\normalfont\bfseries}}
\makeatother

\makeatletter
\def\subsubsection{\@startsection{subsubsection}{3}%
  \z@\z@{-.5em}%
  {\normalfont\itshape}}
\makeatother

\makeatletter
\renewenvironment{proof}[1][\proofname]{\par
  \pushQED{\qed}%
  \normalfont \topsep6\p@\@plus6\p@\relax
  \trivlist
  \normalparindent\z@
  \item[\hskip\labelsep
        \scshape
    #1\@addpunct{.}]\ignorespaces
}{%
  \popQED\endtrivlist\@endpefalse
}
\makeatother



\newcommand{\tent}{\mathrm{Tent}}
\newcommand{\jcost}{V}
\newcommand{\scost}{S} 

\newcommand{\pole}{P}
\renewcommand{\L}{\mathcal L}
\newcommand{\fav}{\mathsf{Fav}}
\newcommand{\phiend}{\phi_\mathrm{end}}
\newcommand{\phistart}{\phi_\mathrm{start}}

\newcommand{\nt}{\mathsf{NT}}
\newcommand{\nz}{\mathsf{NZ}}
\newcommand{\maxloc}{\mathsf{MaxLoc}}
\newcommand{\numnt}{\mathsf{NumNT}}
\newcommand{\numnz}{\mathsf{NumNZ}}
\newcommand{\Corner}{\mathrm{Corner}}
\newcommand{\pass}{\mathsf{Pass}}
\newcommand{\bme}{B_{\mathrm{me}}}
\renewcommand{\epsilon}{\eps}

\newcommand{\intint}[1]{\llbracket 1,#1 \rrbracket}

\renewcommand{\mf}{\mathfrak}
\newcommand{\mrm}{\mathrm}

\renewcommand{\emptyset}{\text{\O}}

\title{Brownian structure in the KPZ fixed point}

\author{Jacob Calvert}
\address{Department of Statistics\\
 U.C. Berkeley \\
  451 Evans Hall \\
  Berkeley, CA, 94720-3840 \\
  U.S.A.}
  \email{jacob\_calvert@berkeley.edu}

 \author{Alan Hammond}
\address{Departments of Mathematics and Statistics\\
 U.C. Berkeley \\
  899 Evans Hall \\
  Berkeley, CA, 94720-3840 \\
  U.S.A.}
 \email{alanmh@berkeley.edu}

 \author{Milind Hegde}
\address{Department of Mathematics\\
 U.C. Berkeley \\
  1039 Evans Hall \\
  Berkeley, CA, 94720-3840 \\
  U.S.A.}
 \email{milind.hegde@berkeley.edu}


\begin{document}

\begin{abstract}
Many models of one-dimensional local random growth are expected to lie in the Kardar-Parisi-Zhang (KPZ)
universality class. For such a model, the interface profile at advanced time may be viewed in scaled coordinates specified via characteristic KPZ scaling exponents of one-third and two-thirds.
When the long time limit of this scaled interface is taken, it is expected---and proved for a few integrable models---that, up to a parabolic shift, the Airy$_2$ process $\mathcal{A}:\R \to \R$ is obtained. 
This process may be embedded via the Robinson-Schensted-Knuth correspondence as the uppermost curve in an $\mathbb{N}$-indexed system of random continuous curves, the Airy line ensemble.

Among our principal results is the assertion that the Airy$_2$ process enjoys a very strong similarity to Brownian motion $B$ (of rate two) on unit-order intervals.  This result yields bounds on the Airy$_2$ probabilities of a large class of events from the counterpart bounds on Brownian motion probabilities. The result has the consequence that the Radon-Nikodym derivative of the law of $\mathcal{A}$ on say $[-1,1]$, with respect to the law of $B$ on the same interval, lies in every $L^p$ space for $p \in (1,\infty)$.  In fact, the quantitative comparison of probability bounds we prove also holds for the scaled energy profile with Dirac delta initial condition of the model of Brownian last passage percolation, a model that lies in the KPZ universality class and in which the energy of paths in a random Brownian environment is maximised.

Our technique of proof harnesses a probabilistic resampling or {\em Brownian Gibbs} property satisfied by the Airy line ensemble after parabolic shift, and this article develops Brownian Gibbs analysis of this ensemble begun in~\cite{corwin2014brownian} and pursued in~\cite{hammond2017brownian}. Our Brownian comparison for scaled interface profiles is an element in the ongoing programme of studying KPZ universality via probabilistic and geometric methods of proof, aided by limited but essential use of integrable inputs. Indeed, the comparison result is a useful tool for studying this universality class. We present and prove several applications, concerning for example the structure of near ground states in Brownian last passage percolation, or Brownian structure in scaled interface profiles that arise from evolution from any element in a  very general class of initial data. 

\end{abstract}

\maketitle

\renewcommand{\baselinestretch}{1.2}\normalsize
\tableofcontents

\chapter{Introduction}
\section{KPZ universality}\label{s.intro}

The field of Kardar-Parisi-Zhang (KPZ) 
universality concerns one-dimensional interfaces that evolve randomly in time, and universal random structures that describe, independently of the microscopic details that specify the local evolution of such random models, the geometry and fluctuation of the interface when time is advanced. The KPZ universality class, whose members are random processes that are expected to evince these late-time characteristics, is very broad. The basic features of a random growth model that indicate that it may be a member of the KPZ universality class are that growth occurs in a direction normal to the present local slope of the interface at a rate influenced by the slope, alongside two competing forces: a smoothing effect generated by surface tension, and a roughening effect caused by forces in the environment that are local and random. 

A fundamental example of a model of local random growth that is expected to demonstrate at advanced time all pertinent geometric features for membership of the KPZ class was already considered in the seminal work~\cite{kardar1986dynamic}. This is the solution  $\mathcal{H}: [0,\infty)\times \R \to \R$ (called a \emph{height} function) of the Kardar-Parisi-Zhang stochastic PDE, in which the interface profile at time $t \geq 0$ is modeled by the random process $\R \to \R: x \mapsto \mathcal{H}(t,x)$ where $\mathcal{H}$, formally speaking, solves the equation 
\begin{align}\label{eq:KPZ}
\partial_t \mathcal{H}(t,x) = \frac{1}{2}\partial^2_x \mathcal{H}(t,x) + \frac{1}{2} (\partial_x \mathcal{H}(t,x))^2 + \xi(t,x) \, .
\end{align}
The rate of interface growth above a point, represented by the left-hand term, takes the form of a smoothening Laplacian term modified by
a slope dependent term---represented by the squared gradient---and a roughening induced by the space-time Gaussian white noise $\xi$, which is a field of randomness that is independent between distinct space-time points.
The rigorous mathematical meaning and resulting analysis of~(\ref{eq:KPZ}) has been the subject of intense recent activity, including analytic advances such as Hairer's theory of regularity structures~\cite{hairer2013solving}; paracontrolled distributions~\cite{gubinelli2012paracontrolled,gubinelli2017kpz}; and energy solutions~\cite{gonccalves2014nonlinear}. 
The equation does, however, have a physically relevant {\it Cole-Hopf} solution that is straightforward to specify. This solution is defined by $\mathcal{H}(t,x) = \log\mathcal{Z}(t,x)$, where $\mathcal{Z}$ is the solution to the stochastic heat equation (SHE)
\begin{align}\label{eq:SHE}
\partial_t \mathcal{Z}(t,x) &= \frac{1}{2}\partial^2_x \mathcal{Z}(t,x) + \xi(t,x) \mathcal{Z}(t,x) \, .
\end{align}
The fundamental solution to the SHE---which is known by the name ``narrow wedge'' and which we may denote by $\mathcal{Z}^{\mathbf{nw}}(t,x)$---has initial data $\mathcal{Z}(0,x)$ set equal to a Dirac delta function $\delta(x)$ at $x=0$. We may think of local random growth as occurring from a seed location at the origin $x=0$ and at time zero.

We will say nothing further about the KPZ equation itself, beyond using this important example as a convenient means of explaining the basic scaling that a model in the KPZ class verifies at late time; scalings that are needed in order to introduce {\em scaled coordinates} through the lens of which we are able to view at advanced time the canonical scaled random structures at the heart of the KPZ universality class. 

A pair of scaling exponents---one-third and two-thirds---dictate the form of these scaled coordinates.  The narrow wedge interface $\mathcal{H}^{\mathbf{nw}}(t,x) : = \log \mathcal{Z}^{\mathbf{nw}}(t,x)$ above a given location, the origin say, is a random function of time, namely $[0,\infty) \to \R: t\mapsto \mathcal{H}^{\mathbf{nw}}(t,0)$. For a model in the KPZ class, the height function takes the form $at + bt^{1/3}L_t$ for certain model-dependent constants $a, b \in \R$, where the random variables $\big\{ L_t : t \geq 1 \big\}$ form a tight collection that converges in the limit of high $t$ (under initial conditions analogous to narrow wedge) to a canonical distribution which arises in the study of extreme eigenvalues of certain natural ensembles of large random matrices. This limiting distribution is the Gaussian Unitary Ensemble (GUE) Tracy-Widom distribution. That is, the interface over a given point has a dominant linear term; when the interface is centred by the subtraction of this leading term, fluctuations occur on a scale of $t^{1/3}$. For $\mathcal H^{\mathbf{nw}}(t,0)$, this was proved in~\cite{amir2011probability} with $a=-1/24$ and $b=2^{-1/3}$.

The second exponent of two-thirds describes the spatial scale on which fluctuation begins to significantly decorrelate from its value when $x = 0$. Above we may write $L_t = L_t(0)$, with a view to setting, in a consistent way,
\begin{equation}\label{e.scalednw}
  L_t(x) = b^{-1}t^{-1/3} \big( \mathcal{H}^{\mathbf{nw}}\big(t,ct^{2/3}x \big) - at \Big) \, ;
\end{equation}
The random process $L_t: \R \to \R$ then offers a scaled coordinate description of random fluctuation on the part of the narrow wedge KPZ equation. Here $c$ is a third model-dependent constant whose value for the KPZ narrow wedge solution is $2^{1/3}$ and is set to obtain a scaling limit with convenient coefficients of one, which we will discuss shortly. The factor of $t^{2/3}$ that appears against the spatial variable $x$ anticipates that it is spatial variation of order $t^{2/3}$ that leads to non-trivial correlation for fluctuation.

The principal results of this monograph offer a very strong assertion concerning the geometry of the canonical random object---the Airy$_2$ process with parabolic curvature---that describes the scaled attributes of the narrow wedge solution to models in the KPZ universality class. Our results further offer a counterpart description that holds in a rather uniform sense in the prelimit 
 for a random model in the KPZ class known as {\em Brownian last passage percolation} (LPP). The geometric inference that we make is a powerful assertion of the Brownian nature of the scaled interface profile---the analogue of $L_t$ in the above notation---and it is the engine for a wide array of applications about scaled random growth models and their KPZ universality limiting structure. 

 (We note in brief that ``prelimit'' in the previous paragraph does \emph{not} correspond to finite $t$ in the preceding discussion on the KPZ narrow wedge solution. The correct prelimit analogue of finite $t$ is a \emph{positive} temperature version of LPP, which is itself a zero temperature model, with the temperature being scaled to $\infty$ at a particular $t$-dependent rate. The interested reader is referred to \cite{alberts2014continuum,alberts2014intermediate} for more details.)

The parabolically shifted Airy$_2$ process is the first in a family of limiting processes, indexed by the initial condition, which are expected to be universal objects in the KPZ universality class, and which may be referred to collectively as the \emph{KPZ fixed point}. The name arises from the physical view that such objects are expected to be fixed points of suitable renormalization operators. (Indeed, an important related object has been constructed by~\cite{matetski2016kpz} in a paper bearing the name ``The KPZ fixed point''. 
Our usage of this term is a corruption of theirs, but only mildly so. We will comment further on~\cite{matetski2016kpz}  in Section~\ref{s.intro kpz fixed point}.) One application of our results will be a statement about a certain form of Brownian regularity for these KPZ fixed point profiles, i.e., limiting scaled interface profiles from general initial data in Brownian LPP.

We defer the definition of Brownian last passage percolation and the presentation of our main theorem, expressed in prelimiting terms that capture Brownian LPP, to Section~\ref{s.notation brownian lpp}. What we are able to indicate in the ensuing paragraphs is the form of our principal assertion in the limiting case of scaled KPZ structure, corresponding to the limit of high~$t$ for the scaled narrow wedge solution of the KPZ solution. Indeed, we will next use the lens of scaled coordinates offered by this scaled solution to specify the pertinent limiting object, namely the parabolic Airy$_2$ process; and then we state our principal conclusion as it applies to this process.

The Airy$_2$ process $\mc{A}:\R \to \R$ is a stationary process first introduced by Pr\"ahofer and Spohn \cite{prahofer2002scale} 
in a scaled description of the polynuclear growth (PNG) model; or equivalently, of another famous last passage percolation model, Poissonian LPP. 
 It was defined in \cite{prahofer2002scale} in terms of its finite-dimensional distributions---written via determinantal formulas involving the Airy kernel---and shown to have a continuous version. 
 Its basic role in KPZ universality may be expressed via the scaled narrow wedge interface $L_t:\R \to \R$ for the KPZ equation in~(\ref{e.scalednw}). It is widely expected, and it is a major open problem to rigorously show, that in the distributional limit of high $t$ with respect to the topology of locally uniform convergence on the space of continuous functions with domain and co-domain the real line, the process $x \mapsto L_t(x)$
 converges to a limit that takes the form $\mc{A}(x) - x^2$ of a parabolically shifted Airy$_2$ process. We set $c=2^{1/3}$ in \eqref{e.scalednw} so that the coefficients of both terms are one.
This inference is conjectural for the KPZ equation, but it has been validated for several stochastic growth models with narrow wedge initial data beyond PNG.  These models include the totally asymmetric simple exclusion process (TASEP) \cite{borodin2008large} as well as the model that will be the principal object of rigorous attention in this monograph, namely  Brownian LPP.  In the latter case, the convergence is proved via a distributional relation with Dyson Brownian motion that will be reviewed in Section~\ref{s.notation brownian lpp}. 


\subsection{Locally Brownian nature of the limiting process} In the case of Brownian LPP, as with the KPZ narrow wedge solution, the limiting process is $\A(x) - x^2$. We define 
$$\L(x):=2^{-1/2}\left(\A(x) - x^2\right),$$
and call it the \emph{parabolic Airy$_2$ process}, in spite of the factor of $2^{-1/2}$, as this is a main object of study that benefits from a shorter name than ``parabolically shifted Airy$_2$ process". The factor $2^{-1/2}$ is included to make comparisons with Brownian motion more convenient and will be made clearer momentarily. 

The limiting process globally adopts a parabolic form, but it is locally Brownian---see Figure~\ref{f.locally brownian}. 
The term ``locally Brownian'' may be interpreted in several ways, with a progression to stronger forms of interpretation, reflecting recent progress in understanding this limiting scaled profile. ``Locally Brownian'' could mean that, for any given $x \in \R$, the distributional process limit of $y \mapsto \e^{-1/2}\mc{L}(x + \e y)$ as $\e \searrow 0$ is standard Brownian motion (where it is the presence of the factor $2^{-1/2}$ in the definition of $\L$ that permits the diffusion rate to equal one, as $\A$ itself is locally of rate two). H\" agg 
 proved such a convergence for finite-dimensional distributions in~\cite{hagg2008} by analysing determinantal structure implicated in the definition of the Airy$_2$ process.

\begin{figure}[t]
\centering{\epsfig{file=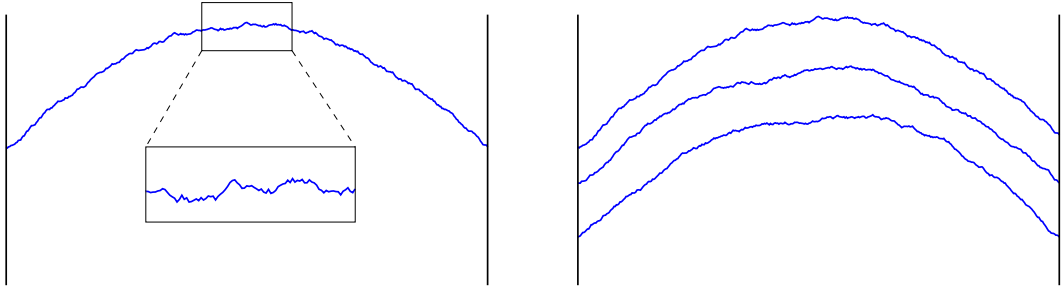, scale=0.8}}
\caption{On the left panel is an illustration of the locally Brownian, globally parabolic nature of the limiting process $\mc{L}$. On the right, we depict the top three curves of the parabolic Airy line ensemble, a collection of continuous, non-intersecting curves whose top line is $\L$.\label{f.locally brownian}}
\end{figure}

``Locally Brownian'' could mean the stronger assertion that the process $\mc{L} - \mc{L}(a)$, when restricted to any given compact interval $[a,b]$,
is absolutely continuous with respect to standard Brownian motion on this same interval. This probabilistic assertion does not seem amenable to approaches that utilise directly the determinantal structure of the Airy$_2$ process. It was proved in~\cite{corwin2014brownian} by a technique that lies at the heart of the investigation of the present monograph. This Brownian Gibbs technique involves embedding the process $\mc{L}: \R \to \R$ as the uppermost curve in an infinite system of random continuous non-intersecting curves, called the parabolic Airy line ensemble
(see the right panel of Figure~\ref{f.locally brownian}). The ensemble of curves may be viewed as a growing system of mutually avoiding rate one Brownian bridges viewed at their edge (i.e., any fixed number of the uppermost curves as the number of curves in the system grows). This implies that the uppermost curves have an attractive Gibbs resampling property involving the rejection sampling of independent Brownian bridges on a condition of avoidance of the lower curves. 

However, to say merely that one measure $\nu$ is absolutely continuous with respect to another $\mu$ is to leave unquantified the relation between $\e$ and $\delta$ in the assertion that $\mu(A) < \epsilon$ implies that $\nu(A) < \delta$. For example, the relation $\delta = \e^{1 - 1/p - o(1)}$ will be valid as $\epsilon \searrow 0$, for given $p \in [1,\infty)$, if it is the case that the Radon-Nikodym derivative of $\nu$ with respect $\mu$ lies in $L^p({\rm d}\mu)$.

Perhaps by ``locally Brownian'', what we mean is that the Radon-Nikodym derivative of $\mc{L} - \mc{L}(a)$ on $[a,b]$ with respect to standard Brownian motion on the same interval
lies in $L^{\infty-}$, i.e., in every $L^p$ space, for $p \geq 1$; if so, this term would indicate that an event whose Brownian motion probability is a low value $\epsilon \in (0,1)$ would have probability at most $\epsilon^{1 - o(1)}$ for  $\mc{L} - \mc{L}(a): [a,b] \to \R$. Our principal result, as it applies to limiting structure, establishes that this is the case. It is shown that the error factor $\epsilon^{-o(1)}$ in the latter probability $\epsilon^{1 - o(1)}$ may take the form $\exp \big\{ O(1) (\log \epsilon^{-1})^{5/6} \big\}$. 

\begin{theorem}\label{T.MAIN THEOREM FOR AIRY}
Let $\A:\R\to\R$ be the Airy$_2$ process; let $\L:\R\to \R$ be given by $2^{-1/2}\left(\A(x) - x^2\right)$; and, for fixed $d>0$, let $\mc C$ be the space of continuous functions on the interval $[-d,d]$ that vanish at $-d$. Let $d\geq 1$, let $A$ be a Borel measurable subset of $\mc C$, and let $\epsilon = \B(A)$, where $\B$ is the law of standard Brownian motion on $[-d,d]$. There exists $\epsilon_0=\epsilon_0(d)>0$ and an absolute finite constant $G$, such that, if $\epsilon \in (0, \epsilon_0)$, then
$$\P\Big(\L(\,\cdot\,) - \L(-d) \in A\Big) \leq \epsilon \cdot \exp\left(Gd(\log \epsilon^{-1})^{5/6}\right).$$
\end{theorem}

Theorem~\ref{T.MAIN THEOREM FOR AIRY} is the culmination of a probabilistic study of narrow wedge KPZ structure across several papers, including~\cite{corwin2014brownian} and \cite{hammond2017brownian}. The form of Brownian comparison made by this result and by its upcoming prelimiting counterpart Theorem~\ref{t.airytail.ln} is strong enough to open up an exciting array of applications 
concerning KPZ universality and last passage percolation models. There are six applications that we discuss in this monograph---two or three are simple and direct, while our treatment of another will take the form of a summary of upcoming work in which the Brownian comparison proved here will play a foundational role; and the final two concern not narrow-wedge but \emph{general} initial conditions, which we will discuss in Section~\ref{s.intro general init condition Brownian regularity}.

We want to indicate promptly several of these applications. In the next section, we introduce an LPP model, Bernoulli LPP, with a simple and pleasing definition. We prove nothing about it rigorously---as mentioned earlier, it is Brownian LPP about which we will prove new results---but we hope that introducing the model will provide a helpful alternative introduction to basic considerations such as scaled coordinates for LPP; and reference to the model will aid in our presentation of applications in Section~\ref{s.intro applications}.

The road to results such as Theorem~\ref{T.MAIN THEOREM FOR AIRY} via the probabilistic Brownian Gibbs technique begins with the absolute continuity comparison made by~\cite{corwin2014brownian}. An important intermediate step was achieved in~\cite{hammond2017brownian}, in which a comparison formally very similar to that made by   Theorem~\ref{T.MAIN THEOREM FOR AIRY} was made, but with the compared processes being affinely shifted so that their interval endpoint values vanish. The relation of  Theorem~\ref{T.MAIN THEOREM FOR AIRY} to this counterpart result in~\cite{hammond2017brownian}
is important both for its formal similarity and its striking differences; and for the technique of proof. Indeed, our proof of  Theorem~\ref{T.MAIN THEOREM FOR AIRY} will harness a substantial part of the technical apparatus of~\cite{hammond2017brownian}, but employ it in a substantially new manner. We turn to a more detailed overview of previous work in Section~\ref{s.intro.previous work} and of the relations of the present work to that of~\cite{hammond2017brownian} in the final section of the introduction, Section~\ref{s.intro method of proof}.


 



\subsection{A further heuristic overview, via Bernoulli LPP}\label{s.intro.bernoulli}

\begin{figure}[t]
\centering{\epsfig{file=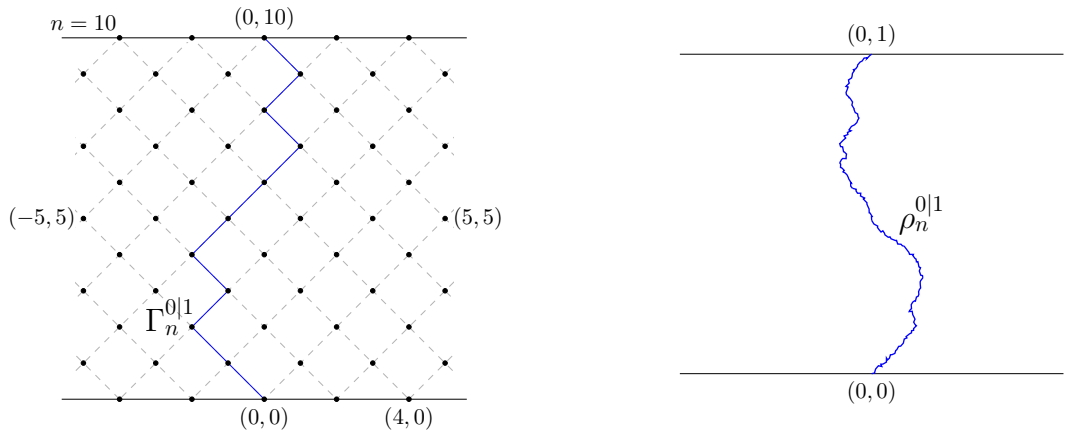, scale=0.85}}
\caption{On the left panel is a depiction of $\Z^2_\angle$ and its coordinate system, with a blue geodesic path $\Gamma^{0|1}_n$ from $(0,0)$ to $(0,n)$ indicated; here $n=10$. On the right panel is the polymer $\rho^{0|1}_n$ after the scaling map $S_n$ has been applied to the geodesic (not the same as in the first panel); here $n$ is unspecified but large. The right panel is simply a depiction and not a simulation.}\label{f.bernoulli}
\end{figure}

Let $\Z^2_\angle$ be the lattice depicted in Figure~\ref{f.bernoulli}, regarded as a subset of $\R^2$ via the depicted coordinate system; this lattice may be thought of as $\Z^2$ rotated by $45^\circ$ counterclockwise and scaled up by a factor of $2^{1/2}$. In Bernoulli last passage percolation, a random environment is specified by assigning independently to each vertex in $\Z^2_\angle$ a value of zero or one. The probability of assigning one equals a given value $p \in (0,1)$ for each vertex; for simplicity, we consider the case that $p =1/2$. A valid directed path in $\Z^2_\angle$ is a nearest-neighbour directed path in that lattice, i.e., each of the path's moves is northwest or northeast by one lattice unit. The energy of any such path is the sum of the values assigned to the vertices in $\Z^2_\angle$ that lie on the path. Let $M_n^{0|1} \big[ (0,0) \to (0,n) \big]$---or, more briefly, $M_n^{0|1}$---denote the maximum energy of directed paths that begin at $(0,0)$ and that end at $(0,n)$ in the coordinate system of $\Z^2_\angle$ illustrated by Figure~\ref{f.bernoulli}, with the $0|1$ indicating the Bernoulli environment. Later, the Brownian analogues of this and other quantities will be referred to by the corresponding symbols with the $0\mid1$ omitted.

The Bernoulli LPP model is widely expected to lie in the KPZ universality class. Indeed, if we set 
$a_n =  n^{-1}\E M^{0|1}_n$, a simple subadditivity argument yields the existence of the leading order growth coefficient $a = \lim a_n$. It is easily seen that $a \in (1/2, 1)$. Indeed, $a$ is the limiting expected proportion of sites on any maximum weight path (necessarily consisting of $n+1$ vertices) which are assigned a value of one by the random environment, so that $a \in [1/2,1]$; and it is an exercise to exclude the possibilities that $a$ equals one-half or one. If we further write 
\begin{equation}\label{e.an}
 M^{0|1}_n = an +  b n^{1/3} \weight^{0|1}_n \, ,
\end{equation}
the random {\em weight} $\weight^{0|1}_n$ is a measure of the scaled fluctuation of the maximum energy for the route $(0,0) \to (0,n)$; this is the analogue of $L_t$ discussed in the context of the narrow-wedge KPZ equation. Indeed, the system of random variables $\big\{ \weight^{0|1}_n : n \in \N \big\}$, as with the system $\{L_t:t\geq 1\}$, may be expected to be tight, and, for a suitable choice of the constant $b > 0$, to converge to the GUE Tracy-Widom distribution. The maximum energy $M_n^{0|1}$ is attained by a {\em geodesic} $\Gamma^{0|1}_n$ from $(0,0)$ to $(0,n)$. This is the directed path between these endpoints whose energy equals $M^{0\mid 1}_n$. In fact, for this discrete model, there is usually more than one maximiser; but it is easily seen that there is a unique leftmost maximiser and, for definiteness, we set the geodesic $\Gamma_n^{0|1}$ equal to this path. 

Just as the weight $\weight^{0|1}_n$ offers a scaled description of energy, it is natural to represent $\Gamma_n^{0|1}$ in scaled coordinates, as a path that traverses a unit-order distance while making random fluctuation also of unit order. To set up such a description, consider the  scaling map $S_n: \R^2 \to \R^2$ that scales horizontally by $n^{2/3}$ and vertically by $n$, i.e., $(0,0)$ is mapped to itself; $(0,n)$ is mapped to $(0,1)$; and $(n^{2/3},0)$ is mapped to $(1,0)$. If we treat any directed path in $\Z^2_\angle$ as a subset of $\R^2$ by viewing it as the union of the nearest-neighbour edges that it crosses, then we may set $\rho^{0|1}_n$ equal to the image under $S_n$ of the geodesic~$\Gamma^{0|1}_n$. This object $\rho^{0|1}_n$ is then the image of a piecewise affine curve in $\R^2$ that connects $(0,0)$ and $(0,1)$. 
We call it a \emph{polymer}, and regard it as a scaled version of $\Gamma^{0|1}_n$. The vertical advancement of $\Gamma^{0|1}_n$ from the origin to $(0,n)$ corresponds to a unit vertical advancement of the polymer between its endpoints. The two-thirds spatial scaling exponent for KPZ is represented by the expectation that $\Gamma^{0|1}_n$ at generic heights, on the interval of heights $(n/4,3n/4)$ say, will be at a horizontal distance from the $y$-axis of order $n^{2/3}$. The scaling map $S_n$ has thus been specified so that the polymer $\rho^{0|1}_n$ may have non-degenerate random horizontal fluctuation of unit order; see the second panel of Figure~\ref{f.bernoulli}.

(This use of the terms weight and polymer is hardly standard, but agrees with the terminology in \cite{hammond2017brownian}. Indeed, in the literature the term polymer often refers to realizations of the path measure in last passage percolation models at \emph{positive} temperature, unlike the zero temperature case discussed here.)

We may further set $\weight^{0|1}_n \big[ (0,0) \to (0,1) \big] = \weight^{0|1}_n$ and  $\rho_n^{0|1} \big[ (0,0) \to (0,1) \big] = \rho^{0|1}_n$ with a view to generalising these objects into a broader scaled description of geodesics and their energies. 
Indeed, for $x,y \in \R$, we may specify $\rho_n^{0|1} \big[ (x,0) \to (y,1) \big]$ to be the image under $S_n$ of the geodesic that runs between 
$S_n^{-1}(x,0) = \big( n^{2/3}x,0 \big)$ and $S_n^{-1}(y,1) = \big(n^{2/3}y,n \big)$. (The latter pair of planar points should lie in the lattice $\Z^2_\angle$, but we neglect this nicety in this heuristic discussion.) The polymer  $\rho^{0|1}_n \big[ (x,0) \to (y,1) \big]$ has weight  $\weight^{0|1}_n \big[ (x,0) \to (y,1) \big]$ given by the scaled expression for the energy of the mentioned geodesic. That is, we set 
$$
M^{0|1} \big[ ( n^{2/3}x,0 ) \to (n^{2/3}y,n ) \big] = a n + b n^{1/3} \weight^{0|1}_n \big[ (x,0) \to (y,1) \big] \, ,
$$
where the constants $a$ and $b$ have been specified after~(\ref{e.an}).

We will often be interested in the situation when the starting point is held fixed at the origin, i.e., $x=0$. For future reference, let us highlight a supremum formula for $\weight^{0|1}_n[(0,0)\to (y,1)]$. For $w\in(0,1)$,
\begin{equation}\label{e.supremum formula}
\weight^{0|1}_n[(0,0)\to (y,1)] = \sup_{x\in \R}\left(\weight^{0|1}_n[(0,0)\to (x,w)] + \weight^{0|1}_n[(x,w)\to (y,1)]\right).
\end{equation}
This formula is easy to understand: the expression in the supremum is the weight of the best scaled path forced to pass through $(x,w)$, and, fixing $w$, the polymer $\rho_n[(0,0)\to(y,1)]$ will clearly pass through $(x,w)$ for the best such $x$. We note also that the two terms inside the supremum are independent by the independence of the environment, and that, if we have also fixed $y$, both terms can be considered as weight profiles. (This independence claim is not precisely correct as the two terms actually share the weight of a single vertex, which is microscopic. We ignore this minor point in this expository discussion.) In fact, as in this decomposition, the sum of two independent weight profiles often arises in LPP studies, and one of the applications we will discuss later, though not proved in this paper, extends Theorem~\ref{T.MAIN THEOREM FOR AIRY} and its prelimiting version Theorem~\ref{t.airytail.ln} to such objects.

The parabolic Airy$_2$ process $\mc{L}:\R \to \R$ 
is the description expected to arise in the limit of high $n$ of the scaled energy of  scaled LPP paths that, in accordance with narrow wedge initial data, emanate from the origin. Simply put, for an appropriate choice of the constant~$c>0$, $\mc{L}(y)$ equals $2^{-1/2}\cdot \weight^{0|1}_\infty \big[ (0,0) \to (cy,1) \big]$, where we set $\weight^{0|1}_\infty$ to be the putatively existing high~$n$ limit of the weight system $\weight^{0|1}_n$; the factor $2^{-1/2}$ arises from the definition of $\L$, and $c$ plays the same role as it did in \eqref{e.scalednw} in the context of the narrow wedge KPZ solution.   

With the example of Bernoulli LPP in mind, we may move to describing the applications of our main results.



\section{Applications of Theorem~\ref{T.MAIN THEOREM FOR AIRY}} \label{s.intro applications}

We have postponed the full statement of our main theorem as it requires the introduction of more general objects than just the parabolic Airy$_2$ process. As indicated, these general objects arise naturally as weight profiles in Brownian last passage percolation, which we will define in Section~\ref{s.notation and regular ensembles}. It is because of this example of Brownian LPP that we have chosen to prove our results in a more general framework than Theorem~\ref{T.MAIN THEOREM FOR AIRY}, and we anticipate a number of future applications that make use of this general form. 

\begin{remark}\label{r.applications extend to weight profile}
While the applications proved in this paper are stated mainly for the parabolic Airy$_2$ process, we point out that essentially the same statements can be obtained for the prelimiting Brownian LPP narrow-wedge weight profile in each case by replacing the usage of Theorem~\ref{T.MAIN THEOREM FOR AIRY} in the proofs with our upcoming main result Theorem~\ref{t.airytail.ln}, which is applicable to Brownian LPP. This is simply because the underlying Brownian motion estimates which serve as input to Theorem~\ref{T.MAIN THEOREM FOR AIRY} in the each of the application's proofs can equally well be used as input for Theorem~\ref{t.airytail.ln}.
\end{remark}

Our main result is a powerful tool, and 
we present six applications. The first two are presented in the next two sections. The first is a simple corollary of Theorem~\ref{T.MAIN THEOREM FOR AIRY} and its proof is given immediately; while the second (and also the third) concerns  near maxima of the Airy$_2$ process and requires more involved arguments with Brownian motion, and their proofs are given in Chapter~\ref{ch.application proofs}. Several of the applications
can be viewed, in a rough but we hope profitable sense, as part of a theme that concerns energy landscapes, their valleys, and chaotic trajectories in these landscapes. In Section~\ref{s.intro.energy landscape}, we briefly survey connections in this vein elucidated  in Sourav Chatterjee's work \cite{chatterjee2014superconcentration}.
After presenting in Section~\ref{s.intro.many near touch} a result concerning the improbability of many well-separated near maximisers,
we turn in Section~\ref{s.intro.rooted weight profile} to our fourth application, an extension of our main result to an object that often appears in LPP problems. In contrast to the others, we merely state this application, which will appear a forthcoming work \cite{LPPtools} of Shirshendu Ganguly and the second author.
The final two applications concern a quantified notion of Brownian regularity for scaled Brownian LPP energy profiles begun from \emph{general} initial conditions. 
After a review of pertinent recent work in KPZ in Section~\ref{s.intro.previous work}, they  
are described in Section~\ref{s.intro general init condition Brownian regularity}, with proofs appearing in Chapter~\ref{ch.patchwork quilt}.


\subsection{Movement of Airy$_2$ in an interval}\label{s.intro.simple application}
An immediate application of Theorem~\ref{T.MAIN THEOREM FOR AIRY} which illustrates its utility is the following corollary, which gives a tail bound on the amount the Airy process or its parabolic version moves in a unit-order interval.
\begin{corollary}\label{c.airy process movement in interval}
Let $d\geq 1$. Then there exist $C<\infty$, $C'<\infty$ and $x_0>0$ such that, for $x>x_0$,
\begin{align*}
&\P\left(\sup_{s\in[-d,d]} |\L(s) -\L(-d)| \geq x\right) \leq e^{-x^2/4d + Cd^{1/6} x^{5/3}}
\end{align*}
and
\begin{align*}
&\P\left(\sup_{s\in[-d,d]} |\A(s) -\A(-d)| \geq x\right) \leq e^{-x^2/8d + C'd x^{5/3}}.
\end{align*}
\end{corollary}
Recall that $\L$ and $\A$ are respectively, in an idealised sense, rate one and rate two processes, and that we are considering the tail probability of an increment over an interval of length $2d$. For Brownian motions of rates one and two, these probabilities can be respectively understood as being roughly $\exp(-x^2/4d)$ and $\exp(-x^2/8d)$, and this accounts for the dominant terms in the exponents in the bounds in Corollary~\ref{c.airy process movement in interval}; the remaining terms in the exponents of the form $d^{1/6}x^{5/3}$ or $dx^{5/3}$ are sub-dominant corrections arising from Theorem~\ref{T.MAIN THEOREM FOR AIRY}.

We remark that estimates on similar quantities have previously appeared in the literature with a weaker tail bound exponent of $3/2$, instead of $2$ as obtained here, such as in \cite[Proposition 1.6]{dauvergne2018basic} and \cite[Theorem 2.14]{hammond2017brownian}; additionally, we obtain an explicit coefficient for the $-x^2$ term in the exponent, as well as a quantified sub-dominant correction. The two just cited estimates are for the prelimiting Brownian LPP weight profiles, denoted in this paper by $\weight_n[(0,0)\to(\cdot,1)]$ but, as mentioned in Remark~\ref{r.applications extend to weight profile} above, the statement of Corollary~\ref{c.airy process movement in interval} can be easily obtained for $\weight_n$ as well.

\begin{proof}[Proof of Corollary~\ref{c.airy process movement in interval}]
Let $B$ be standard Brownian motion on $[-d,d]$ started at zero. Writing $N(0,2d)$ for a normal random variable with mean zero and variance $2d$, we see that, by the reflection principle for Brownian motion,
$$\P\left(\sup_{s\in[-d,d]} |B(s)| \geq x\right) \leq 2\cdot\P\left(\sup_{s\in[-d,d]} B(s) \geq x\right) = 4\cdot \P\Bigl(N(0,2d) \geq x\Bigr) \leq 4\cdot e^{-x^2/4d}.$$
The last inequality is due to the Chernoff bound. Now we apply Theorem~\ref{T.MAIN THEOREM FOR AIRY} after letting $x$ be large enough that the last quantity is less than the $\epsilon_0$ from Theorem~\ref{T.MAIN THEOREM FOR AIRY}, and raise the value of $G$ obtained from Theorem~\ref{T.MAIN THEOREM FOR AIRY} further to absorb the multiplicative constant of $4$, to get the first bound in Corollary~\ref{c.airy process movement in interval}. The second follows from the first by noting that $\sup_{s\in[-d,d]}|\L(s)-\L(-d)|$ differs from $2^{-1/2}\sup_{s\in[-d,d]}|\A(s)-\A(-d)|$ by at most $2^{-1/2}d^2$, and by bounding $xd/4 + Cd^{1/6}x^{5/3}$ by $C'dx^{5/3}$, where $C'$ is defined by say $C' = C+1$.
\end{proof}

The proof of Corollary~\ref{c.airy process movement in interval} illustrates that the usefulness of Theorem~\ref{T.MAIN THEOREM FOR AIRY} lies in allowing us to use all of the many powerful probabilistic tools and symmetries available for Brownian motion and the normal distribution in the problem of estimating the probabilities of very naturally arising events for the parabolic Airy$_2$ process. The remaining applications make more sophisticated use of Brownian motion.

\subsection{A quantified local version of Johansson's conjecture}\label{s.intro.local johansson}
As we saw in the discussion on Bernoulli LPP in Section~\ref{s.intro.bernoulli}, the weight profile $\weight^{0|1}_n[(0,0)\to (y,1)]$ captures the weight of the polymer path as the endpoint $(y,1)$ varies.
When the starting point $(0,0)$ is held fixed, the parabolic Airy$_2$ process is the limiting weight profile process as $n$ tends to infinity in a number of LPP models, and so it is of interest to understand its maximiser---this corresponds to maximum weight scaled paths with fixed starting points. Johansson conjectured in \cite{johansson2003discrete} that the process $\L$ defined above almost surely has a unique global maximiser, a fact that now has a number of proofs: Corwin and Hammond \cite{corwin2014brownian} establish this via comparison to Brownian motion; Moreno Flores, Quastel, and Remenik \cite{flores2013endpoint} prove an explicit formula for the maximiser; and an argument of Pimentel \cite{pimentel2014location} shows that any stationary process minus a parabola has a unique maximiser.

In many LPP situations, it is of interest not only to look at the energy-maximising path, but also at paths which are \emph{nearly} energy-maximising, as these would perhaps become the maximising path under a small perturbation of the random environment. (In the context of Bernoulli LPP, the existence of nearly energy-maximising paths can also be interpreted as the non-uniqueness of the maximising path.) So, it is useful to know that paths far away from the maximising path are not very close in energy to the maximum. 

In terms of the parabolic Airy$_2$ process, this corresponds to understanding the occurrence of \emph{near maxima} at some distance from the unique maximiser, and this is our next focus of study. In order to respect the Brownian scaling we expect to see on unit order intervals, the closeness of the near maximum energy we consider has a square root relationship with the magnitude of the separation we impose from the unique actual maximiser. We prove a result bounding the probability of an event of this type, which we refer to as a quantified local Johansson result.

Let $X:[-d,d]\to\R$ be a stochastic process with an almost surely unique maximiser in $[-d,d]$, and let this unique maximiser be $x_{\max}$. Let $M = X(x_{\max})$. For $a,\eta \in(0,1)$, define the \emph{near touch} event $\nt(X, \eta, a)$ by
$$\nt(X,\eta,a) := \left\{\sup_{|z|\geq 1} X(x_{\max}+z\eta) \geq M - a\eta^{1/2}\right\},$$
where for $x\not\in[-d,d]$ we take $X(x) = -\infty$ by convention. We omit the $d$-dependence of the event in the notation.

\begin{theorem}[Quantified local Johansson]\label{t.local johansson}
Let $d\geq 1$ and $a,\eta\in(0,1)$. There exist $a_0$ and $G<\infty$ such that, if $0<a< a_0$,
$$\P\big(\nt(\L,\eta,a)\big)
 \leq a\cdot \exp\left(Gd (\log a^{-1})^{5/6}\right).$$
\end{theorem}
Note that when we consider the event $\nt(\L,\eta, a)$, we do not say that the global maximiser lies in $[-d,d]$; we are considering only near touches with the maximiser when restricted to $[-d,d]$. We note the estimate \cite[Corollary 4.6]{corwin2014brownian}, which bounds the probability that the global maximiser lies outside a given interval centred at 0. By taking large $d$, applying this estimate, and using the parabolic curvature of $\L$, along with a union bound, it is plausible that Theorem~\ref{t.local johansson} can be extended to a statement about the global maximiser. However, we do not pursue this line of reasoning further here.

Theorem~\ref{t.local johansson} is proved in Section~\ref{ch.application proofs}. The idea of the proof is again to make strong use of information known about Brownian motion. We will use the classical decomposition of Brownian motion around its maximiser in terms of independent Brownian meanders and the explicit transition probability formulas for the latter process.

\subsection{The energy landscape}\label{s.intro.energy landscape}
Before stating our next two applications in the two upcoming subsections, we pause to discuss a useful perspective on these results, that of the \emph{energy landscape}. In fact, the previous application of a quantified local Johansson result also falls within the purview of this viewpoint.

In many complex statistical mechanical systems, the structure of the energy landscape is vitally consequential. For instance, in the context of last passage percolation, we view the landscape as being the space of directed paths, and the energy associated to each path is the energy defined in Section~\ref{s.intro.bernoulli}, i.e., the sum of weights along the path, but with an extra negative sign. (This conflict of signs should not cause confusion as it is restricted to only this subsection.) Thus in this energy landscape the geodesic is the \emph{energy minimiser}; such a state of minimum energy is known as a ground state.

In recent decades the importance of the energy landscape of statistical systems for understanding challenging problems---from protein folding~\cite{onuchic1997theory} to the physics of spin glasses~\cite{mezard1987spin} to machine learning~\cite{ballard2017energy}---has been identified in an array of scientific disciplines.  This is because investigating certain properties of the energy landscape can lead to an understanding of important features of the system's behaviour. For example, how many ground states are there? How many states come close to attaining the minimum energy? (Our second application concerning near maxima is an aspect of this question in LPP.) Such near ground states form valleys in the energy landscape; can the structure of mountain passes, or routes of minimal highest energy that connect these valleys, be conveniently described? These are often difficult questions for any particular mathematical model, but their answers yield insights into the behaviour of the complex system; for instance, the existence of many valleys in the energy landscape of a LPP model has connections to the variance of the energy being sublinear in $n$ (with the prediction from KPZ universality being that the variance scales as $n^{2/3}$). For more on these ideas and a formal statement of the connection between the energy landscape and sublinearity of variance in a particular LPP model, the reader is referred to Sourav Chatterjee's monograph \cite{chatterjee2014superconcentration}.

In fact, \cite{chatterjee2014superconcentration} elucidates an additional important principle of the energy landscape of complex systems, which is that the existence of multiple valleys is essentially equivalent to the presence of \emph{chaos} in the system. For the purposes of the discussion here, we interpret the presence of chaos in a system as a high sensitivity to small perturbations of the system. Thus understanding the existence of multiple valleys and the effect of small perturbations is of interest in last passage percolation, and these two themes are present in our next two applications.

\subsection{Many near touch}\label{s.intro.many near touch}
Before presenting our third application, we discuss a Bernoulli LPP form of a well-known problem in last passage percolation, the \emph{slow bond} problem; it was solved in its Poissonian and exponential LPP forms by Basu, Sidoravicius, and Sly \cite{basu2014last}.

Fix $\epsilon>0$ and consider the unscaled independent Bernoulli environment as in Figure~\ref{f.bernoulli}, with the following change: for vertices on the vertical line $x=0$, the probability of the assigned value being one is increased from $1/2$ to $1/2+\epsilon$, with the probability of zero being reduced from $1/2$ to $1/2-\epsilon$. In the original environment, the coefficient of linear growth was some $a\in(1/2, 1)$---in the modified environment, does the coefficient become strictly larger than $a$? In other words, is the system sensitive to the microscopic perturbation of slightly increased value on the vertical line, i.e., does this change lead to a macroscopically visible increase in the geodesic energy? In the analogous perturbations of standard Poissonian or exponential LPP models, this is the question that was answered in the affirmative in \cite{basu2014last}.

A natural further question is the rate of decay as $\epsilon \searrow 0$ of the magnitude of the increase in the linear coefficient. One physical prediction~\cite{janowsky1994exact,costin2012blockage}, in the model of exponential LPP, is that the magnitude of increase should decay as $\exp(-c\epsilon^{-1})$ for some $c>0$. Recently, Allan Sly has announced~\cite{allansly2019talk} a forthcoming result, joint with Sourav Sarkar and Lingfu Zhang, that this quantity, again in the model of exponential LPP, decays at least superpolynomially in $\epsilon$ as $\epsilon$ tends to zero. 

Allan Sly has conveyed to one of us that control over the number of near maxima occurring in a bounded interval is of importance in their proof; the following statement provides such a bound for the limiting $\L$, but, as in the other applications, can also be given for the prelimiting narrow wedge Brownian LPP weight profile. Of course, the proof of Sarkar, Sly, and Zhang would require such an estimate for exponential LPP and so this result is not of direct applicability in their argument, but we hope the context showcases the interest of such results. We set up the notation to state this result next.

In order to have a reasonable notion of the number of near maxima, we define for $d\geq 1$ the \emph{number of near touches} random variable $\numnt(X)$ for a random process $X:[-d,d]\to\R$ as follows, where $M = \sup_{t\in[-d,d]}X(t)$:
$$\numnt(X, \eta) = \max \left\{|S| \ :\  \begin{tabular}{@{}c@{}}
$S\subseteq [-d,d]$, s.t. $s\in S \implies X(s)\geq M-\eta^{1/2}$\\ 
and $s,t\in S, s\neq t\implies |s-t|\geq \eta$.
\end{tabular}\right\}.$$
In words, this quantity is the size of the maximum collection of $\eta$-separated times at which $X$ comes within~$\eta^{1/2}$ of its global maximum on $[-d,d]$.

The following result says that the number of near maxima has exponential tails.

\begin{theorem}[Many near touch]\label{t.many near touch}
Let $d\geq 1$ and $0<\eta<d$. Then there exists $\ell_0$ and $c>0$ such that, for $\ell>\ell_0$,
$$\P\Big(\numnt(\L, \eta) \geq \ell\Big) \leq e^{-c\ell}.$$
\end{theorem}

We again note that the stated result studies near maxima with respect to the maximiser on $[-d,d]$. However, unlike Theorem~\ref{t.local johansson}, the result proved implies the same bound for the number of near touches in $[-d,d]$ of the \emph{global} maximiser, i.e., the random variable with the same definition as $\numnt(\L, \eta)$ but with $M$ being the global maximum value. This is simply because if $x\in[-d,d]$ is a point of near touch with the global maximum, it must also be a point of near touch with the maximum on $[-d,d]$. Bounding the number of near touches of the global maximiser which occur anywhere, not necessarily in $[-d,d]$, is again likely to be tractable using \cite[Corollary 4.6]{corwin2014brownian} and the parabolic curvature of $\L$.

As with Theorem~\ref{t.local johansson}, the proof of Theorem~\ref{t.many near touch} is given in Chapter~\ref{ch.application proofs} and relies on bounding the probability for the same event under Brownian motion using information about Brownian meander.


The next application we discuss involves Brownian LPP and is not proved in this paper; as such, we aim at giving only a heuristic description underscoring the points of interest.

\subsection{Brownianity of the rooted weight profile}
\label{s.intro.rooted weight profile}


We explain a pertinent random function in the LPP setting,  harnessing the notation that we have introduced for the Bernoulli model. Recall the discussion around the supremum formula \eqref{e.supremum formula} for $\weight_n^{0|1}$. Let $w \in (0,1)$ be a height in the scaled copy of the plane through which the polymer $\rho_n = \rho_n \big[ (0,0) \to (0,1) \big]$ passes. Define the {\em rooted weight profile} $Z_w:\R \to \R$ indexed by height $w$ to be the function that reports the highest weight of a scaled path that is forced to pass through a given location at height $w$. That is, regarding $\weight_n$ as the analogue of $\weight^{0|1}_n$ in the setting of Brownian LPP, we set
\begin{equation}\label{e.za}
 Z_w(x) = \weight_n \big[ (0,0) \to (x, w) \big] + \weight_n \big[ (x, w) \to (0,1) \big] \, ,
\end{equation}
the right-hand side being the maximum scaled energy of a scaled path that begins at $(0,0)$; ends at $(0,1)$; and makes a visit to the location $(x,w)$. Thus, $Z_w(x)$ is analogous to (what may be called) the supremand in~\eqref{e.supremum formula}. 
Note that by \eqref{e.supremum formula}, any point in $\rho_n$ of the form $(x,w)$ is such that $Z_w(x)$ achieves the maximum value of $Z_w$.  
A natural question in the study of near polymers, which is related to the question of the existence of multiple valleys, is: ``how probable is it that a near maximiser of $Z_w$
is achieved at a significant distance from the maximiser?'' 
The two right-hand terms in~(\ref{e.za}) locally resemble standard Brownian motion in the sense of Theorem~\ref{T.MAIN THEOREM FOR AIRY} (though as we are discussing a weight profile in Brownian LPP, this is a consequence of the upcoming main result, Theorem~\ref{t.airytail.ln});
 so $Z_w$ plausibly resembles Brownian motion of rate two.  
 
 For example, this intuition suggests that the probability of occurrence of a near maximiser of $Z_w$ should obey a similar bound as in Theorem~\ref{t.local johansson}. That is, if $x$ is a maximiser of $Z_w$, the probability that $Z_w(y)$  attains a value that exceeds $Z_w(x) - \alpha h^{1/2}$ at a location $y \in \R$ for which $\vert y - x \vert \geq h$ should be at most $\alpha$, up to a multiplicative error that rises subpolynomially in the limit $\alpha \searrow 0$. This is in fact true, and will follow from the next result in the same way that Theorem~\ref{t.local johansson} follows from Theorem~\ref{T.MAIN THEOREM FOR AIRY}.


\begin{theorem}\label{t.rootedweightprofile}
Let $w$ lie in a compact interval in $(0,1)$, $d\geq 1$, and $\mc C$ be as in Theorem~\ref{T.MAIN THEOREM FOR AIRY}.
The rooted weight profile $Z_w: \R \to \R$ is strongly comparable to rate-two Brownian motion $B$ on $[-d,d]$. That is, there exists $G<\infty$ and $g>0$ such that if $A$ is a Borel measurable subset of $\mc{C}$
for which $\PP \big( B \in A \big) = \epsilon \in (0,1)$, and if $\epsilon$ is both sufficiently small and greater than $\exp(-gn^{1/12})$, then $\P \big( Z_w - Z_w(-d) \in A \big) \leq \epsilon \exp \big\{ G d ( \log \epsilon^{-1} )^{5/6} \big\}$.
\end{theorem}

This theorem will appear in \cite{LPPtools}, where it will be derived from our main result, Theorem~\ref{t.airytail.ln}.

The stretched-exponential-in-$n$ lower bound condition on $\epsilon$ expressed in Theorem~\ref{t.rootedweightprofile} is an artifact of Theorem~\ref{t.airytail.ln}, which is the prelimiting counterpart of Theorem~\ref{T.MAIN THEOREM FOR AIRY} that is valid for Brownian LPP. In most applications the condition is irrelevant, as events of interest do not usually have probabilities which decay faster than polynomially in $n$, (and sometimes do not decay with $n$ at all). The condition does not appear in Theorem~\ref{T.MAIN THEOREM FOR AIRY} because there $n$ takes the value $\infty$, in which case the lower bound condition is vacuously true.

We have placed Theorem~\ref{t.rootedweightprofile} in the context of existence of near maxima or multiple valleys. It is also plausible that the strong control in the unperturbed environment provided by this theorem would prove valuable for studying the behaviour of the system under small perturbations, consistent with the ideas expounded in \cite{chatterjee2014superconcentration} and briefly discussed in Section~\ref{s.intro.energy landscape}.

Recall that we have till now focused on the narrow wedge initial condition, which leads to the parabolic Airy$_2$ process in the limit. The last two applications will be discussed in Section~\ref{s.intro general init condition Brownian regularity}, after we have described a form of Brownian regularity for scaled energy profiles begun from \emph{general} initial conditions. But we first turn to reviewing related work in whose purview this article falls.

\section{Pertinent recent work}\label{s.intro.previous work}
We start by giving a brief account of the general background of the Airy$_2$ process. The interested reader is referred to the survey \cite{quastel2014airy} for a more detailed review, though from the slightly different viewpoint of integrable probability.

The one-point distribution of the stationary Airy$_2$ process is the GUE Tracy-Widom distribution, first discovered in random matrix theory as the distribution of the limiting scaled fluctuations of the largest eigenvalue of the Gaussian Unitary Ensemble \cite{tracy1994level}. A breakthrough in the field of KPZ was Baik, Deift, and Johansson \cite{baik1999distribution} proving that the same GUE Tracy-Widom distribution arises as the distribution of the limiting scaled fluctuations of the point-to-point energy in Poissonian LPP, through an equivalent description in terms of the longest increasing subsequence of a uniform random permutation.

The jump from the one-point GUE Tracy-Widom distribution to the full Airy$_2$ process was made in \cite{prahofer2002scale}, where it was shown that the weight profile in Poissonian LPP (which has a bijection with the PNG model) converges weakly to the Airy$_2$ process minus a certain parabola, in the sense of finite dimensional distributions. This convergence was strengthened to hold on the space of continuous functions in a closely related model in \cite{johansson2003discrete}.

The locally Brownian nature of the Airy$_2$ process has been previously established in a number of different formulations. One relatively weak version is to consider \emph{local limits} of the Airy$_2$ process; i.e., to study the Gaussianity of $\epsilon^{-1/2}\left(\A(x+\epsilon)-\A(x)\right)$ for a given $x\in \R$ as $\epsilon \searrow 0$. The appearance of Brownian motion in this limit was proven in~\cite{hagg2008,CATOR2015538,quastel2013local}. The final of these three articles, \cite{quastel2013local}, also establishes H\"older $\frac{1}{2}-$ continuity of the Airy$_2$ (as well as Airy$_1$) process, which is extended to limiting weight profiles arising from a very general class of initial conditions in \cite[Theorem 4.13]{matetski2016kpz}. A stronger notion of the locally Brownian nature of the Airy$_2$ process is absolute continuity of $\A$ with respect to Brownian motion on a unit order compact interval. This was first proved in \cite{corwin2014brownian}, and was used in the same paper to prove Johansson's conjecture mentioned above.

Another line of work has established various Brownian features in the \emph{pre-limiting} weight profiles. For instance, \cite{basu2018time} establishes local Brownian fluctuations (in the sense of sub-Gaussian tails) in the weight profile of point-to-point exponential LPP, while \cite{hammond2017modulus} establishes a (sharp) version of the Holder $\frac{1}{2}-$ continuity mentioned above for the pre-limiting weight profiles in Brownian LPP (which also applies with quite general initial conditions).

However, none of these results addresses the question of bounding probabilities involving the Airy$_2$ process in terms of Brownian probabilities, or, equivalently, providing growth bounds on the Radon-Nikodym derivative with respect to some Brownian process.

A result in this direction was proved in \cite{hammond2017brownian}. There the comparison was between a modification of $\L$, denoted $\L^{[-d,d]}$, that is defined by affinely shifting $\L$ to be zero at both endpoints of $[-d,d]$, and Brownian bridge, instead of between a vertically shifted version of $\L$ and Brownian motion as in Theorem~\ref{T.MAIN THEOREM FOR AIRY}. The form of the result, however, is otherwise much the same:

\begin{theorem}[Theorem 1.10 of \cite{hammond2017brownian}]\label{t.old bb comparison}
Let $d\geq 1$ and let $\mc C_{0,0}$ be the space of continuous functions which vanish at both endpoints of $[-d,d]$. Let $A$ be a Borel measurable subset of $\mc C_{0,0}$, and let $\epsilon = \B^{[-d,d]}(A)$, where $\B^{[-d,d]}$ is the law of standard Brownian bridge on $[-d,d]$ (i.e., with vanishing endpoints). There exists $\epsilon_0 = \epsilon_0(d)>0$ and an absolute finite constant $G$ such that, if $\epsilon \in (0, \epsilon_0)$, then
$$\P\left(\L^{[-d,d]}(\,\cdot\,) \in A\right) \leq \epsilon \cdot \exp\left(Gd^2(\log \epsilon^{-1})^{5/6}\right).$$
\end{theorem}

This also follows immediately from Theorem~\ref{T.MAIN THEOREM FOR AIRY} and the fact that performing the affine shift described on Brownian motion results in Brownian bridge.

Theorem~\ref{t.old bb comparison} and our new Theorem~\ref{T.MAIN THEOREM FOR AIRY} are formally very similar, the latter obtained merely by substituting Brownian motion for Brownian bridge. However, it is found in  many contexts that Theorem~\ref{t.old bb comparison} is unable to provide the kind of information that is desired. This is because, though the process $\L(\,\cdot\,) - \L(-d)$ can be obtained from the bridge $\L^{[-d,d]}(\,\cdot\,)$ and the endpoint $\L(d)$, the desired information gets away from us due to potentially pathological correlations between these two random objects. Controlling this correlation is especially required to understand the slope or maximum of $\L$ on an interval; the slope or maximum are often of relevance in LPP problems, as can be seen in the applications discussed in Section~\ref{s.intro applications}.

The proof of Theorem~\ref{T.MAIN THEOREM FOR AIRY} is significantly more involved and subtle than the proof of Theorem~\ref{t.old bb comparison} in \cite{hammond2017brownian} because of the need to handle these correlations.
We make some more comments contrasting the proofs in Section~\ref{s.intro method of proof}.

Theorem~\ref{t.old bb comparison} was a crucial tool in the four-part study of Brownian LPP undertaken in \cite{hammond2017brownian,hammond2017modulus,hammond2017rarity,hammond2017patchwork}. In the final paper \cite{hammond2017patchwork}, a form of Brownian regularity was proved for pre-limiting weight profiles for general initial conditions, to which we return shortly. But we first turn to discussing the Brownian Gibbs property, a crucial idea in the proofs of Theorem~\ref{t.old bb comparison} as well as our own main result.

\subsection{The Brownian Gibbs property}
A central player in our approach is the \emph{Brownian Gibbs property}, and here we discuss previous work in this line of study. The Brownian Gibbs property was first employed in \cite{corwin2014brownian}, to study the \emph{Airy line ensemble}. The Airy line ensemble is an $\N$-indexed collection of continuous, non-intersecting curves, whose uppermost curve is the Airy$_2$ process. The Brownian Gibbs property is an explicit spatial Markov property enjoyed by the Airy line ensemble after a parabolic shift and multiplication by a factor $2^{-1/2}$, resulting in the \emph{parabolic Airy line ensemble}. In short, the Brownian Gibbs property says that the conditional distribution of any set of $k$ consecutive curves on an interval $[a,b]$, conditionally on all the other curves on all of $\R$ and the $k$ curves themselves on $(a,b)^c$, is given by $k$ independent rate one Brownian bridges between appropriate endpoints and conditioned to intersect neither each other nor the preceding and succeeding curves.

The Brownian Gibbs property and various softenings of it have proved to be a versatile tool in probabilistic investigations of KPZ. Beyond the already mentioned \cite{corwin2014brownian}, there have been numerous works on line ensembles enjoying this or an analogous property, which we briefly discuss.

The Brownian Gibbs property itself was a central theme in the previously mentioned four-part study \cite{hammond2017brownian,hammond2017modulus,hammond2017rarity,hammond2017patchwork} of Brownian LPP. While \cite{corwin2014brownian} established that the Brownian Gibbs property is enjoyed by the Airy line ensemble, and hence by the limiting weight profiles in a number of LPP models, Brownian LPP is special in that its weight profile satisfies the Brownian Gibbs property even in the \emph{pre-limit}. This is a crucial integrable input first observed by \cite{o2002representation} (who related the energy profiles in Brownian LPP to Dyson Brownian motion), and is the reason why Brownian LPP is the setting of the mentioned four-part study, as well as why our main results will apply to it. Apart from this four-part study, we mention some other works in this vein. The work \cite{corwin2014ergodicity} establishes the ergodicity of the Airy line ensemble using the Brownian Gibbs property. The fractal nature of a certain limiting weight difference profile in Brownian LPP is investigated in \cite{basu2019fractal}, using inputs from the four-part study mentioned earlier. The Brownian Gibbs property is used in \cite{caputo2019confinement,caputo2019tightness} to analyse tightness of families of non-intersecting Brownian bridges above a hard wall, subject to a tilting of measure in terms of the area the curves capture below them; they also establish that an area-tilted form of the Brownian Gibbs property is enjoyed by the limiting ensemble.

A softened version of Brownian Gibbs, in which intersection is not prohibited but suffers an energetic penalty, was used in an investigation of the {\em scaled} solution to the KPZ equation with narrow-wedge initial condition \cite{corwin2016kpz}, establishing for that process absolute continuity with respect to Brownian motion on compact intervals. This form of Brownian Gibbs was also used in the recent \cite{corwin2018kpz} to obtain bounds on the one-point upper and lower tails for the solution to the KPZ equation from quite general initial data, and in \cite{corwin2019kpz} to establish the rate of decay of correlations with time of the narrow wedge solution at the origin. A discrete Gibbsian property was used in \cite{corwin2018transversal} to study the transversal fluctuation exponent and tightness of the appropriately scaled height function in the asymmetric simple exclusion process and stochastic six vertex model, started with step initial conditions. A sequence of discrete line ensembles associated to the inverse gamma directed polymer, which obeys a softened discrete version of the Brownian Gibbs property, was shown to be tight in \cite{wu2019tightness}.

Finally, we mention the valuable contribution \cite{dauvergne2018directed}, aided by~\cite{dauvergne2018basic}, which establishes the existence of the space-time Airy sheet using Brownian LPP and the Brownian Gibbs property. We shall say more on this in the next subsection on limiting weight profiles from general initial conditions.

\subsection{The KPZ fixed point and the directed landscape}\label{s.intro kpz fixed point}
As mentioned, the Airy$_2$ process arises as a limiting process under very particular initial conditions, often called narrow-wedge (which corresponds to the step initial condition for TASEP). Given our knowledge of the strong comparison to Brownian motion that the Airy$_2$ process enjoys via Theorem~\ref{T.MAIN THEOREM FOR AIRY}, a natural question is whether such a comparison extends to the limiting processes arising from general initial conditions.

Of course, before wondering about the Brownian regularity of a limiting profile process under general initial conditions, we must show that such a limiting process exists. Two recent results are pertinent to this basic question.

The first is the construction in \cite{matetski2016kpz} of a scale invariant Markov process, the {\em KPZ fixed point}  \footnote{The usage of this term in our title mildly corrupts that in~\cite{matetski2016kpz}. In speaking of Brownian structure in the KPZ fixed point, we refer to Brownian structure in the random function obtained by evolving the Markov process in question from given initial data, whether narrow wedge or more general. Properly, and as used in \cite{matetski2016kpz}, the KPZ fixed point is the Markov process, rather than the resulting profile.},  in the context of the TASEP model.  In this work, the authors establish formulas for the finite dimensional distributions of the height function under very general initial data, which, on taking the scaling limit according to KPZ exponents, results in a scale invariant Markov process of central importance in the KPZ class. 
Simply put, the time-one evolution of the Markov process of \cite{matetski2016kpz} on given initial data coincides with the limiting weight profile begun from the same initial condition, 
modulo the fact that \cite{matetski2016kpz} works in TASEP and not Brownian LPP. For example, the time-one evolution of this limiting Markov process begun at the narrow-wedge initial condition (i.e., step initial condition for TASEP) results in the parabolic Airy$_2$ process. The statement that the time one evolution of the Markov process on general initial data coincides with the limiting weight profile arising from the same initial data is not yet known to hold exactly because the analysis in \cite{matetski2016kpz} is done in the pre-limiting model of TASEP, which lacks exact Brownian structure. For the same reason, our results do not directly apply to these limiting processes. If the scale invariant Markov process of \cite{matetski2016kpz} is constructed using Brownian LPP, our results can say something further.

However, a Brownian structure is brought to the pre-limiting model in the recent advance, made by Dauvergne, Ortmann, and Vir\'ag in \cite{dauvergne2018directed} and assisted by \cite{dauvergne2018basic}, which proves the existence of the \emph{space-time Airy sheet}. The space-time Airy sheet is a previously conjectured \cite{corwin2015renormalization} universal object in the KPZ universality class and is a process with two temporal and two spatial arguments; these arguments should be thought of as the coordinates of a pair of planar points, with each point's coordinates consisting of one spatial and one temporal argument. For fixed values of the temporal arguments, which corresponds to restricting the mentioned points to lie on two fixed lines, the marginal of the space-time Airy sheet in either of its two spatial arguments is the Airy$_2$ process. The marginal spatial process when the temporal arguments take fixed values is sometimes referred to as simply the Airy sheet.

The pre-limiting model used in the construction of the Airy sheet in \cite{dauvergne2018directed} is Brownian last passage percolation, and an analysis of the bulk behaviour of curves deep in the parabolic Airy line ensemble is undertaken in \cite{dauvergne2018basic} in a manner that assists the construction of the scaling limit of Brownian LPP, i.e., the process limit of $\weight_n[(x,s)\to(y,t)]$ as a function of all four arguments $x$, $s$, $y$, and $t$; this scaling limit is called the \emph{directed landscape} in~\cite{dauvergne2018directed} and is a parabolically shifted space-time Airy sheet. (The directed landscape can be thought of as the scaled KPZ energy landscape, similar to the energy landscape discussed in Section~\ref{s.intro.energy landscape}, under the interpretation that the directed landscape assigns scaled energies or weights to scaled limiting paths, and modulo the change in sign of the energy compared to Section~\ref{s.intro.energy landscape}.) The proof of the construction of this scaling limit in \cite{dauvergne2018directed} proceeds via equating last passage percolation values in the original environment to an LPP problem in a new environment defined by last passage values in the original environment, a novel extension of the Robinson-Schensted-Knuth correspondence. Using this result, it follows that limiting weight profiles from general initial conditions exist, and we may further say something about its Brownian regularity using our results. We move in this direction next.

\section{A form of Brownian regularity for weight profiles with general initial conditions}\label{s.intro general init condition Brownian regularity}


 In this section we describe the two remaining applications of Theorem~\ref{T.MAIN THEOREM FOR AIRY}. As was just mentioned, a fairly direct consequence of the results of \cite{dauvergne2018directed} is the existence of the limiting weight profile from general initial conditions in Brownian LPP (Proposition~\ref{p.limit of general weight profile}). A natural next question is whether these general limiting weight profiles enjoy a similar comparison to Brownian motion as provided by Theorem~\ref{T.MAIN THEOREM FOR AIRY} for the limiting narrow wedge weight profile, i.e., the parabolic Airy$_2$ process. In fact, convergence of the \emph{local} limit to Brownian motion was already known in different senses for various classes of initial conditions; for example, this Brownian local limit was shown for the Airy$_1$ process in \cite{quastel2013local} in the sense of finite dimensional distributions; for a class of Airy processes arising from quite general initial conditions as constructed in \cite[Theorem 3.13]{matetski2016kpz}, again in the sense of finite dimensional distributions, in \cite[Theorem 4.14]{matetski2016kpz}; while the same local limit in the space of continuous functions has been proven in \cite{pimentel2018local} under an assumption which is verified for certain specific initial conditions (such as the mixed profile flat$\to$stationary).

However, fully general initial conditions do not enjoy the many algebraic properties that are present in the narrow wedge and the other previously studied cases mentioned, and so results are much weaker or absent for those weight profiles. In particular, though it is believed that even with general initial condition the limiting weight profile should be locally Brownian in a strongly quantifiable sense, it seems difficult to establish something along the lines of Theorem~\ref{T.MAIN THEOREM FOR AIRY}.

Nevertheless, it is possible to use the general form of Theorem~\ref{T.MAIN THEOREM FOR AIRY} to prove a form of Brownian regularity similar to one which was introduced in \cite{hammond2017patchwork}. But before addressing the Brownian regularity of the limiting profile, let us look at the pre-limiting weight profile, as this has a direct interpretation as the weight associated to polymers.

\subsection{Discovering a polymer forest and a patchwork quilt}

We return to the model of Bernoulli LPP to illustrate what is meant by general initial conditions. The initial condition is described by a function $f:\R\to\R\cup\{-\infty\}$, and for such a fixed $f$, we consider scaled (via the scaling function $S_n$) paths which may begin anywhere on the real line at time zero, but must end at $(y,1)$. The \emph{$f$-rewarded weight} of the path is the sum of $f$ evaluated at the starting point---a reward---and the weight collected by the path on its journey. The maximum $f$-rewarded weight over all such paths is denoted $\weight^{f,\,0\mid 1}_{n}[(*,0)\to(y,1)]$. More precisely,
$$\weight^{f,\,0|1}_n[(*,0)\to (y,1)] := \sup\left\{\weight^{0|1}_n[(x,0)\to (y,1)] + f(x) : x\in\R\right\};$$
the $*$ in the notation is to indicate that the starting point is free. A mild condition that $f$ grows at most linearly is needed to make this object suitable for study.

The narrow-wedge initial condition we have been focusing on thus far corresponds to the case that $f(x)$ is zero when $x$ is zero, and $-\infty$ otherwise. In that case, we saw that the weight profile has a globally parabolic shape, and that, in the $n\to\infty$ limit, it converges to the parabolic Airy$_2$ process. What will the weight and limiting weight profiles look like when $f$ is more general?

For example, suppose that we set $f(x)$ to be zero for $x=0$ and $x=1$ and $-\infty$ everywhere else; in other words, we allow growth from the two starting locations $x=0$ and $x=1$. For various ending points $(y,1)$, we may trace back the polymer with that ending point and observe at which starting point it originated. In fact, the polymers originating from either of the points 0 and 1 will form a tree structure; the energy-maximising objective means that the trees can be viewed as \emph{competing} with one another. Within the canopy of a single tree, where its weight is dominant, we can expect the weight profile to look like that of growth from a single point (see Figure~\ref{f.polymer forest}).

\begin{figure}[h]
\centering{\epsfig{file=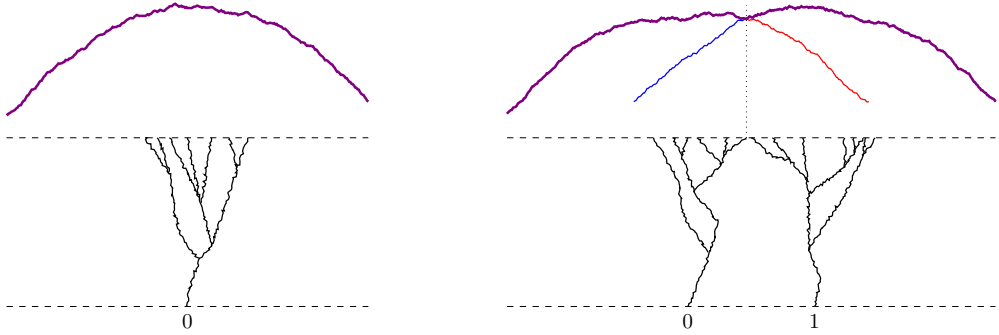, width=0.8\textwidth}}
\caption{An illustration of polymer forest and the patchwork quilt. Growth is from a single point in the left figure, and from two points in the right one. In both figures, the actual weight profile is represented as a thicker curve in purple, and is the upper envelope. In the right figure, the two polymer trees compete, and the canopies of each are adjacent to each other; the boundaries of the canopies correspond to where the weight profile is divided into two patches. The fabric piece on each patch is the weight profile of growth from a single point, and are represented by the red and blue parabolic profiles, which become thicker and purple in the patch where they agree with the actual weight profile.}\label{f.polymer forest}
\end{figure}

Thus we may surmise that in this situation of growth from two points, the full weight profile can be seen as a piecewise function, where each piece (called a \emph{patch}) has the distribution of growth from a single point; each patch can be expected to enjoy a certain Brownian regularity similar to that in Theorem~\ref{T.MAIN THEOREM FOR AIRY}. Under fully general initial conditions, where growth may be from any point on the lower line with a certain reward associated to each point that is added to the energy of the path, essentially the same picture holds: the weight profile can be broken up into a number of patches, corresponding to the canopies of the surviving polymer trees. The only difference is that the number of patches will be random. (Actually, this description is slightly simplified: for technical reasons, the final patches will be sub-patches of the patches that we have described. The reader is referred to \cite{hammond2017patchwork} for a fuller discussion.)

This leads us to the following notion of regularity for general weight profiles. The (pre-limiting) weight profile on a unit interval is divided into a random number of subintervals, the patches, with random boundary points in such a way that the restriction of the profile (called a \emph{fabric piece}) to each patch enjoys a comparison to Brownian motion similar to the one described in Theorem~\ref{T.MAIN THEOREM FOR AIRY}, though perhaps in the weaker $L^{p-}$ form rather than $L^{\infty-}$. In this way, we may say that the weight profile is a \emph{patchwork quilt of Brownian fabrics}. The strength of the regularity of the patchwork quilt depends largely on the control available for the number of patches, but also on the Brownian motion regularity guaranteed for each fabric piece as specified by the value of $p$. A precise definition is provided in Definition~\ref{d.patchwork quilt}. Proving an analogue to Theorem~\ref{T.MAIN THEOREM FOR AIRY} would correspond to establishing this regularity with a \emph{single} patch and with $p=\infty$.

Using our main theorem, we are able to show that this notion of Brownian motion regularity holds for the pre-limiting weight profiles from a general class of initial conditions, with a certain decay on the number of patches. We will state this precisely in Chapter~\ref{ch.patchwork quilt} as Theorem~\ref{t.patchwork quilt} after the relevant definitions are made. This result is a refinement of a result in \cite{hammond2017patchwork}, where the comparison of each fabric piece was made to Brownian bridge instead of Brownian motion, using a generalised form of Theorem~\ref{t.old bb comparison} that was proved in that article. 

It is natural to expect that this form of Brownian regularity for the pre-limiting weight profiles should pass to the limiting profile if it exists, as was remarked in \cite{hammond2017patchwork}. With the existence of the limiting profile established (in Proposition~\ref{p.limit of general weight profile}) using the result of \cite{dauvergne2018directed}, we also show that this notion of regularity passes to the limit, which we state informally now and formally as Theorem~\ref{t.limiting weight profile is quiltable}. This theorem may be seen as the most general form of Brownian regularity in the KPZ fixed point proved in this paper.

Again, we emphasise that this result is proved for the limiting weight profile obtained via Brownian LPP and not Bernoulli LPP. Though the formal definitions have been deferred, we denote the limiting Brownian LPP weight profile with reward function $f$ as the function $y\mapsto \weight^f_{\infty}[(*,0)\to (y,1)]$, i.e., without the $0\mid1$ in the superscript and with $n$ taking the formal value of $\infty$.

\begin{theorem}[Informal version of Theorem~\ref{t.limiting weight profile is quiltable}]\label{t.informal limiting quiltable}
Let $f:\R\to\R\cup\{-\infty\}$ be such that there exists $\Psi<\infty$ such that $f\not\equiv -\infty$ and $f(x)\leq \Psi(1+|x|)$. Then we have that $y\mapsto\weight_{\infty}^f[(*,0)\to(y,1)]$ is Brownian motion patchwork quiltable; the comparison with Brownian motion may be made in $L^{3-}$, and the random number of patches has a polynomial tail with exponent $2-\epsilon$ for any $\epsilon>0$.
\end{theorem}

With the fact that the limiting weight profile enjoys the Brownian motion patchwork quilt description, we may state our final application as our final theorem, which gives a uniform Brownian-motion-like bound on the $2-\eta$ moment of an increment of the limiting weight profile, for any $\eta>0$ and an extremely broad class of initial conditions. 

\begin{theorem}\label{t.weight profile increment moment bound}
Let $f:\R\to\R\cup\{-\infty\}$ be such that there exists $\Psi<\infty$ such that $f\not\equiv -\infty$ and $f(x)\leq \Psi(1+|x|)$, and let $0<\eta<\frac{1}{2}$. Then there exist constants $G<\infty$ and $y_0>0$ such that, for $|y|<y_0$,
$$\E\left[\left|\weight^f_{\infty}[(*,0)\to (y,1)] - \weight^f_{\infty}[(*,0)\to (0,1)]\right|^{2-\eta}\right] \leq G |y|^{1-\eta/2}.$$
\end{theorem}

In fact, $G$ can be taken as a uniform constant over all $f$ in a certain class of initial conditions that we will define in Section~\ref{s.notation and regular ensembles}. We will state and prove that result as Corollary~\ref{c.precise increment moment bound}, which immediately implies Theorem~\ref{t.weight profile increment moment bound}.

Theorem~\ref{t.weight profile increment moment bound} applies, for example, to flat initial conditions, where $f\equiv 0$. The limiting weight profile in this case is expected to be the Airy$_1$ process, and this is indeed known in the Poissonian and exponential LPP models; in Brownian LPP, however, we were unable to locate such a result in the literature. For Airy$_1$, it is widely expected that the second moment of the increment grows linearly in the size of the increment when the size is small \cite{ferrari2008survey}, but a formal result does not appear to be known.

A slight improvement to Theorem~\ref{t.weight profile increment moment bound} will lead to the bound for Airy$_1$ increments posited in \cite{ferrari2008survey}. The improvement required concerns control over the decay of the number of patches in the patchwork quilt. Currently, the probability there are at least $\ell$ patches intersecting a unit interval is only known to decay as $\ell^{-(2-\epsilon)}$ for any $\epsilon>0$, as stated in Theorem~\ref{t.informal limiting quiltable}; the improvement to Theorem~\ref{t.weight profile increment moment bound} would require the bound to be at most $\ell^{-(2+\epsilon)}$ for some $\epsilon>0$. Obtaining this additional control may be technically challenging and would require a furthering of the methods of \cite{hammond2017patchwork}, which we do not pursue here.

Lemma 3.3 of \cite{quastel2013local} states a bound related to Theorem~\ref{t.weight profile increment moment bound} in the case of flat initial conditions. More precisely, it gives a bound on even moments of a truncated version of the increment. However, the truncation is a significant enough modification that the behaviour of the bound is no longer Brownian, as it is in Theorem~\ref{t.weight profile increment moment bound}, and so it does not seem possible to recover the expected Brownian nature of the second moment that was raised in \cite{ferrari2008survey}.


\section{Method of proof}\label{s.intro method of proof}
In this final section of the introduction we compare, on a high level, the method of proof of the main theorem with that of Theorem~\ref{t.old bb comparison} as given in \cite{hammond2017brownian}. Chapter~\ref{ch.proof framework} is devoted to describing the general framework for the proof of our main result, with Section~\ref{s.heuristics for proof} addressing the conceptual framework specific to the main result Theorem~\ref{t.airytail.ln}.

At the highest level, the method of proof of Theorem~\ref{t.old bb comparison} in \cite{hammond2017brownian} relies on embedding the parabolic Airy$_2$ curve as the uppermost curve in the parabolic Airy line ensemble and employing the Brownian Gibbs property. In \cite{hammond2017brownian}, a significant amount of additional technical apparatus, known as the jump ensemble, was developed to further this technique, which culminated in the proof of Theorem~\ref{t.old bb comparison}.

The proof of our main theorem is based squarely on the Brownian Gibbs property via the jump ensemble as well, but the details of the proof differ quite substantially from that of Theorem~\ref{t.old bb comparison} because of the difficulties that arise from possibly pathological correlations between the bridge $\L^{[-d,d]}$ and the endpoints $\L(-d)$ and $\L(d)$.

\begin{figure}[h]
\centering{\epsfig{file=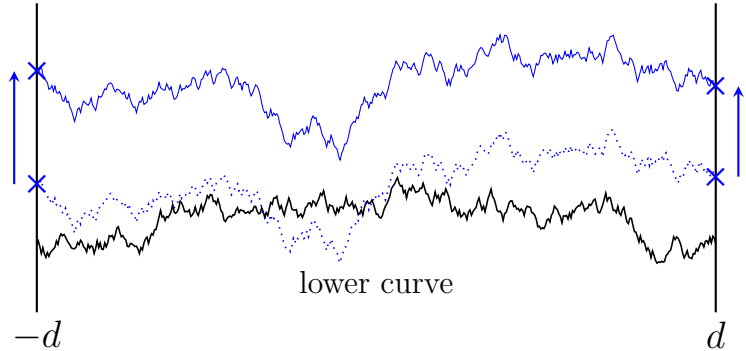, width=0.6\textwidth}}
\caption{The Brownian Gibbs property and an independence property of Brownian bridges essentially reduces the proof of Theorem~\ref{t.old bb comparison} given in \cite{hammond2017brownian} to understanding the probability of non-intersection with the lower curve (i.e., the second curve of the parabolic Airy line ensemble) conditionally on the endpoint values at $-d$ and $d$. This probability has an important monotonicity property which substantially simplifies the proof of Theorem~\ref{t.old bb comparison}: if we raise the endpoint values, the top curve is more likely to be fully above the lower curve. This can be seen by the stochastic domination depicted here, as the top curve with lower endpoint values (blue and dotted) intersects the lower curve (thick and in black), while on raising the endpoint values, non-intersection is achieved. This monotonicity will \emph{not} be available in the proof of the main result of this paper. This is because we do not have access to the independence property of Brownian bridges, which is what allowed the decoupling of the probability of the event under consideration from the probability of non-intersection.}\label{f.monotonicity simplification}
\end{figure}

A flavour of this difficulty can be seen even in a purely Brownian toy example quite easily, and this example will be fairly representative because of the Brownian Gibbs property. Suppose we are trying to bound the probability that a Brownian process lies in a particular measurable subset of continuous functions. We are contrasting the situation when the Brownian process is Brownian motion with when it is Brownian bridge; we note that applying the affine shifting procedure described before Theorem~\ref{t.old bb comparison}, which defines $\L^{[-d,d]}$ from $\L$, to Brownian motion results in Brownian bridge. Let $B$ be a standard rate one Brownian motion on $[-d,d]$ started at zero, and let $B^{[-d,d]}$ be the Brownian bridge on $[-d,d]$ resulting from the affine shifting procedure.

A standard fact is that $B^{[-d,d]}$ is \emph{independent of the original endpoint value} $B(d)$ of the Brownian motion. Thus, when evaluating the probability that $B^{[-d,d]}$ lies in some subset of continuous functions, one simply has to integrate over $B(d)$; the conditional probability given $B(d)$ is the same for all of them. 

In contrast, consider the probability that $B$ lies in a subset $A$ of continuous functions. If we here try to decompose the process by conditioning on its endpoint value $B(d)$, the conditional probability of $A$ depends on $B(d)$. More importantly, the nature of the dependence is not the same for all $A$, and so there is no clear way to decouple the conditional probability of $A$ from the endpoint values in an event-agnostic way.

The Brownian Gibbs property in some sense relates the statement to be proved, here regarding the process $\L$ in the form of Theorem~\ref{T.MAIN THEOREM FOR AIRY}, to considerations similar to this toy example. Recall that, in a 
loose sense, the Brownian Gibbs property says that the conditional distribution of $\L$ on an interval is that of a Brownian bridge with appropriate endpoints conditioned on being above a lower curve over the whole interval; the lower curve is the second curve of the parabolic Airy line ensemble. We are considering the probability that $\L^{[-d,d]}$ belongs to an event $A$. On a heuristic level, applying the Brownian Gibbs property and the independence from endpoints enjoyed by Brownian bridge, bounding the conditional probability of $A$ given non-intersection and the endpoint values $\L(-d)$ and $\L(d)$ reduces to bounding the probability of non-intersection given $\L(-d)$ and $\L(d)$; the probability of $A$ under Brownian bridge factors out. 

A simplifying feature of the conditional probability of non-intersection given $\L(-d)$ and $\L(d)$ is that it enjoys an intuitive monotonicity in the endpoint values: when they are higher, avoiding the lower curve is more probable (see Figure~\ref{f.monotonicity simplification}). Using this monotonicity, it is sufficient for the proof of Theorem~\ref{t.old bb comparison} to bound the non-intersection probability by obtaining a bound on the endpoint value density in only the case when the endpoints are very low. This is a crucial technical result in \cite{hammond2017brownian}, stated as Lemma~5.17. (This description is not completely accurate as in the proof of Theorem~\ref{t.old bb comparison} the technical apparatus of the jump ensemble allows the non-intersection condition to be not with the entire lower curve but only a certain subset of it. We ignore this point here.)

However, for the process $\L(\,\cdot\,) - \L(-d)$, analogous to the Brownian motion discussion,  the probability of an event $A$ and the probability of non-intersection cannot be decoupled given the endpoint values, and the probability of the combined event does not enjoy a monotonicity property in the endpoint values. (Of course, for certain events this monotonicity property would be true, but it does not hold in an event-agnostic manner.) Thus, while in the proof of Theorem~\ref{t.old bb comparison} it was sufficient to have an endpoint value density bound in only the case where the endpoint values are very low, for the main result of this paper we will need corresponding density bounds for the remaining ranges of endpoint values as well. The case of low endpoint values is handled by using the same statement of \cite{hammond2017brownian}, Lemma 5.17 there, here stated as Proposition~\ref{p.negative density}, but the other ranges of endpoint values give rise to additional cases of greater technical difficulty. 

A fuller discussion of the ideas and approach of the proof is provided in Chapter~\ref{ch.proof framework}, with Section~\ref{s.jump ensemble} discussing the jump ensemble and Section~\ref{s.heuristics for proof} discussing the framework specific to the proof of the main result.


\section{Organization of the article}
In Chapter~\ref{ch.application proofs} we prove some of the consequences of Theorem~\ref{T.MAIN THEOREM FOR AIRY} discussed in Section~\ref{s.intro applications}. In Chapter~\ref{ch.notation and setup}, we introduce the Brownian Gibbs property and the more general objects to which our results apply, and then state the main result in its general form as Theorem~\ref{t.airytail.ln}. Chapter~\ref{ch.proof framework} sets up the framework in which our proof operates: in Section~\ref{s.jump ensemble} we introduce the jump ensemble, and in Section~\ref{s.heuristics for proof} we provide a conceptual framework for the proof of the principal result. Finally, the main theorem is proved in Chapter~\ref{ch.the proof} across four sections, each covering a different case. Chapter~\ref{ch.patchwork quilt} is devoted to introducing the definitions of the patchwork quilt and stating and proving precise versions of the statements qualitatively described in Section~\ref{s.intro general init condition Brownian regularity}. In particular, Theorem~\ref{t.weight profile increment moment bound} is proved as Corollary~\ref{c.precise increment moment bound}.

We have attempted to state lemmas as they are required and give proofs immediately in most cases, with a few exceptions. One exception we highlight here is in Chapter~\ref{ch.application proofs}, where some straightforward calculations involving Brownian meander required for the proofs of Theorems~\ref{t.local johansson} and \ref{t.many near touch} have been deferred to Appendix~\ref{a.brownian meander calculations}.

\subsection*{Acknowledgments} The authors thank Ivan Corwin for pointing them to \cite[Lemma 3.3]{quastel2013local}, and Shirshendu Ganguly for helpful discussions. Alan Hammond is supported by the NSF through grants DMS-1512908 and DMS-1855550. Milind Hegde acknowledges the generous support of the U.C. Berkeley Mathematics Department through a summer grant and the Richman Fellowship.


\chapter{Proofs of applications}\label{ch.application proofs}

In this chapter we use Theorem~\ref{T.MAIN THEOREM FOR AIRY} to provide the proofs of two of the applications described in Section~\ref{s.intro applications}, namely Theorems~\ref{t.local johansson} and \ref{t.many near touch}. The strategy of proof, of course, is to prove a similar bound as desired for Brownian motion and then translate it to one for the parabolic Airy$_2$ process $\L$.

\section{Local Johansson}\label{s.application proofs.local johansson}
We start with the proof of Theorem~\ref{t.local johansson}, the quantified and local version of Johansson's conjecture. In fact, we will prove a stronger result which immediately implies Theorem~\ref{t.local johansson}. For an interval $I\subseteq [-d,d]$, and a stochastic process $X$ on $[-d,d]$, we define the \emph{maximiser location} event $\maxloc(X,I)$ that $x_{\max}\in I$. Let $\mathrm{ArcSin}_d(I)$ be the measure of $I$ under the arcsine law on $[-d,d]$, which has density $\pi^{-1}\left(d^2-x^2\right)^{-1/2}$ for $x\in(-d,d)$.

\begin{proposition}\label{p.near touch  and max loc for L}
Let $d\geq 1$, $I\subseteq [-d,d]$, $a\in(0,1)$, $\eta>0$, and $\epsilon = \mathrm{ArcSin}_d(I)$. Then there exist $\epsilon_0$ and $G<\infty$ such that, for $0<a\epsilon< \epsilon_0$,
$$\P\big(\nt(\L,\eta,a)\cap \maxloc(\L,I)\big) \leq a\epsilon\cdot \exp\left(Gd (\log (a\epsilon)^{-1})^{5/6}\right).$$
\end{proposition}

By taking $I = [-d,d]$, we obtain Theorem~\ref{t.local johansson}. To prove Proposition~\ref{p.near touch  and max loc for L}, we isolate the result for Brownian motion as a separate proposition.

\begin{proposition}\label{p.near touch for bm}
Let $d\geq 1$, $\eta>0$, and $B$ be standard Brownian motion on $[-d,d]$. Then
$$\P\big(\nt(B,\eta,a)\cap \maxloc(B,I)\big) \leq 4a\cdot \mathrm{ArcSin}_d(I).$$
\end{proposition}

We first apply Theorem~\ref{T.MAIN THEOREM FOR AIRY} via Proposition~\ref{p.near touch for bm} to obtain Proposition~\ref{p.near touch  and max loc for L}.

\begin{proof}[Proof of Proposition~\ref{p.near touch  and max loc for L}]
This follows immediately from Proposition~\ref{p.near touch for bm} and Theorem~\ref{T.MAIN THEOREM FOR AIRY} after noting that both events are unchanged by a vertical shift, i.e.,
\begin{equation*}
\nt(\L, \eta, a) = \nt(\L(\,\cdot\,)-\L(-d),\eta, a) \quad \text{and} \quad 
\maxloc(\L, I) = \maxloc(\L(\,\cdot\,)-\L(-d), I). \qedhere
\end{equation*}
\end{proof}

To prove Proposition~\ref{p.near touch for bm} (as well as other statements in this chapter), we will make use of the classic connection between Brownian motion around its unique maximiser and \emph{Brownian meander}. Recall that standard Brownian meander on $[0,1]$ is a (non-homogeneous) Markov process which can be intuitively understood as a Brownian motion conditioned to remain positive (though this singular conditioning requires rigorous treatment). See \cite{biane1988quelques,imhof1984density,pitman1999brownian,revuz2013continuous} for references on Brownian meander.

We will need a number of calculations involving the Brownian meander; to maintain this article's focus, in this chapter we will state the results of the calculations as needed, but we will defer the proofs to Appendix~\ref{a.brownian meander calculations} in cases where we think the proof is not illuminating.

We now record the classical decomposition of Brownian motion around its maximiser to which we alluded above; a proof can be found, for example, in \cite{denisov1984random}.

\begin{proposition}\label{p.decomposition around max}
Let $x_{\mathrm{max}}$ be the a.s. unique maximiser of a standard Brownian motion $B$ on $[0,1]$, and let $B(x_{\mathrm{max}}) = M$. Then the processes 
\begin{align*}
\bme^\rightarrow &:=(1-x_{\mathrm{max}})^{-1/2}\Big(M-B\big(x_{\mathrm{max}} + u(1-x_{\mathrm{max}})\big)\Big)_{u\in[0,1]}, \qquad \text{and}\\
\bme^\leftarrow &:=x_{\mathrm{max}}^{-1/2}\Big(M-B\big(x_{\mathrm{max}}(1-u)\big)\Big)_{u\in[0,1]}
\end{align*}
are standard Brownian meanders, independent of each other and of $x_{\mathrm{max}}$.
\end{proposition}

This proposition suggests that we should look at the event that a Brownian meander comes near zero. Let use define, for a Brownian meander $\bme$, the event $\nz(\bme,\eta,a)$ by
$$\nz(\bme,\eta,a) = \left\{\inf_{u\in[\eta,1]}\bme(u) < a\eta^{1/2}\right\},$$
where the event is empty if $\eta>1$. The following lemma is the main input for Proposition~\ref{p.near touch for bm}.

\begin{lemma}\label{l.near zero meander}
Let $\bme$ be a standard Brownian meander on $[0,1]$. Let $a\in(0,1)$ and $\eta>0$. Then
$$\P\Big(\nz(\bme,\eta,a)\Big) \leq 2a.$$
\end{lemma}

The proof of Lemma~\ref{l.near zero meander} proceeds by conditioning on $\bme(\eta)$ and using transition probability formulas for Brownian meander to obtain the bound; as the calculations are straightforward, we have deferred the proof to Appendix~\ref{a.brownian meander calculations}. We now show how we may use Proposition~\ref{p.decomposition around max} and Lemma~\ref{l.near zero meander} to obtain Proposition~\ref{p.near touch for bm}, which completes the proof of Theorem~\ref{t.local johansson}.

\begin{proof}[Proof of Proposition~\ref{p.near touch for bm}]
By Brownian scaling and translation, it is sufficient to prove Proposition~\ref{p.near touch for bm} on the interval $[0,1]$, which corresponds to $d=1/2$. Observe that
\begin{equation}\label{e.nt condition on x_max}
\P\Big(\nt(B,\eta,a)\cap \maxloc(B,I) \ \big|\  x_{\max}\Big) = \P\Big(\nt(B,\eta,a)  \ \big|\   x_{\max}\Big)\cdot\one_{x_{\max\in I}};
\end{equation}
and that, from Proposition~\ref{p.decomposition around max},
\begin{equation}\label{e.nt breakup}
\P\Big(\nt(B,\eta,a)  \ \big|\  x_{\max}\Big)\leq \P\Big(\nz(\bme^\rightarrow, \tilde\eta_1,a) \ \big|\  x_{\max}\Big) + \P\Big(\nz(\bme^\leftarrow, \tilde\eta_2,a) \ \big|\  x_{\max}\Big),
\end{equation}
where $\tilde\eta_1$ and $\tilde\eta_2$ are defined as $\tilde\eta_1 = (1-x_{\max})^{-1}\eta$ and $\tilde\eta_2 = x_{\max}^{-1}\eta.$

From  Lemma~\ref{l.near zero meander} and the independence of $\bme^\rightarrow$, $\bme^\leftarrow$, and $x_{\max}$, we have that each term on the right-hand side of \eqref{e.nt breakup} is bounded by $2a$, giving an overall bound of $4a$ for $\P\left(\nt(B,\eta,a)\mid x_{\max}\right)$. Finally, the statement of Proposition~\ref{p.near touch for bm} follows from \eqref{e.nt condition on x_max} and the well-known arcsine law of the maximiser of Brownian motion (see, for example, \cite[Chapter 5]{morters2010brownian}).
\end{proof}

\section{Many near touch}

Now we turn to proving Theorem~\ref{t.many near touch}. We will of course first establish a similar result for Brownian motion.

\begin{proposition}\label{p.many near touch for bm}
Let $d\geq 1$, $B:[-d,d]\to\R$ be a standard Brownian motion, and let $0<\eta<d$. Then there exist $\ell_0$ and $c>0$ such that, for $\ell>\ell_0$,
$$\P\big(\numnt(B, \eta) \geq \ell\big) \leq e^{-c\ell}.$$
\end{proposition}

Similarly to before, combining this with Theorem~\ref{T.MAIN THEOREM FOR AIRY} readily yields Theorem~\ref{t.many near touch}.

\begin{proof}[Proof of Theorem~\ref{t.many near touch}]
Again observing that $\numnt(\L,\eta) = \numnt(\L(\,\cdot\,)-\L(-d),\eta)$, we may apply Theorem~\ref{T.MAIN THEOREM FOR AIRY} to Proposition~\ref{p.many near touch for bm} to get that, for large enough $\ell$ and some $G<\infty$,
$$\P\big(\numnt(\L, \eta) \geq \ell\big) \leq e^{-c\ell + Gd\ell^{5/6}}.$$
We may reduce $c$ to $c'=c/2$ and enforce a sufficiently high lower bound on $\ell$ to obtain that this probability is bounded by $e^{-c'\ell}$, which is the statement of Theorem~\ref{t.many near touch} with $c'$ replacing $c$ in the exponent.
\end{proof}

To prove Proposition~\ref{p.many near touch for bm} we will again rely on Proposition~\ref{p.decomposition around max}. To do so, we must bound the number of times that standard Brownian meander comes within (a particular multiple of) $\eta^{1/2}$ of $0$. Let us define the random variable recording the \emph{number of near zeroes} $\numnz(X)$ of a process $X:[0,1]\to\R$ with $X(0)=0$ by
$$\numnz(X,\eta) = \max \left\{|S| \ \middle|\ \begin{tabular}{@{}c@{}}
$S\subseteq [0,1]$  finite, $s\in S \implies |X(s)|\leq \eta^{1/2}$,\\
 and  $s,t\in S$  with $s\neq t \implies |s-t|\geq \eta$,
\end{tabular} \right\}.
$$
Let $\bme$ denote the standard Brownian meander. The following estimate on $\numnz(\bme,\eta)$ is the main estimate needed for Proposition~\ref{p.many near touch for bm}.

\begin{proposition}\label{p.num near zero meander}
Let $0<\eta\leq 1/2$. Then there exists $\ell_0$ and $c>0$ such that, for $\ell>\ell_0$,
$$\P\big(\numnz(\bme, \eta) \geq \ell\big) \leq e^{-c\ell}.$$
\end{proposition}

We show how Proposition~\ref{p.many near touch for bm} follows from Proposition~\ref{p.num near zero meander}, and then turn to proving the latter.

\begin{proof}[Proof of Proposition~\ref{p.many near touch for bm}]
Again by Brownian scaling and translation of the domain, it is enough to prove Proposition~\ref{p.many near touch for bm} on the interval $[0,1]$, which corresponds to $d=1/2$. (Notice that $\eta<d$ before scaling becomes $\eta<1/2$ after.) 

As in Proposition~\ref{p.decomposition around max}, conditionally on $x_{\max}$, let $\bme^\rightarrow$ and $\bme^\leftarrow$ denote the two independent Brownian meanders on $[0,1]$, with $\bme^\rightarrow$ corresponding to the right of the maximiser and $\bme^\leftarrow$ to the left. It will be helpful in the proof to consider separately the number of near touches that occur on the right and on the left. Accordingly, we define, with $M= X(x_{\max})$,
\begin{align*}
\numnt^\rightarrow(X, \eta) &= \max \left\{|S| \ \middle|\ \begin{tabular}{@{}c@{}}
$S\subseteq [0,1]$  finite, $s\in S \implies M-X(s)\leq \eta^{1/2}$, and\\
 $s,t\in S$  with $s\neq t \implies |s-t|\geq \eta,\  \min S>x_{\max}$
\end{tabular} \right\} \quad \text{and}\\
\numnt^\leftarrow(X, \eta) &= \max \left\{|S| \ \middle|\ \begin{tabular}{@{}c@{}}
$S\subseteq [0,1]$  finite, $s\in S \implies M-X(s)\leq \eta^{1/2}$, and\\
 $s,t\in S$  with $s\neq t \implies |s-t|\geq \eta,\  \max S < x_{\max}$
\end{tabular} \right\},
\end{align*} 
the superscript arrows indicating in which direction from $x_{\max}$ we are considering the process. Now we have
\begin{align*}
\P\big(\numnt(B, \eta) \geq \ell\big) \leq \P\big(\numnt^{\rightarrow}(B, \eta) \geq \ell/2\big) + \P\big(\numnt^{\leftarrow}(B, \eta) \geq \ell/2\big).
\end{align*}
We give only the argument to bound the first term in this break-up, as the argument for the second is identical. Observe that, if $\ell\geq 5$, we have
\begin{align*}
\P\big(\numnt^{\rightarrow}(B, \eta) \geq \ell/2\mid x_{\max}\big) = \P\big(\numnt^{\rightarrow}(B, \eta) \geq \ell/2 \mid x_{\max}\big)\cdot\one_{x_{\max} \leq 1-2\eta},
\end{align*}
%
%
since the number of near touches that can occur in an interval of length $2\eta$ is at most two.
Now for the given $\eta>0$, define 
$\tilde\eta = (1-x_{\max})^{-1}\eta.$
Observe that, when $x_{\max}\leq 1-2\eta$, we have that $\tilde\eta \leq 1/2$, allowing us to apply Proposition~\ref{p.num near zero meander}. Applying Proposition~\ref{p.decomposition around max}  and then Proposition~\ref{p.num near zero meander}, we have
\begin{align*}
\P\big(\numnt^\rightarrow(B, \eta) \geq \ell/2\mid x_{\max}\big)\cdot\one_{x_{\max\leq 1-2\eta}} = \P\big(\numnz(\bme^\rightarrow, \tilde \eta) \geq \ell/2\mid x_{\max}\big)\cdot\one_{x_{\max\leq 1-2\eta}} \leq e^{-c\ell},
\end{align*}
so that the proof is complete.
%
\end{proof}

So it remains to prove Proposition~\ref{p.num near zero meander}. The strategy is to consider a random variable $N$ which dominates $\numnz(\bme, \eta)$. 
Define a sequence of stopping times $\tau_i$ inductively. We set $\tau_0=0$, and, for $k\geq 1$,
\begin{align*}
\tau_k &= \begin{cases}
\tau_{k-1}+\eta & \text{if }\bme(\tau_{k-1}+\eta)\leq 1.1\eta^{1/2}, \tau_{k-1}<1-9\eta\\
\inf\big\{t\geq \tau_{k-1}+\eta : \bme\left(t\right) \leq \eta^{1/2}, t\leq 1-10\eta\big\} & \text{if } \bme(\tau_{k-1}+\eta)> 1.1\eta^{1/2}, \tau_{k-1}<1-9\eta\\
\infty & \text{if } \tau_{k-1}\geq 1-9\eta.
\end{cases}
\end{align*}
We adopt the convention that the infimum of the empty set is $+\infty$. Also, let $k_0=0$ and $k_1<k_2< \ldots $ be the indices $k$ where $\tau_{k-1}+\eta< \tau_k <\infty$, i.e., where the second branch of the definition of $\tau_{k}$ was active, and let $K$ be the number of times this occurred; i.e., $K=\big|\{j : \tau_{j-1}+\eta < \tau_j <\infty\}\big|$. 

Now we define $N$ to be
$$N = 10 + \max\left\{j : \tau_j<\infty\right\}.$$
In essence, $N$ counts the consecutive $\eta$-separated instances of $\bme(x)< 1.1 \eta^{1/2}$; but, when this condition is violated, no count is made until $\bme$ returns to $\eta^{1/2}$. As such, $N$ is an overcount of $\numnz(\bme, \eta)$. Because some of our upcoming estimates on Brownian meander only hold on $[0,1-10\eta]$, we simply assume that the maximum possible number of $\eta$-separated near zeroes occur in the interval $[1-10\eta,1]$, which is 10. From here the claimed stochastic domination of $N$ over $\numnz(\bme, \eta)$ is apparent.

We may write $N$ in terms of the $k_i$. Define $N_1=k_1-1$, $N_i = (k_i-k_{i-1})$ for $1<i<K$, and $N_K = N-9-k_{K-1}$. Then
$$N = 10 + \sum_{i=1}^K N_i.$$





The idea of the proof is that the $N_i$ are dominated by geometric random variables (whose parameter is given by a uniform in $t$ bound on the probability that $B(t+\eta)\leq 1.1\eta^{1/2}$ given $B(t)\leq 1.1\eta^{1/2}$), and $K$ is also dominated by a geometric variable. By the strong Markov property, all these variables are independent. The needed uniformity in the bounds on the parameters of these random variables is a consequence of the following probability estimates. The proofs of these are straightforward but somewhat tedious calculations which have been deferred to Appendix~\ref{a.brownian meander calculations}.

The first two lemmas bound the probability that the Brownian meander is within $1.1\eta^{1/2}$ of 0 after an interval of time $\eta$, given that it is so at the current time. The second lemma treats the case that the current time is not 0, while the first does so when the time is 0, as the transition densities are slightly different in the two cases. These bounds will give the parameter for the geometric variables bounding $N_i$.

\begin{lemma}\label{l.brownian meander from 0 estimate}
Let $\bme$ be a standard Brownian meander on $[0,1]$. For $0<\eta\leq 1/2$,
$$\P\Big(\bme(\eta) < 1.1\eta^{1/2}\Big) \leq \frac{1}{2}.$$
\end{lemma}

\begin{lemma}\label{l.brownian meander increment estimate}
Let $\bme$ be a standard Brownian meander on $[0,1]$. Fix $\eta>0$ and let $t\in (\eta,1-4\eta]$. For $\eta>0$ and $x\leq 1.1\eta^{1/2}$,
$$\P\Big(\bme(t) < 1.1\eta^{1/2}  \ \Big|\  \bme\left(t-\eta\right)=x\Big) \leq \frac{3}{4}.$$
\end{lemma}

The final lemma bounds the probability that the Brownian meander returns to level $\eta^{1/2}$  before time one if it is currently above $1.1\eta^{1/2}$. This estimate will be needed for the parameter of the geometric variable bounding~$K$.

\begin{lemma}\label{l.prob of return from above}
Let $\bme$ be a standard Brownian meander on $[0,1]$. Fix $\eta>0$ and $t \in (\eta, 1-10\eta]$. There exists an absolute constant $\delta>0$ (independent of $\eta$ and $t$) such that, for $\eta>0$ and $x>1.1\eta^{1/2}$,
$$\P\left(\inf_{s\in[t-\eta, 1]} \bme(s) < \eta^{1/2}  \ \Big|\  \bme\left(t-\eta\right) =x\right) \leq 1-\delta.$$

\end{lemma}

We will also need a tail bound on a sum of independent geometric random variables. Here, ``$X$ is distributed as $\mathrm{Geo}(p)$'' means
$$\P\left(X\geq k\right) = (1-p)^k$$
for $k=1,2, \ldots$ For the tail bound we cite a simplified form of \cite[Theorem 2.1]{janson2018tail}.

\begin{lemma}\label{l.tail of sum of geometrics}
Let $G_i$ be i.i.d. $\mathrm{Geo}(p)$ random variables, and let $G=\sum_{i=1}^n G_i$, for $n\in\N$. Also let $\mu = \E[G] = n/p$ and $\lambda\geq 1$. Then,
$$\P\bigl(G\geq \lambda\mu\bigr)\leq \exp(-p\mu\lambda).$$
\end{lemma}

Given these estimates, we may turn to the proof of Proposition~\ref{p.num near zero meander}.

\begin{proof}[Proof of Proposition~\ref{p.num near zero meander}]
As noted, we have $\P\left(\numnz(\bme, \eta)\geq \ell\right) \leq \P(N\geq \ell)$. We want to dominate $N$ by a sum of a geometric number of independent geometric random variables using the well-known strong Markov property of $\bme$. 

Let $G_i$ be distributed as $\mathrm{Geo}(1/4)$ for $i =1,2, \ldots$ and $K'$ as $\mathrm{Geo}(\delta)$, all independent of each other. Then it is straightforward from the above description and the probability estimates of Lemmas~\ref{l.brownian meander from 0 estimate}, \ref{l.brownian meander increment estimate}, and \ref{l.prob of return from above}, along with the strong Markov property, that we have a coupling of the $N_i$, $G_i,$ $K$, and $K'$ such that
$$N = 10+\sum_{i=1}^K N_i \leq 10+\sum_{i=1}^{K'} G_i.$$
Indeed, for $i\leq K-1$, $N_i$ is the number of times that the meander is within $\eta^{1/2}$ of zero when checked at times separated by exactly $\eta$; and so the joint domination of $N_i$ by $G_i$ follows from Lemmas~\ref{l.brownian meander from 0 estimate} and \ref{l.brownian meander increment estimate} and the strong Markov property. For $i=K$, we have essentially the same thing except that the count is truncated when the time reaches $1-10\eta$; as it is a truncation, the same stochastic domination still holds. Similarly, $K$ is the number of consecutive times the process is able to get from $1.1\eta^{1/2}$ to $\eta^{1/2}$ in a unit interval, and so by the strong Markov property and Lemma~\ref{l.prob of return from above}, $K$ is dominated by $K'$.

Thus we have, noting that $\E\big[\sum_{i=1}^{K'} G_i\mid K'\big] = K'\cdot\E[G_i] = 4K'$ and using Lemma~\ref{l.tail of sum of geometrics},
\begin{align*}
 \P\big(N\geq \ell \mid K'\big) \leq \P\left(\sum_{i=1}^{K'} G_i \geq \ell/2\  \middle| \ K'\right)
 %
 &= \exp\left(-\ell/8\right)\cdot\one_{\ell\geq 8K'} + \one_{\ell\leq 8K'},
 \end{align*}
for $\ell\geq 20$. Taking expectations, we find
\begin{align*}
\P\big(N\geq \ell \big)  &\leq \exp(-\ell/8) + \P\left(K'\geq \ell/8\right)
= \exp(-\ell/8) + (1-\delta)^{\ell/8},
\end{align*}
completing the proof.
\end{proof}


\chapter{Notation and setup}\label{ch.notation and setup}

In this chapter we introduce some notation we will be using throughout the article; give the definitions of the main objects of study; and then state the main result, Theorem~\ref{t.airytail.ln}.

\section{Notation, Brownian Gibbs, and regular ensembles}\label{s.notation and regular ensembles}

\subsection{General notation}\label{s.notation general}
We take the set of natural numbers $\N$ to be $\{1,2, \ldots \}$. For $k\in \N$, we use an overbar to denote a $k$-vector, i.e., $\overline x\in \R^k$. We denote the integer interval $\{i,i+1, \ldots, j\}$ by $\llbracket i, j\rrbracket$. For a function $f:\intint{k}\times \R\to\R$, we write $\overline f(x)$ for $(f(1,x), \ldots, f(k,x))$. A $k$-vector $\overline x = (x_1, \ldots, x_k)\in \R^k$ is called a $k$-decreasing list if $x_1>x_2> \ldots >x_k$. For a set $I\subseteq \R$, let $I^k_> \subseteq I^k$ be the set of $k$-decreasing lists of elements of $I$, and $I^k_{\geq}$ be the analogous set of $k$-non-increasing lists.

For a real valued function $f$ whose domain of definition contains an interval $[a,b]$, we define $f^{[a,b]}:[a,b]\to \R$ to be the affinely shifted bridge version of $f$ that is zero at both endpoints, i.e., for $x\in[a,b]$, 
$$f^{[a,b]}(x) := f(x) - \frac{x-a}{b-a}\cdot f(b) - \frac{b-x}{b-a}\cdot f(a).$$
For an interval $[a,b]\subseteq \R$, we denote the space of continuous functions with domain $[a,b]$ which vanish at~$a$ by $\mc C_{0,*}([a,b], \R)$, and the space of continuous functions which may take any value at the endpoints by $\mc C_{*,*}([a,b], \R)$. The asterisk should be thought of as a wildcard indicating that any value may be taken.

\subsection{Line ensembles and the Brownian Gibbs property}
\begin{definition}[Line ensembles]
Let $\Sigma$ be an (possibly infinite) interval of $\Z$, and let $\Lambda$ be a (possibly unbounded) interval of $\R$.
Let $\mc X$ be the set of continuous functions $f:\Sigma\times \Lambda \rightarrow \R$ endowed with the topology of uniform convergence on compact subsets of $\Sigma\times\Lambda$, and let $\mathscr{C}$ denote the Borel $\sigma$-algebra of $\mc X$.

A {\it $\Sigma$-indexed line ensemble} $\L$ is a random variable defined on a probability space $(\Omega,\mathscr{B},\PP)$, taking values in $\mc X$ such that $\L$ is a $(\mathscr{B},\mathscr{C})$-measurable function. We regard $\L$ as a $\Sigma$-indexed collection of random continuous curves (despite the usage of the word ``line''), each of which maps $\Lambda$ into $\R$. We will slightly abuse notation and write $\L:\Sigma\times \Lambda \rightarrow \R$, even though it is not $\L$ which is such a function, but rather $\L(\omega)$ for each $\omega \in \Omega$. 
A line ensemble is {\em ordered} if, for all $i,j \in \Sigma$ with $i<j$, it holds that $\L(i,x)>\L(j,x)$ for all $x \in \Lambda$.
Statements such as this are understood as being asserted almost surely with respect to $\P$.
\end{definition}

\begin{definition}[Normal, Brownian bridge, and Brownian motion laws]
We will use $N(m,\sigma^2)$ to denote the normal distribution with mean $m$ and variance $\sigma^2$, and sometimes, with abuse of notation, a random variable with this distribution.

Let $k \in \N$, $a, b \in \R$ with $a<b,$ and $\overline{x}, \overline{y} \in \R_{>}^{k}$. We write $\B_{k; \overline x, \overline y}^{[a,b]}$ for the law of $k$ independent Brownian bridges $(B_1,\ldots, B_k)$ of diffusion parameter one, with $B_i : [a,b] \to \R$ and $B_i(a) = x_i$ and $B_i(b) = y_i$, for $i=1, \ldots, k$.

We will also need the law of standard Brownian motion started at 0 on the interval $[a,b]$, which we will denote by $\B_{0, *}^{[a,b]}$; i.e., $\B^{[a,b]}_{0,*}$ is the law of a rate one Brownian motion $B$ with $B(a)=0$. 

Now let $f:[a,b]\to \R\cup\{-\infty\}$ be a measurable function such that $x_k> f(a)$ and $y_k> f(b)$. Define the non-intersection event on a set $A\subseteq [a,b]$ with lower boundary curve $f$ by
$$\notouch^A_f =  \Big\{ \, \textrm{for all } x \in A \, ,  \,  B(i,x) > B(i+1,x) \textrm{ for each } 1\leq i\leq k-1, \textrm{ and  $B(k,x) > f(x)$} \,\Big\} \, . $$
When $A=[a,b]$, we omit its mention in the notation, i.e., we write $\notouch_f$.
\end{definition}

With this definition, we can move to defining the Brownian Gibbs property.

\begin{definition}[Brownian Gibbs property]
Let $n\in \N$, $I\subseteq \R$ be an interval, $k\in\intint{n}$, and $a,b\in I$ with $a<b$. Let $D_{k;a,b} = \intint{k}\times(a,b)$ and $D_{k;a,b}^c = (\intint{n}\times I)\setminus D_{k;a,b}$. Let $\L:\intint{n}\times I\to \R$ be an ordered line ensemble. We say that $\L$ has the Brownian Gibbs property if the following holds for all such choices of $k, a$, and $b$:
$$\mathrm{Law}\left(\left.\mathcal{L}\right|_{D_{k; a, b}} \text { conditionally on }\left.\mathcal{L}\right|_{D_{k ; a, b}^{c}}\right)=\B_{k; \overline{x}, \overline{y}}^{[a, b]}\left(\,\cdot \mid \notouch_{f}\right),$$
where $\overline x = \overline \L(a)$, $\overline y=\overline \L(b)$, and $f(\,\cdot\,) = \L(k+1,\cdot\,)$ on $[a,b]$.

In words, the conditional distribution of the top $k$ curves of $\L$ on $[a,b]$, given the form on $\L$ on $D^c_{k;a,b}$, is the law of $k$ independent Brownian bridges, the $i^\text{th}$ from $\L(i,a)$ to $\L(i,b)$, which are conditioned to intersect neither each other nor the lower curve $\L(k+1,\cdot\,)$ on $[a,b]$.
\end{definition}

In the next definition we define \emph{regular ensembles}, which are the general objects to which our main result will apply. The definition is the same as~\cite[Definition $2.4$]{hammond2017brownian}, with the parameter $\overline \varphi$ in that definition taking the value $(1/3,1/9, \infty)$; the value of $\infty$ for the third parameter is a formal device to indicate that the range of $s$ in point (2) below is $[1,\infty)$ instead of $[1, n^{\varphi_3}]$ for a finite value of $\varphi_3$.

\begin{definition}[Regular Brownian Gibbs ensemble]\label{d.regularsequence} 
Consider a Brownian Gibbs ensemble that has the form
$$
\mc{L}: \intint{\nmac} \times \big[-\xnmac , \infty \big) \to \R   \, ,
$$

and which is defined on a probability space under the law~$\PP$.
The number $\nmac = \nmac(\L)$ of ensemble curves and the absolute value $\xnmac$ of the finite endpoint may take any values in $\N$ and $[0,\infty)$. (In fact, we may also take $\xnmac = \infty$, in which case we would take the domain of $\mc{L}$ to be $\intint{n} \times \R$.) Let $\rsC$ and $\rsc$ be two positive constants. The ensemble $\mc{L}$ is said to be $(\rsc,\rsC)$-regular if the following conditions are satisfied.

\begin{enumerate}

\item {\bf Endpoint escape.} $\xnmac \geq  \rsc \nmac^{1/3}$.

\item {\bf One-point lower tail.} If $z \geq -\xnmac$ satisfies $\vert z \vert \leq \rsc \nmac^{1/9}$, then

$$
\PP \Big( \mc{L} \big( 1,z\big) + 2^{-1/2}  z^2 \leq - s \Big) \leq \rsC \exp \big\{ - \rsc s^{3/2} \big\}
$$

for all $s \in \big[1, \infty \big)$.

\item {\bf One-point upper tail.}  If $z \geq -\xnmac$ satisfies $\vert z \vert \leq \rsc \nmac^{1/9}$, then

$$
\PP \Big( \mc{L} \big( 1,z\big) +  2^{-1/2} z^2 \geq  s \Big) \leq \rsC \exp \big\{ - \rsc s^{3/2} \big\}
$$

for all $s \in [1, \infty)$.

\end{enumerate}
We reserve the symbols $c$ and $C$ for this usage in the remainder of this paper.

The symbol $n$ will be reserved in the rest of the paper for the number of curves in the regular ensemble under consideration, which we will denote by $\L_n$.
\end{definition}

Though the definition of regular ensembles only includes one-point tail information for the top curve, this actually extends to the lower curves as well \cite[Proposition 2.7]{hammond2017brownian}. Though we do not state this result, we will have need of two associated sequences of constants for the statement of our main results. For a $(c,C)$-regular ensemble, define $C_1 = 140 C$, $c_1 = 2^{-5/2}c\wedge 1/8$; and, for each $k\geq 2$,
\begin{equation*}\label{e.formere}
 \formerE_k = \max \Big\{  10 \cdot 20^{k-1} 5^{k/2} \Big( \tfrac{10}{3 - 2^{3/2}} \Big)^{k(k-1)/2} C \, , \, e^{c/2} \Big\} 
 \end{equation*}
and 
\begin{equation}\label{e.littlec}
 c_k =   \big( (3 - 2^{3/2})^{3/2} 2^{-1} 5^{-3/2} \big)^{k-1} c_1.
\end{equation}
These symbols will retain these meanings throughout the article.
%

One example of a regular Brownian Gibbs line ensemble is the \emph{parabolic Airy line ensemble}, given by 
$$\L(i,x) = 2^{-1/2}(\A(i,x) - x^2),$$
for $(i,x)\in \N\times \R$, where $\A:\N\times \R\to \R$ is the Airy line ensemble. (We again include a factor of $2^{-1/2}$ to allow comparisons to be made with rate \emph{one} Brownian objects.) The Airy line ensemble was constructed as an ensemble of continuous non-intersecting curves in \cite[Theorem 3.1]{corwin2014brownian}, and tightness estimates furnished by each of \cite{dauvergne2018basic}, \cite{dauvergne2019uniform}, and \cite{hammond2017brownian} lead to simplified constructions. It is defined as follows.

\begin{definition}[Airy line ensemble]
The Airy line ensemble $\mc{A}: \N \times \R \to \R$  is a collection of random continuous curves $\mc{A}(j,\cdot\,)$ for $j\in \N$. For any finite set $I \subset \R$, define the random object $\mc{A}[I]$ to be the point process  on $I \times \R$ given by $\big\{ \big( s,\mc{A}(j,s) \big) \mid j \in \N  \, , \, s \in I \big\}$. The law of $\mc{A}$ is defined as the unique distribution supported on such collections of continuous curves such that, for each finite $I = \{ t_1,\cdots,t_m \}$, $\mc{A}[I]$ is a determinantal point process whose kernel is the extended Airy$_2$ kernel~$K^{{\rm ext}}_2$, specified by
$$
K^{{\rm ext}}_2 \big( s_1,x_1;s_2,x_2  \big)  = \begin{cases}
  \int_0^\infty e^{-\lambda(s_1 - s_2)}   {\rm Ai}\big(x_1+\lambda\big){\rm Ai}\big(x_2+\lambda\big) \dd \lambda \, \, & \textrm{if $s_1 \geq s_2$}
 \, , \\
 - \int_{-\infty}^0 e^{-\lambda(s_1 - s_2)}   {\rm Ai}\big(x_1+\lambda\big){\rm Ai}\big(x_2+\lambda\big) \dd \lambda \, \, & \textrm{if $s_1 < s_2$} \, , 
\end{cases}
$$ 
where ${\rm Ai}:\R \to \R$ is the Airy function. The Airy line ensemble's curves are ordered, with $\mc{A}(1,\cdot\,)$ uppermost.
\end{definition}

\section{An important example of regular ensembles: Brownian LPP weight profiles}\label{s.notation brownian lpp}
Here we introduce the Brownian last passage percolation model, which will generate an important example of regular ensembles via the RSK correspondence, and weight profiles from general initial conditions. These definitions are not logically required for the proof of our main theorem, but do motivate our decision to prove the result in the more general context of regular ensembles. Additionally, these objects will be used in the applications involving the patchwork quilt. Many of these objects were introduced in Section~\ref{s.intro} in their analogous forms in Bernoulli LPP with the superscript ``$0|1$'', which is now dropped for their Brownian LPP versions.

\subsection{The model's definition}
We work in a probability space with a law $\P$, which is rich enough to support a collection $B:\Z\times \R\to \R$ of independent two-sided standard Brownian motions $B(k,\cdot\,):\R\to\R$ for $k\in \Z$.

Let $i,j\in \Z$ with $i\leq j$, and $x,y\in\R$ with $x\leq y$. Let $z_k\in [x,y]$ for $k\in \llbracket i+1, j \rrbracket$ be a non-decreasing list, i.e. $z_{i+1}\leq  \ldots  \leq z_j$. We adopt the convention that $z_i =x$ and $z_{j+1}=y$.

To any such list, we associate an energy, which is $\sum_{k=i}^{j} \left(B(k+1,z_{k+1}) - B(k,z_k)\right)$. The maximum energy over all such lists (for fixed $i,j,x,y$) is denoted by $M[(x,i)\to(y,j)]$.

The process $M[(0,1) \to (\cdot,n)]: [0,\infty) \to \R$ may have been considered first by~\cite{glynn1991departures}; it was studied  further in~\cite{o2002representation}.

Taking into account the KPZ scaling exponents of one-third and two-thirds discussed in Section~\ref{s.intro}, we define a centred and scaled version of the maximum energy process $M[(\cdot, 0)\to (\cdot, n)]$, where we have set $i=0$ and $j=n$. We call this centred and scaled quantity the \emph{weight}; it is specified by the formula
\begin{equation}\label{e.weightmzeroone} 
  \weight_{n}[(x,0)\to (y,1)] \,     =  \,   2^{-1/2} n^{-1/3} \Big(  M[(2n^{2/3}x,0) \to (n  + 2n^{2/3}y,n)] - 2n  -  2n^{2/3}(y-x) \Big) \, .
\end{equation}

The quantity  $\weight_{n}[(x,0)\to (y,1)]$ may be expected to be, for given real choices of $x$ and $y$, a unit-order random quantity, whose law is tight in the scaling parameter $n \in \N$. In fact, more is true: for $x=0$, the function $y \mapsto\weight_{n}[(0,0)\to (y,1)]$ (which we call the \emph{weight profile}) is a tight sequence of random functions which converges to~$\L$, the parabolic Airy$_2$ process mentioned in Section~\ref{s.intro}, which is the top curve in the parabolic Airy line ensemble mentioned above. These inferences follow from the relation between $M[(0,1)\to(\cdot, n)]$ and Dyson Brownian motion proved in \cite{o2002representation} and the fact that the scaling limit of Dyson Brownian motion is the Airy$_2$ process \cite{adler2005pdes} in the sense of finite-dimensional distributions, upgraded to the space of continuous functions by \cite{corwin2014brownian}. The equality in distribution with Dyson Brownian motion for the top line alone was proved earlier in \cite{gravner2001limit} and also \cite{baryshnikov2001gues}.

We may regard this function $y \mapsto\weight_{n}[(0,0)\to (y,1)]$ as the top line in an ensemble of $n$ continuous curves which we denote $\L_n^{\scal}:\intint{n}\times [-\frac{1}{2}n^{1/3},\infty)\to \R$, i.e., $\L_n^{\scal}(1,y) = \weight_{n}[(0,0)\to (y,1)]$. We will not formally define the remaining curves in this ensemble as doing so is not required for this article, but merely say that their definition goes via the Robinson-Schensted-Knuth correspondence and the weights of multi-polymer watermelons, which are maximal energy collections of disjoint (except for endpoints) point-to-point paths. We also have that~$\L_n^{\scal}$ converges to the parabolic Airy line ensemble under the standard notion of weak convergence given the locally uniform topology on curves. This is proved by the same references mentioned in the previous paragraph for the top line of the ensemble.


Our reason for considering this ensemble of curves is that it enjoys the Brownian Gibbs property and is in fact regular.

\begin{proposition}[Proposition 2.5 of \cite{hammond2017brownian}]\label{p.lereg}
There exist choices of the positive constants $\rsc$ and $\rsC$ such that each of the scaled Brownian LPP line ensembles
$\mc{L}_n^{\scal}: \intint{n} \times \big[- \tfrac{1}{2} n^{1/3}  , \infty \big) \to \R$, $n \in \N$,
is $\big(\rsc,\rsC\big)$-regular.

 \end{proposition}

\begin{remark}
In fact, \cite[Proposition 2.5]{hammond2017brownian} proves that $\L_n^{\scal}$ is $(c,C)$ regular in a slightly weaker sense, namely with point (2) in Definition~\ref{d.regularsequence} holding for $s\in[1, n^{1/3}]$ only. The argument to extend this to all of $s\in[1,\infty)$ simply replaces the use of \cite[Lemma A.1(1)]{hammond2017brownian} in the proof of \cite[Proposition 2.5]{hammond2017brownian} with \cite[Theorem 3.1]{dauvergne2018basic}; this latter theorem is an improved moderate deviation bound for the $k^\text{th}$ line of Dyson Brownian motion (equivalently, the $k^\text{th}$ eigenvalue of the Gaussian Unitary Ensemble), which we need for only $k=1$. 

We have stated this slightly improved regularity of $\L_n^\scal$ in comparison to the statement of \cite[Proposition 2.5]{hammond2017brownian} for completeness. If Proposition~\ref{p.lereg} were used in place of \cite[Proposition 2.5]{hammond2017brownian} in the arguments of \cite{hammond2017brownian}, minor improvements to certain statements quoted from \cite{hammond2017brownian} that we use later in this paper, in Section~\ref{s.jump ensemble}, could be made; additionally, an improvement in the regularity of the Brownian motion patchwork quilt result proved in Chapter~\ref{ch.patchwork quilt} as Theorem~\ref{t.patchwork quilt} (in particular, $\beta_4 = 1/12$ stated in that theorem would improve to $1/6$) would be available. However, in view of the minor and technical nature of these improvements, we do not formally claim, state, or use them in our arguments, and therefore we will not carry through these improved effects of Proposition~\ref{p.lereg} further.
\end{remark}

\subsection{Basic parabolic symmetry of regular ensembles.}\label{s.notation parabolic invariance} Here we record a straightforward proposition that allows us to translate the interval of consideration and still retain a regular ensemble (with an extra linear term). 

Let $Q : \mathbb{R} \rightarrow \mathbb{R}$ denote the parabola
$Q(x)=2^{-1 / 2} x^{2},$ and let $l : \mathbb{R}^{2} \rightarrow \mathbb{R}$ be given by $l(x, y)=-2^{-1 / 2} y^{2}-2^{1 / 2} y(x-y) .$ Note that
$x \mapsto l(x, y)$ is the tangent line of the parabola $x \mapsto-Q(x)$ at the point $(y,-Q(y)) .$ Note also that,
for any $x, y \in \mathbb{R},$
\begin{equation}\label{e.parabola line relation}
Q(x) = -l(x,y) + Q(x-y).
\end{equation}
For $z_{n} \geq 0,$ consider a regular ensemble $\L_{n} :\intint{n} \times\left[-z_{n}, \infty\right) \to \R .$ For any $y_{n}>-z_{n},$ define
$\L_{n, y_{n}}^{\mathrm{shift}} :[1, n] \times\left[-z_{n}-y_{n}, \infty\right) \to \R$ to be the shifted ensemble given by
$$\L_{n, y_{n}}^{\mathrm{shift}}(i, x)=\L_{n}\left(i, x+y_{n}\right)-l\left(x+y_{n}, y_{n}\right)$$
By \eqref{e.parabola line relation}, $\L_{n, y_{n}}^{\mathrm{shift}} = \L_n(i,x+y_n) + Q(x+y_n) - Q(x)$.

\begin{lemma}[Lemma 2.26 of \cite{hammond2017brownian}]\label{l.shifted is regular}
Let $c,C>0$ and $n\in \N$. Suppose that $\L_n : \intint{n}\times [-z_n, \infty)\to \R$ is a $(c,C)$-regular ensemble. Whenever $y_n\in\R$ satisfies $|y_n|\leq c/2\cdot n^{1/9}$, the ensemble $\L_{n, y_{n}}^{\mathrm{shift}}$ is $(c/2, C)$-regular.
\end{lemma}

This lemma will allow our main result to apply to an interval $[K-d,K+d]$ not necessarily centred at the origin.

\subsection{Other initial conditions}
The initial condition of the weight profile $\weight_{n}[(0,0)\to (y,1)]$, since it is started from the single point $(0,0)$, is the \emph{narrow-wedge} initial condition, and we would like to define energy profiles from general initial conditions. We introduced the notation $\weight^f_{\infty}[(*,0)\to (y,1)]$ in Section~\ref{s.intro general init condition Brownian regularity} for the limiting weight profile from general initial conditions, and we now formally define the pre-limiting version, $\weight_{n}^f[(*,0)\to (y,1)]$, whose limit will be $\weight^f_{\infty}[(*,0)\to (y,1)]$.

Given a general initial condition $f$, the setting should be understood as taking the highest energy path ending at $(y,1)$, where the initial point is allowed to be $(x,0)$ for any $x\in \R$, but with a reward $f(x)$ associated to each $x$ which is added to the weight of any path starting at $(x,0)$. We define spaces of admissible reward functions and the general initial condition weight profile associated to a reward function from one of these spaces in the next two definitions.

\begin{definition}\label{d.if}
Writing  $\bar\Psi = \big(\Psi_1, \Psi_2, \Psi_3 \big) \in (0,\infty)^3$ for a triple of positive reals, we let $\mc I_{\bar \Psi}$
denote the set of measurable functions $f:\R \to \R \cup \{ - \infty \}$ such that
$f(x) \leq \Psi_1 \big( 1 + |x|\big)$
and $\sup_{x \in [-\Psi_2,\Psi_2]} f(x) > - \Psi_3$.
\end{definition}

\begin{definition}
For $f$ lying in one of the function spaces $\mc I_{\bar\Psi}$, we now define the \emph{$f$-rewarded} line-to-point polymer weight  $\weight^f_{n}[(*,0)\to (y,1)]$ according to  
$$
 \weight^f_{n}[(*,0)\to (y,1)] \,  = \,  \sup \, \Big\{ \,  \weight_{n}[(x,0)\to (y,1)]    + f(x) \ : \  x \in (-\infty,2^{-1}n^{1/3} + y] \, \Big\}  \, .
$$
\end{definition}

Unlike in the narrow-wedge case, for general initial conditions we do not define a corresponding ensemble of non-intersecting curves, as this ensemble does not enjoy the Brownian Gibbs property. However, as previously discussed in Section~\ref{s.intro general init condition Brownian regularity}, we may still consider a form of Brownian regularity possessed by the $f$-rewarded weight profile using the Brownian motion regularity of the narrow-wedge weight profile, which we discuss more fully in Chapter~\ref{ch.patchwork quilt}.


\section{Main result}\label{s.main results}

For $k\in\N$, let $D_k$ be a sequence of constants depending only on $k$, given by
\begin{equation}\label{e.D_k value}
D_k = \max \left\{k^{1 / 3} c_{k}^{-1 / 3}\left(2^{-9 / 2}-2^{-5}\right)^{-1 / 3}, 36\left(k^{2}-1\right), 2\right\}
\end{equation}
for $k\geq 2$, and set $D_1=D_2$; here $c_k$ is as given in \eqref{e.littlec}. This will be the value of $D_k$ for the rest of the article.

Our main result will concern an interval $[K-d, K+d]$ for $K\in\R$ and $d\geq 1$. For such $K$, define the linear function $\ell_{K,d}:\R\to\R$ by
$$\ell_{K,d}(x) = 2^{1/2}K(x- K +d).$$
Our main result is a generalization of Theorem~\ref{T.MAIN THEOREM FOR AIRY} that applies to the $k^\text{th}$ curve (for fixed $k\in\N$) of any regular ensemble with sufficiently many curves.

\begin{theorem}\label{t.airytail.ln}

Suppose that $\mc{L}_n$ is an $n$-curve $(c,C)$ regular ensemble for some $(c,C)\in (0,\infty)^2$. Let $\ipdval \geq 1$ denote a parameter. Let $K \in \R$ satisfy $[K-d,K+ d] \subset \rsc/2 \cdot [-n^{1/9},n^{1/9}]$,  and let $k \in \N$.

Suppose that $n \geq k  \vee  (c/3)^{-18} \vee  6^{36}$.
For any Borel measurable $A \subset \mc{C}_{0,*}\big([K-d,K+d]\big)$, write  $\epsilon = \mc{B}_{0,*}^{[K-d,K+d]}(A)$. Suppose that $\epsilon$ satisfies the $(k, d)$-dependent upper bound 
$\epsilon < e^{-1} \wedge (17)^{-1/k} C_k^{-1/k} \const^{-1} \wedge \exp(-(24)^{6}d^6/D_k^3)$;
 as well as the $n$-dependent lower bound
\begin{equation}\label{e.alowerbound}
\epsilon \geq \exp \Big\{   - \big( \rsc/2 \wedge 2^{1/2} \big) \const^{-1} n^{1/12} \Big\} \, .
\end{equation} 
Then there exists $G<\infty$ such that
\begin{align}
 \P \,  \Big(\L_n\big(k, \cdot\big) - \L_n\big(k, K-d\big) +\ell_{K,d}(\,\cdot\,)\in A \, \Big)
 & \leq  \epsilon \cdot G\cdot\exp \Big\{4932 \cdot d\cdot D_k^{5/2}  \big( \log \epsilon^{-1} \big)^{5/6} \Big\}   \, . \nonumber
\end{align}
%
%
Specifically, this probability is $\epsilon  \cdot \exp \big\{ (\log \epsilon^{-1})^{5/6} O_k(1) \big\}$, where $O_k(1)$ denotes a $k$-dependent term that is independent of $\epsilon$.

\end{theorem}

\begin{remark}\label{r.epsilon bounds}
The upper bound on $\epsilon$ is only a technical one and is of no real consequence. The rapid decay in $n$ of the lower bound \eqref{e.alowerbound} means that no difficulty is created in applications, since, roughly put, events whose probabilities have decay that is superpolynomial in $n$ are in practice irrelevant. In the case that $n=\infty$, such as for the parabolic Airy line ensemble, this lower bound becomes the vacuous $\epsilon>0$.
\end{remark}

\begin{remark}
The linear term $\ell_{K,d}$ introduced in the event in the general result is necessary. It arises from the parabolic curvature of regular ensembles, which cannot be ignored when the interval $[K-d,K+d]$ is far from the origin. In fact, we will prove the theorem for $K=0$, and then use the parabolic invariance introduced in Section~\ref{s.notation parabolic invariance} to get the general statement, as we have $\L_{n, K}^{\mathrm{shift}}(k, \cdot\,)-\L_{n, K}^{\mathrm{shift}}(k, -d) = \L_n\big(k, \cdot+K\big) - \L_n\big(k, K-d\big) +\ell_{K,d}(\,\cdot\,+K)$ on $[-d,d]$.
\end{remark}





\chapter{Proof framework}\label{ch.proof framework}

In this chapter we introduce the two frameworks required for our proof: the first is the \emph{jump ensemble}, a general technique introduced in \cite{hammond2017brownian} which allows one to analyse regular Brownian Gibbs line ensembles using a more explicitly Brownian proxy; while the second is specific to our proof of Theorem~\ref{t.airytail.ln} and is a conceptual framework of \emph{costs}. We will also reduce the proof of Theorem~\ref{t.airytail.ln} to a statement, Theorem~\ref{t.prob estimate for J}, about the jump ensemble, and Chapter~\ref{ch.the proof} will be devoted to providing a major part of the proof of this statement using the introduced framework of costs.

\section{The jump ensemble}\label{s.jump ensemble}

We start with a working description of the technical framework in which our proof approach operates, known as the jump ensemble. The jump ensemble should be thought of as a sort of ``half-way house'' between Brownian motion and the line ensemble $\L_n$ that we wish to study. Roughly speaking, what we mean by this is that the jump ensemble conditioned on a certain manageable event has the same distribution as $\L_n$; but since the jump ensemble can be described in terms of Brownian objects, we can estimate probabilities involving the jump ensemble using knowledge about Brownian motion. 

The construction we describe is the same as that in \cite[Chapter 4]{hammond2017brownian}. The reader is referred to that article for a fuller discussion; here we restrict ourselves to providing a complete, though perhaps sometimes not fully motivated, description of the jump ensemble that allows the reader to understand the proofs of the paper. The notation used in this section is largely the same as in \cite{hammond2017brownian} for the convenience of the reader. We stress that some of the proofs underlying the correctness and usefulness of the jump ensemble as given in \cite{hammond2017brownian} are technically involved, and so we choose to not reproduce them here, instead focusing only on illustrating the ideas and statements of the jump ensemble. 

We use only three statements from \cite{hammond2017brownian}, reproduced here as Lemma~\ref{l.corner}, Proposition~\ref{p.T_3(J) bound}, and Lemma~\ref{l.fav complement prob}. We call these three statements \emph{the side interval test}; the \emph{jump ensemble candidate proficiency}; and the \emph{high probability of the favourable event}. The reason for the use of these names will become clearer over the next few subsections.

\subsection{Motivation and main themes}\label{s.jump ensemble heuristic}
Before turning to the details of the jump ensemble, let us bring to focus some of the main themes. Recall that we aim to study the $k$\textsuperscript{th} curve of $\L_n$. To do this, we initially consider the top $k$ curves together. The basic tool we have at our disposal in studying regular line ensembles is the Brownian Gibbs property. To recall it, let $\FBB$ be the $\sigma$-algebra generated by the following collection of random variables (where the $\mathrm{BB}$ subscript stands for ``Brownian bridge''):
\begin{itemize}
	\item all the lower curves $\L_n:\llbracket k+1,n \rrbracket\times [-z_n,\infty)\to \R$;
	\item and the top $k$ curves $\L_n: \intint{k}\times ([-z_n, \ell]\cup [r, \infty))\to \R$ outside $(\ell, r)$.
\end{itemize}
(Though this $\sigma$-algebra's definition clearly depends on $k$, we suppress this dependence in the notation $\FBB$.)

The statement of the Brownian Gibbs property is then that, conditionally on $\FBB$, the top $k$ curves of $\L_n$ on $[\ell,r]$ have the same distribution as a collection of $k$ independent Brownian bridges, the $i^\text{th}$ from $\L_n(i,\ell)$ to $\L_n(i,r)$, with the curves in the collection conditioned on intersecting neither $\L_n(k+1,\cdot\,)$ nor each other on all of $[\ell, r]$.

\subsubsection{Candidate ensembles}
We interpret this description as a \emph{resampling property}, which is to say that, given the data in $\FBB$, the top $k$ curves of $\L_n$ on $[\ell, r]$ are obtained by rejection sampling collections of $k$ independent Brownian bridges with the given endpoints until they fully avoid $\L_n(k+1,\cdot\,)$ and each other on $[\ell, r]$. We call the curves' avoidance of each other on $[\ell,r]$ \emph{internal} non-intersection. 

This resampling interpretation suggests a slightly different viewpoint on the Brownian Gibbs property. Let us call the collection of $k$ independent Brownian bridges with the given endpoints a \emph{candidate} ensemble; it has forgotten all information about the lower curve, as its definition only involves $\L_n$ on $\intint{k}\times\{\ell ,r\}$. Our desire is for the candidate ensemble to gain the correct $\FBB$-conditional distribution of the top $k$ curves of $\L_n$ on $[\ell, r]$. In order for this to happen, the candidate ensemble must \emph{reconstruct} the effect of the forgotten data as well as satisfy the other constraints that $\L_n$ does. 

The basic relation between the $(k+1)^{\text{st}}$ curve of $\L_n$ and the top $k$-curves of $\L_n$ is that the top $k$ curves must not intersect the $(k+1)^{\text{st}}$; beyond this, the additional constraint that the top $k$ curves of $\L_n$ satisfy which the Brownian bridge candidate does not necessarily is of internal non-intersection. The mentioned reconstruction is done by passing a \emph{test of non-intersection}, both with the lower curve and internally. The reinterpretation of the Brownian Gibbs property is that the candidate ensemble, on passing the non-intersection test, gains the target line ensemble's $\FBB$-conditional distribution. In terms of rejection sampling, the rejection sampling probability is exactly the probability of the candidate ensemble passing this test. 

Here, the role of $\FBB$ is to specify the data which the candidate ensemble, on passing the non-intersection test, must conform to. In particular, the distribution attained by the candidate ensemble on passing the non-intersection test is the $\FBB$-conditional distribution of $\L_n$. The data in $\FBB$ should be thought of as the data conditioned on, and so available to the candidate ensemble, some of which it then forgets. The data not in $\FBB$ is, of course, not available to the candidate ensemble at all.

This idea that the candidate ensemble forgets some amount of data available to it is an important one, and one which we will develop further over the next few pages. In particular, we will consider the effects of retaining and forgetting different quantities of data; as we shall see shortly, the example here, of forgetting the entire bottom curve, is too extreme and will not be useful for our purposes.




So the broad theme may be described as follows in a two-step process. First, we condition on a certain selection of data, here represented by $\FBB$; and  second, we consider candidate ensembles which retain some subset of this data and forget the rest. The candidate ensemble recovers the correct conditional distribution, specified by the data first conditioned on, in spite of the forgotten data, by resampling till the appropriate constraint is met, which is that of non-intersection. 

\begin{remark}
The language of ``retaining data'' we are using in this discussion is slightly at odds with the usage in \cite{hammond2017brownian}. There, retained data refers to the data contained in a $\sigma$-algebra such as  $\FBB$, with respect to which the conditional distribution of $\L_n$ is considered. Here, by retained data we mean the data contained in this $\sigma$-algebra which is \emph{further} retained by the candidate ensemble, in the sense that the further retained data is involved in the specification of the candidate ensemble. Thus, by retaining different quantities of data in this sense, we can generate various candidate ensembles which, on passing the respective non-intersection tests, will each have the distribution of $\L_n$ conditionally on the same $\sigma$-algebra.
\end{remark}

\subsubsection{Features of a useful candidate}
So we see that we must consider other candidate ensembles, and the jump ensemble will be one such. What are the features of a useful candidate?

The final aim is to estimate probabilities for the $k^{\text{th}}$ curve of the line ensemble $\L_n$. So the features we need of a candidate ensemble to reach this aim is that we must be able to 
\begin{enumerate}[label=(\roman*)]\label{list.aims of candidate}
	\item estimate probabilities of interest for the candidate ensemble; and
	\item translate them to estimates on probabilities for the line ensemble. 
\end{enumerate}

To successfully estimate probabilities for the candidate ensemble, it must be amenable to the tools at our disposal, which in practice means it must be sufficiently Brownian (this is also imposed by our intention to use the Brownian Gibbs property); while to successfully translate estimates to the line ensemble, it should be intuitively clear that we need the probability of passing the non-intersection test to be not too low.

\subsubsection{The high jump difficulty}\label{s.jump ensemble.high jump}
So let us consider how the Brownian bridge candidate fares in meeting the aims (i) and (ii). Since the candidate is an ensemble of independent Brownian bridges, the point (i) from the previous paragraph is clearly easily met. But on (ii) unfortunately, because of the weak control that we have over $\L_n(k+1,\cdot\,)$ and the intricacies of that random function, it is difficult to obtain sufficiently strong control on the probability of passing the non-intersection test with the lower curve. (Roughly speaking, the Brownian bridge candidate was the one used in \cite{corwin2014brownian} to analyse the absolute continuity of Brownian Gibbs ensembles with respect to Brownian motion, and a large part of that paper was spent obtaining control over exactly this non-intersection test passing probability.)

In these terms, we do not have good control over the test passing probability of pure Brownian bridge, and so this candidate is not directly useful. This points to the need to look for better-suited candidate processes. To understand how a better candidate process should be designed, let us consider what made the Brownian bridge candidate have a low test passing probability.

\begin{figure}[h]
\centering{\epsfig{file=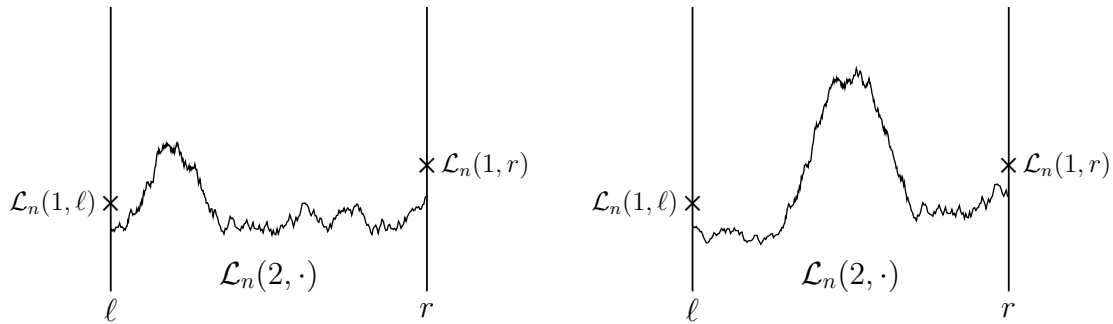,width=0.9\textwidth}}
\caption{Two illustrations in the $k=1$ case of instances of data from the lower curve which are difficult for the Brownian bridge candidate to handle. The black crosses indicate the values of $\L_n(1,\cdot\,)$ at the endpoints, which are the points between which the Brownian bridge candidate must move. In the left panel there is a moderate sized peak very close to the left side of $[\ell, r]$, which causes difficulty because of the immediacy of the jump required by the candidate. In the right panel there is a large peak which causes difficulty because of its height.}\label{f.high jump}
\end{figure}

 In essence, the Brownian bridge candidate ensemble forgot too much of the data in $\FBB$, and is thus too far in nature from $\L_n$, to have a high probability of lower curve non-intersection. In particular, it forgot \emph{all} data of the profile of $\L_n(k+1,\cdot\,)$ that it might have used to increase its probability of avoidance. We shall consider two instances of the lower curve data which is difficult for the Brownian bridge to avoid in order to illustrate two features that our replacement candidate will need. These instances are depicted in Figure~\ref{f.high jump}.

 For the first instance, suppose that the lower curve has a peak inside $[\ell, r]$ which is close to one side of the interval, say $\ell$, as illustrated in the first panel of Figure~\ref{f.high jump}. Then the Brownian bridge candidate, to succeed in the non-intersection test, must execute a jump immediately. The difficulty is that since the space to make the jump is limited, a more extreme jump is needed even if the peak is not very large. The low probability of Brownian bridges making such a jump in turn makes the non-intersection test passing probability of the Brownian bridge candidate ensemble low. This discussion suggests the first feature that will aid a successful candidate: we can provide it extra space to make a \emph{run-up} before any required jump.

Turning to the second panel of Figure~\ref{f.high jump}, the second instance of difficult lower curve data is when it exhibits an extremely large peak somewhere inside $[\ell ,r]$ (which is not necessarily close to either side). It may seem that giving space for a run-up would address this difficulty as well, as a Brownian bridge is clearly more likely to make a bigger jump over a larger interval. However, giving a run-up is in fact not sufficient to handle this sort of data while maintaining a not-too-low non-intersection probability; quantitative reasoning for this conclusion is explored more fully in the beginning of \cite[Chapter 4]{hammond2017brownian} and also briefly, in the context of the jump ensemble, in Remark~\ref{r.run-up not sufficient} ahead. (The discussion in \cite{hammond2017brownian} concerns a setup incorporating a run-up which we will introduce shortly in Subsection~\ref{s.jump ensemble.making space}.) 

Heuristically, the reason for the difficulty of this data is that the Brownian bridge, having forgotten the entirety of the lower curve, does not know when the jump is required. And so, as alluded to earlier, the second feature of assistance a successful candidate should make use of is to retain more information about the lower curve. More formally, by ``using retaining data", we mean that the candidate ensemble will be conditioned to avoid intersection with a curve formed from the retained data (apart from retaining data to specify the values of the candidate ensemble at the endpoints). This will become clearer as our discussion progresses.

(One might wonder about the likelihood of encountering the sort of lower curve data we have been discussing, and whether we cannot exclude such difficult data from $\FBB$ in the analysis. In our final argument we will indeed restrict ourselves to data in a $\sigma$-algebra analogous to $\FBB$ which is favourable and extremely likely. However, even under such a restriction to more favourable data, our control on the lower curve is not strong enough to exclude data such as what has been discussed.)

Next we move to discussing in more detail the two changes we have mentioned: retaining more data and giving a run-up.

\subsubsection{A coarsened lower curve profile} 
We first discuss the second feature we mentioned, namely retaining a selection of data from the lower curve profile. It should be clear that we should not retain all the data, as this would result in the candidate essentially being the same as the top $k$ curves of $\L_n$ itself; it is difficult to estimate the probabilities of such an ensemble. So we must make a careful selection which balances between retaining no information, as in the Brownian bridge ensemble, and retaining full information, as in the pure line ensemble; further, the retained information must provide a rough view of the overall geometry of the lower curve.

In fact, we will have the candidate ensemble retain a \emph{coarsened} version of the lower curve. More precisely, the candidate will be conditioned on avoiding this coarsened version. This coarsened non-avoidance can be thought of as a preliminary test to the full non-intersection test; the candidate, on passing the preliminary test, will naturally have a more suitable overall geometry to pass the final test, and thus will have a higher probability of doing so. The exact form of this coarsening, which we will describe in Section~\ref{s.jump ensemble.coarsening}, is at the heart of the jump ensemble method.



\subsubsection{Making space} \label{s.jump ensemble.making space}
Now let us turn to see how we can provide the first kind of assistance, namely to provide the candidate ensemble with space to make a run-up to more successfully jump over the lower curve. The only way to make space is to step back from the interval $[\ell, r]$. In fact, we will work in an interval $[-2T,2T]$ which contains $[\ell, r]$, with the parameter $T$'s value to be assigned later. Let us label as \emph{side intervals} the intervals $[-2T, \ell]$ and $[r, 2T]$, and as the \emph{middle interval} the interval $[\ell, r]$. 

Working in $[-2T,2T]$ means that the values of the candidate are not pre-determined at $\ell$ and $r$, as in Figure~\ref{f.high jump}, but at $\pm2T$. Of course, simply working on a bigger interval does not gain us anything immediately, since, in our current setup of conditioning on $\FBB$, the non-intersection must now be done on the larger interval.

To deal with this, we change the setup by changing the data we condition on. Instead of conditioning on $\FBB$, we consider the $\sigma$-algebra $\F$ generated by the following collection of random variables:
\begin{itemize}\label{i.F defitinion}
\item all the lower curves $\L_n : \llbracket k+1, n\rrbracket \times [-z_n,\infty)\to \R$;

\item the top $k$ curves $\L_n: \intint{k}\times\left([-z_n, -2T]\cup [2T,\infty)\right)\to \R$ outside $(-2T,2T)$;

\item and the $2k$ standard bridges $\L_n^{[-2T,\ell]}(i,\cdot\,)$ and $\L_n^{[r, 2T]}(i,\cdot\,)$ for $i=1, \ldots, k$.
\end{itemize}

(Recall here the notation $f^{[a,b]}$ introduced in Section~\ref{s.notation general} for the affinely shifted bridge version of a function $f$, though mildly abused here to refer to the bridge version of the $i^\text{th}$ curve of the ensemble and not the $i^\text{th}$ curve of an undefined bridge version of the ensemble.) We again suppress the $k$ dependence of the $\sigma$-algebra in the notation $\F$.

In words, we retain data of the entirety of all the lower curves; the top $k$ curves outside $(-2T,2T)$; and, on $[-2T,\ell]$ and $[r,2T]$, the standard bridge paths of the top $k$ curves on these intervals, which we will call the \emph{side bridges}. Nothing is retained on $[\ell, r]$, and, in particular, the values of the candidate ensemble at $\ell$ and $r$ are not determined. 

\begin{remark}\label{r.bridges not complicated}
The side bridges may appear to be complicated objects to condition upon; in fact, they are easy to handle because of the Brownian Gibbs property and an independence property possessed by the corresponding side bridge decomposition of Brownian bridges. See Lemma~\ref{l.bb independent decomposition property} ahead.
\end{remark}

The $\sigma$-algebra $\F$ and the selection of data included in it is of great importance for the jump ensemble method, and will be used throughout the arguments of Theorem~\ref{t.airytail.ln}. As such, the conditional law $\P(\,\cdot\mid \F)$ will be used extensively, and so we use the notation 
$$\PF(\,\cdot\,) := \P(\,\cdot \mid \F)$$
to denote it. In this notation, our aim is to understand the law of the top $k$ curves of $\L_n$ on $[-2T,2T]$ under~$\PF$.

Why does conditioning on $\F$ help? Our reasoning was that lower curve avoidance on $[\ell, r]$ without a run-up is difficult, and so we need to give a run-up. This was done by expanding the interval to $[-2T,2T]$. However, we then need to enforce lower curve avoidance on all of $[-2T,2T]$, which is more difficult. But by including the side bridges of the top $k$ curves of $\L_n$, we can use that data to help make the non-intersection easier on the side intervals. This is because the geometries of the top $k$ side bridges of $\L_n$ are \emph{already} well suited for lower curve avoidance with $\L_n(k+1,\cdot\,)$, and the candidate ensemble can piggyback on this success. Thus, we get almost the best of both worlds: the lower curve avoidance is made easier in the middle interval of $[\ell, r]$ due to the space for a run-up, while at the same time the lower curve avoidance on the side intervals is manageable using the data of the top $k$ side bridges of $\L_n$.


How do we make use of this data? We will combine the candidate ensemble on $[\ell, r]$ with the data from $\F$ to create a new ensemble on $[-2T,2T]$. The form of this combination is dictated by the Brownian Gibbs property and the linear operation involved in the definition of $f^{[a,b]}$. 

Let the candidate ensemble be denoted by $X:\intint{k}\times[\ell, r]\to \R$; the new ensemble created using $X$ and data from $\F$ will be called the \emph{resampled} ensemble $\L^{\mathrm{re},X} : \intint{k}\times[-2T,2T]\to\R$. Intuitively, the values of the candidate ensemble at $\ell$ and $r$ are used to affinely shift the side bridges; the affinely shifted bridges define the resampled ensemble on $[-2T,\ell]\cup[r, 2T]$, while the candidate ensemble determines the resampled ensemble on $[\ell, r]$. This is illustrated in Figure~\ref{f.mcm-reconstruction}, and the formal definition of $\L^{\mathrm{re}, X}$ is given by the following, for $i=1, \ldots, k$:
\begin{equation}\label{e.reconstruction definition}
\L^{\mathrm{re},X}(i,x) = \begin{cases}
\L_n^{[-2T,\ell]}(i,x)+ \frac{x+2T}{\ell + 2T}\cdot X(i,\ell) + \frac{\ell-x}{\ell + 2T}\cdot \L_n(i,-2T)  & x\in[-2T, \ell]\\
\L_n^{[r, 2T]}(i,x)+ \frac{2T-x}{2T- r}\cdot X(i, r) + \frac{x- r}{2T - r}\cdot \L_n(i,2T)  & x\in[r, 2T]\\
X(i,x) & x\in [\ell, r].
\end{cases}
\end{equation}

\begin{figure}[h]
\centering {\epsfig{file=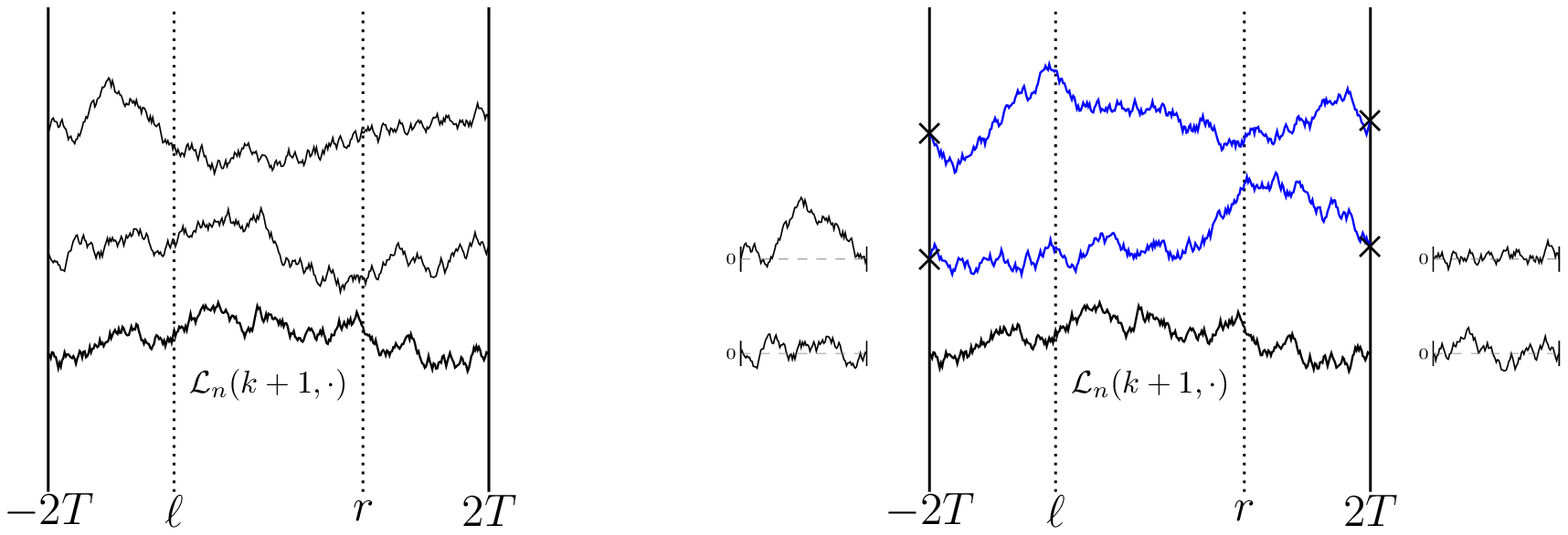, width=0.95\textwidth}\vspace*{1cm}
\epsfig{file=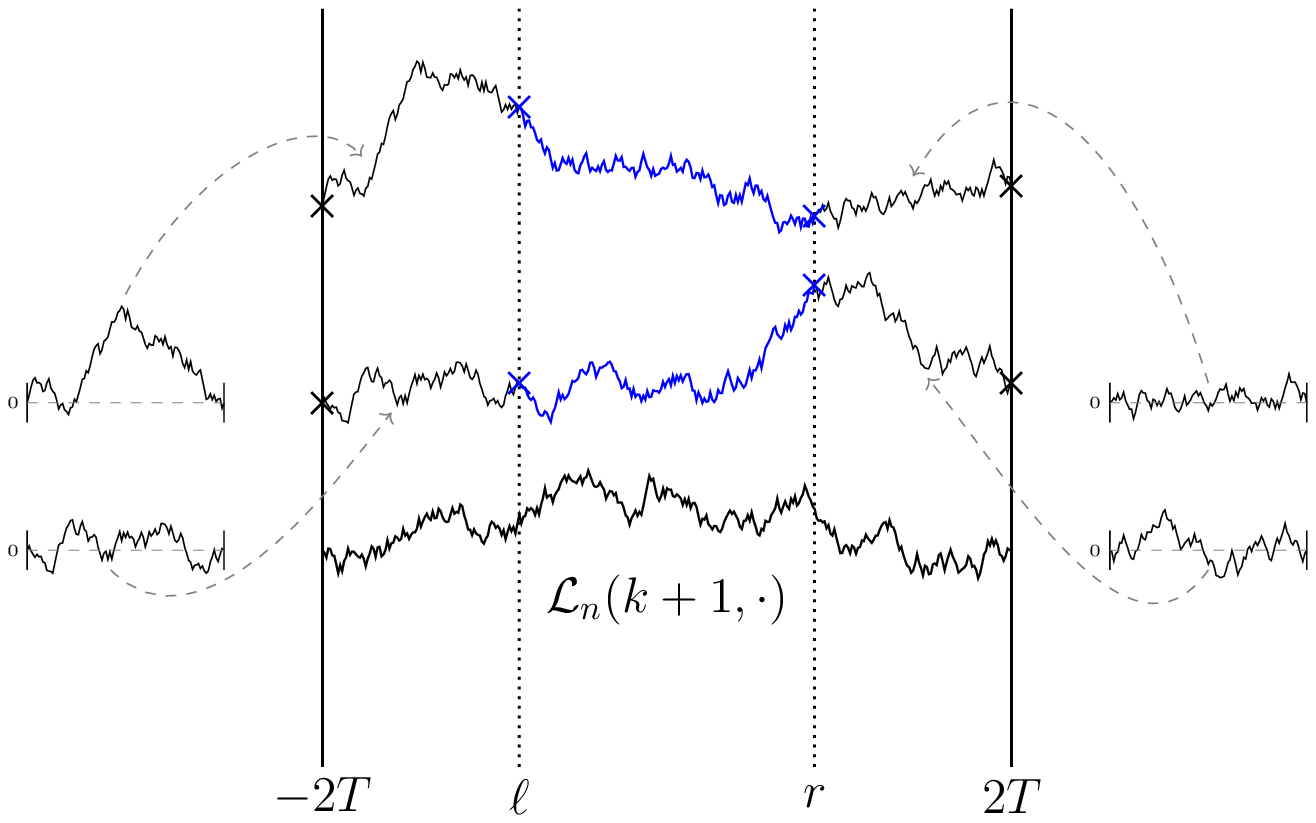, width = 0.75\textwidth}}
\caption{
Constructing $\L^{\mathrm{re}, X}$ from the candidate process $X$ (top two curves in blue on $[-2T,2T]$ in the second figure, and on $[\ell,r]$ in the third) when $k=2$. In the first figure we have the original line ensemble $\mathcal L_n$. In the second figure, the black elements are the data available in $\mathcal F$, namely the entirety of the $(k+1)^\text{st}$ curve, the positions at $\pm 2T$ of the first $k$ curves (denoted by crosses), and the bridges obtained from the side intervals by affine shift. The blue curves comprise the candidate process $X$ (though technically $X$ is restricted to $[\ell, r]$). In the final figure we complete the reconstruction by pasting the side interval bridges according to the positions dictated by $X$ (blue crosses at $\ell$ and $r$) on $[\ell, r]$. Note that in this figure, $\L^{\mathrm{re},X}$ passes both the side interval tests and the middle interval test.}
\label{f.mcm-reconstruction}
\end{figure}

(Implicit in the above discussion is the promise that the resampled ensemble $\L^{\mathrm{re}, X}$ will be able to pass the non-intersection test on the side intervals sufficiently well due to the use of data in $\F$. We discuss and make good on this promise in Section~\ref{s.jump ensemble side interval}.)

As with the earlier discussion of the Brownian bridge ensemble and $\FBB$, the Brownian Gibbs property says that for certain candidate ensembles $X$, the distribution of $\L^{\mathrm{re},X}$, conditioned on passing the non-intersection tests on $[-2T,2T]$, will be the $\F$-conditional law of $\L_n$ on $\intint{k}\times[-2T,2T]$. The candidate ensembles $X$ for which this is true are Brownian bridge ensembles conditioned on avoiding the lower curve on some subset of $[\ell, r]$; the jump ensemble, which will be conditioned on avoiding a coarsened version of the lower curve, will fit this description. Section~\ref{s.jump ensemble.brownian gibbs} is devoted to setting up a precise version of this statement, which is recorded in Lemma~\ref{l.jump ensemble brownian gibbs}.

In summary, we are looking to define a candidate process which has estimable probabilities by virtue of being in some sense Brownian, and which has a not-too-low probability of passing the non-intersection test. To accomplish this, we saw in this subsection and the previous that the candidate process will make use of a coarsened version of the lower curve profile in its definition; and will use extra space for a run-up, for which we work with a more sophisticated selection of data captured by $\F$. This data will be combined with the candidate ensemble to give the resampled ensemble. 

In the next subsection we expand on the idea that including the data of the side bridges in $\F$ makes it easy for the candidate ensemble to pass the non-intersection test on the side intervals.

\subsection{The side intervals test}\label{s.jump ensemble side interval}
We formulate the non-intersection test on the side intervals as \emph{the side intervals test}. The side intervals test has two parts: that $\L^{\mathrm{re},X}(k, \cdot\,)$ does not intersect $\L_n(k+1,\cdot\,)$; and that $\L^{\mathrm{re},X}(i,\cdot\,)$ does not intersect $\L^{\mathrm{re}, X}(i+1,\cdot\,)$ for $i=1, \ldots, k-1$---both of these on $[-2T,\ell] \cup [r, 2T]$. A look at the first two cases in the definition \eqref{e.reconstruction definition} of $\L^{\mathrm{re}, X}$ suggests that whether this test is passed is simply a question of whether $X(i, \ell)$ and $X(i, r)$ are high enough in value, as the remaining quantities are $\F$-measurable and thus not affected by the candidate $X$. This intuition is roughly correct, and a precise version is the content of the next lemma, which we refer to as \emph{the side intervals test criterion}.

\begin{lemma}[Side intervals test criterion, Lemma 3.8 of \cite{hammond2017brownian}]\label{l.corner}
There exist $\F$-measurable random vectors $\gbarxminstd, \gbaryminstd \in \R^k_\geq$ such that $\L^{\mathrm{re}, X}$ passes the side intervals tests if and only if $\bar X(x) - \gbargenmin \in (0,\infty)^k_{>}$ for $x=\ell, r$. 
\end{lemma}

\begin{figure}[h]
\centering {\epsfig{file=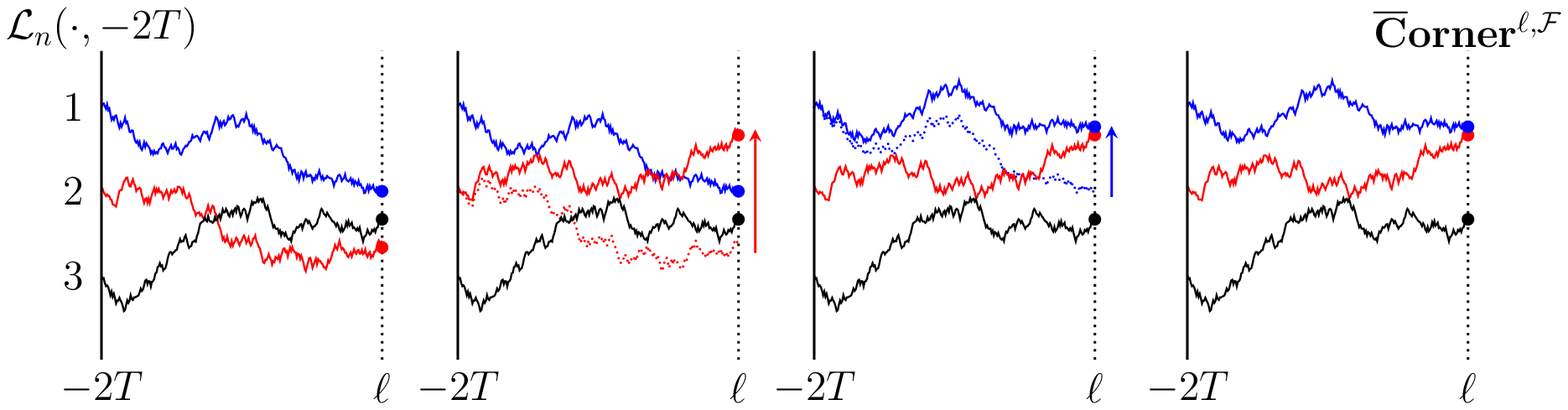, width=0.95\textwidth}}
\caption{Building $\gbarxminstd$ for $k=3$. $\gbarxminstd$ is a vector, the $(i-1)^{\text{st}}$ entry of which is the unique smallest value which $\L_n(i-1,\ell)$ can adopt via affine translation before crossing $\L_n(i,\cdot\,)$ on $[-2T, \ell]$. The line $\L(2,\cdot\,)$ (red) intersects $\L_n(3,\cdot\,)$ (black) on $[-2T, \ell]$ (first panel), so $\L_n(2,\cdot\,)$ (dotted red) is affinely translated until just touching, but not crossing $\L_n(3,\cdot\,)$ (second panel); the translated value of $\L_n(2, \ell)$ is $\gxminstd_2$. This process is then repeated for the line $\L_n(1,\cdot\,)$ (blue). Then $\L_n(1,\cdot\,)$ (dotted blue) intersects the translated curve $\L_n(2,\cdot\,)$ (red), so $\L_n(1,\cdot\,)$ is affinely translated until just touching, but not crossing $\L_n(2,\cdot\,)$ (third panel); the translated value of $\L_n(1,\ell)$ is $\gxminstd_1$. The result of these translations is shown in the fourth panel. 
This procedure depends on the collection of bridges on $[-2T,\ell]$, justifying the dependence of $\gbarxminstd$ on $\F$. }\label{f.corner}
\end{figure}

The proof is given in \cite{hammond2017brownian}, but we include here Figure \ref{f.corner} which captures the essential argument.

The conclusion we draw from Lemma~\ref{l.corner} is that analysing the passing of the side interval tests by the candidate $X$ is very simple in practice: we merely need to consider the event that  $\bar X(\ell)- \gbarxminstd \in (0,\infty)^k_>$ and $\bar X(r) - \gbaryminstd\in(0,\infty)^k_>$.

Next we give a precise description of candidate ensembles $X$ which are such that $\L^{\mathrm{re},X}$, conditioned on the non-intersection tests, has the $\F$-conditional distribution of the top $k$ curves of $\L_n$.

\subsection{Applying Brownian Gibbs to $\L^{\mathrm{re},X}$}\label{s.jump ensemble.brownian gibbs}

The resampled ensemble $\L^{\mathrm{re}, X}$ passing the non-intersection tests on $[-2T,2T]$ requires
\begin{itemize}
	\item $\L^{\mathrm{re},X}(i, x) > \L^{\mathrm{re},X}(i+1, x)$ for $i=1, \ldots, k-1$ and $x\in[-2T,2T]$; and
	\item $\L^{\mathrm{re},X}(k, x) > \L_n(k+1, x)$ for $x\in[-2T,2T]$.
\end{itemize}

We denote by $\pass(X)$ the indicator for the event described by these two bullet points. (In \cite{hammond2017brownian}, an analogous indicator obtained by restricting these two bullet points to $x\in [\ell ,r]$ is denoted $T_3(X)$, where 3 represents the test of non-intersection on the middle interval being the third in a sequence of tests.) 

Now we may describe a class of candidate ensembles which, conditioned on $\{\pass(X) = 1\}$, have the desired $\F$-conditional distribution.
Let $A\subseteq [\ell, r]$ be an $\F$-measurable random closed set. Define the candidate ensemble $X:\intint{k}\times[\ell, r]\to \R$ as a collection of $k$ independent Brownian bridges, the $i^\text{th}$ one from $(-2T,\L_n(i,-2T))$ to $(2T,\L_n(i,2T))$, conditioned on $\notouch^A_{\L_n(k+1,\cdot)}$. (Recall that $\notouch_f^A$ is the event that the bottom curve $X(k,x)$ is larger than $f(x)$ for all $x\in A$.) Also define $X':\intint{k}\times[\ell, r]\to \R$ in the same way, with the additional conditioning that $\bar X'(x) - \gbargenmin \in (0,\infty)^k_>$ for $x\in\{\ell ,r\}$; we introduce this variant candidate ensemble as it is the form that the jump ensemble will take.

Both $X$ and $X'$ have the desired $\F$-conditional distribution on passing the non-intersection tests. To prove this, we first need a fact about decompositions of Brownian bridges. This property of Brownian bridges also explains why there is no difficulty in conditioning on the potentially complicated objects, the side bridges of $\L_n$, as mentioned in Remark~\ref{r.bridges not complicated}. The proof of this fact is a straightforward checking of covariances and is omitted.

\begin{lemma}\label{l.bb independent decomposition property}
Let $T>0$ and $x_1, \ldots, x_m \in[-2T,2T]$ with $x_1 <  \ldots <x_m$, for some $m\in \N$. Let $x_0 = -2T$ and $x_{m+1} = 2T$. Let $B$ be a Brownian bridge (with arbitrary fixed starting and ending point values) on $[-2T,2T]$.  Then, conditionally on $\big(B(x_1), \ldots,  B(x_m)\big)$, the distribution of $\big(B^{[x_i, x_{i+1}]}\big)_{i=0}^m$ is that of $m+1$ independent Brownian bridges, with the $i^{\text{th}}$ one of duration $x_i-x_{i-1}$.
\end{lemma}

\begin{lemma}\label{l.jump ensemble brownian gibbs}
For $X$ as defined above, conditionally on $\F$, the following two laws on $\mc C_{*,*}([-2T, 2T], \R)^k$ are equal:
$$\PF\bigl(\L^{\mathrm{re}, X} \in \cdot \,\mid \pass(X)=1\bigr) \quad \text{and} \quad \PF\bigl(\L_n \in \cdot\,\bigr).$$
The same holds with $X'$ in place of $X$.
\end{lemma}

\begin{proof}
Let $B:\intint{k}\times[\ell, r]\to \R$ be the restriction to $[\ell, r]$ of a collection of $k$ independent Brownian bridges on $[-2T,2T]$, with the $i^{\text{th}}$ having starting and ending points $(-2T, \L_n(i,-2T))$ and $(2T,\L_n(i,2T))$. Lemma~\ref{l.bb independent decomposition property}, combined with the Brownian Gibbs property possessed by $\L_n$, implies that the $\F$-conditional distribution of $\L_n$ restricted to the top $k$ curves on $[\ell, r]$ is that of $\L^{\mathrm{re}, B}$ conditioned on the event $\{\pass(B)=1\}$. It is immediate that this latter distribution is the same as that of $\L^{\mathrm{re}, X}$ conditioned on $\{\pass(X) = 1\}$, as the law of $X$ is just that of  $B$ with an additional conditioning that is consistent with $\{\pass(B)=1\}$. That is, $\{\pass(B)=1\}$ is a subset of $\notouch^A_{\L_n(k+1,\cdot)}$, and the distribution of $B$ conditioned on $\notouch^A_{\L_n(k+1,\cdot\,)}$ is that of $X$.

Since the law of $X'$ is that of $X$ conditioned on passing the side interval tests, and since conditioning on $\{\pass(X)=1\}$ is a stronger one than conditioning on $X$ passing the side intervals test, i.e., the former event is contained in the latter, the argument of the previous paragraph holds for the candidate ensemble $X'$ as well.
\end{proof}

Having completed the general set up and groundwork of candidate ensembles, we may now turn to describing the jump ensemble itself.

\subsection{Parameters of the jump ensemble}
 We start with two parameters, $k\in \N$ and $\epsilon>0$. The first is simply the number of curves of $\L_n$ that we are studying, which will also be the number of curves in the jump ensemble. The second is to be understood as the Brownian probability of the event that we wish to analyse under the law of $\L_n$, but is formally simply a positive parameter. The logic of the jump ensemble is to set the parameters according to the event we wish to study. 

Though in the discussion in the preceding subsections we were working with a deterministic interval $[\ell, r]$, for the jump ensemble we will in fact need to work on a particular random subinterval $[\mf l, \mf r]$ of $[-2T, 2T]$ which will be defined shortly. All the arguments and statements of the previous subsections of Section~\ref{s.jump ensemble} will hold true with $\mf l$ and $\mf r$ in place of $\ell$ and $r$, as can be easily checked, since $\mf l$ and $\mf r$ will be defined in terms of only the lower curve $\L_n(k+1,\cdot\,)$; this data is present in $\F$, and so, conditional on $\F$, $\mf l$ and $\mf r$ can safely be thought of as being deterministic.

In Section~\ref{s.jump ensemble heuristic} we introduced a parameter $T$. For the jump ensemble, the value of $T$ is determined by both parameters $k$ and $\epsilon$, and is given by
\begin{equation}\label{e.T value}
T := D_k(\log \epsilon^{-1})^{1/3},
\end{equation}
where $D_k$ is given by \eqref{e.D_k value}. The value of $T$ given in \eqref{e.T value} will be its fixed value for the remainder of this article.

\begin{remark}\label{r.choice of T}
Though not needed for our arguments, here is a heuristic idea of why this is the form of $T$ we select; a fuller discussion is available in the beginning of \cite[Chapter 4]{hammond2017brownian}. We see that a larger value of $T$ gives more space for a run-up, which helps the candidate pass the non-intersection test on the middle interval. However, as $T$ gets larger we must also grapple with the globally parabolic curvature of $\L_n$, which means that the starting and ending points of the candidate ensemble will fall quadratically. Simply put, there is more space for the run-up before the jump, but the required jump is higher as the starting point is lower. This selection of $T$---in particular the 1/3 exponent of $\log\epsilon^{-1}$---balances these opposing forces and in some sense maximises the non-intersection test probability of the to-be-defined jump ensemble.
\end{remark}

\begin{remark}\label{r.run-up not sufficient}
This same reasoning of balancing these opposing forces of curvature and run-up advantage is what shows that a run-up is not sufficient to handle the second instance of lower curve data discussed in Subsection~\ref{s.jump ensemble.high jump}, as explained in the beginning of \cite[Chapter 4]{hammond2017brownian}. Indeed, what that discussion essentially shows is that even with the well-tuned choice of $T$ made above, there exists data in $\F$ with sufficiently high probability for which the Brownian bridge candidate ensemble (which has forgotten all data about the lower curve) is unable to pass the non-intersection test with sufficiently high probability.
\end{remark}

We now record a certain upper and lower bound that the parameter $\epsilon$ is required to meet for technical reasons; these constraints also previously appeared in the statement of Theorem~\ref{t.airytail.ln}.
\begin{equation}\label{e.epsilon constraint}
\begin{split}
\epsilon &< e^{-1} \wedge (17)^{-1/k} C_k^{-1/k} \const^{-1} \wedge \exp\Big\{-(24)^{6}d^6/D_k^3\Big\} \quad \text{and}\\
\epsilon &> \exp \Big\{- \big( \rsc/2 \wedge 2^{1/2} \big) \const^{-1} n^{1/12} \Big\}.
\end{split}
 \end{equation}
As we noted in Remark~\ref{r.epsilon bounds}, both these bounds do not cause any difficulties in practice. And, in the case that $n=\infty$, the lower bound becomes simply $\epsilon>0$.

With these definitions, we may start making precise the notion introduced earlier of a coarsened version of the underlying curve.

\subsection{Coarsening the lower curve}\label{s.jump ensemble.coarsening}
Let $\mf c_+ : [-T,T]\to \R$ be the least concave majorant of $\L_n(k+1,\cdot\,):[-T,T]\to \R$, and define a random interval $[\mf l, \mf r]$ by
\begin{align*}
\mf l &= \inf\left\{x\in[-T,T] \mid \mf c_+'(x)\leq 4T\right\} \quad \text{and}\\
\mf r &= \sup\left\{x\in[-T,T] \mid \mf c_+'(x)\geq -4T\right\}.
\end{align*}
We can think of $\mf c_+$ as a first coarsening of the lower curve $\L_n(k+1,\cdot\,)$. As indicated earlier, the interval $[\mf l, \mf r]$ will play the role of $[\ell, r]$ in Section~\ref{s.jump ensemble heuristic}. Note that though random, $\mf l$ and $\mf r$ are functions of the curve $\L_n(k+1,\cdot\,)$ and not of the $k$ curves we are attempting to study. The purpose of defining this random interval is that on it we are guaranteed some control over the coarsened lower curve profile (which we will be further coarsening before using in the definition of the jump ensemble); we will not use any data about $\L_n(k+1,\cdot\,)$ outside of $[\mf l, \mf r]$ in defining the jump ensemble, though it is available.

By the concavity of $\mf c_+$, it follows that $\mf l\leq\mf r$, and in fact on a high probability favourable event we will discuss in Section~\ref{s.jump ensemble fav}, $[\mf l,\mf r]$ will be an interval of length at least $T$. We use the $\sigma$-algebra $\F$ defined in Section~\ref{s.jump ensemble heuristic} on page~\pageref{i.F defitinion}, except with $\mf l$ and $\mf r$ in the place of $\ell$ and $r$. Though $\mf l$ and $\mf r$ are random, the definition of $\F$ is adequate as $\mf l$ and $\mf r$ are determined by the lower curve data, which is already present in its entirety in $\F$.

We came to the conclusion in the heuristic discussion in Section~\ref{s.jump ensemble heuristic} that the candidate process $X$ would need to use information about a coarsened version of the lower curve $\L_n(k+1,\cdot\,)$ in its definition, in order to have a high enough probability of passing the middle interval non-intersection test. So far we have defined a preliminary coarsening, the least concave majorant $\mf c_+$, which was used to define the interval $[\mf l, \mf r]$. Now we define a further (and final) coarsening, which we will then use in the next subsection to finally define the jump ensemble $J$.

To precisely describe the final coarsening of $\L_n(k+1,\cdot\,)$ on $[\mf l,\mf r]$, we first define a subset of extreme points of $\mf c_+$. Let $\mrm{xExt}(\mf c_+)$ be the $x$-coordinates of the set of  extreme (or corner) points of the convex set $\{(x,y) : \mf l\leq x\leq \mf r, y\leq \mf c_+(x)\}$. Note that necessarily $\mf l,\mf r\in \mrm{xExt}(\mf c_+)$. Then define the \emph{pole set} $P$ to be a subset of $\mrm{xExt}(\mf c_+)$ such that
\begin{itemize}
	\item $\mf l, \mf r\in P$,
	\item $p_1, p_2\in P, p_1\neq p_2 \implies |p_1-p_2| \geq \dip$, and
	\item if $x\in\mrm{xExt}(\mf c_+)$, then some element $p\in P$ satisfies $|p-x|\leq \dip$.
\end{itemize}
Here $\dip \in [1, \mf r-\mf l]$ is a parameter called the \emph{inter-pole distance}; typically it is set to a unit order quantity independent of $\epsilon$ and $k$, and usually it is comparable to the interval of interest under study. For example, in the proof of our results, we will set it to be a multiple of $d$. The parameter $\dip$ defines the minimum separation between consecutive elements of $P$. These elements of the pole set $P$ will be called \emph{poles.}  These above three properties do not necessarily define $P$ uniquely, and to address this we take $P$ to be the subset satisfying these conditions of maximal cardinality, and then maximal in lexicographic order. 

\begin{remark}
By the definition of $\dip$, it is clear that the size of the pole set, $|P|$, is at most $2T/\dip$. So for constant order values of $\dip$, $|P|$ is potentially rather large. However, in arguments the only poles which must be considered are essentially the ones within or adjacent to the interval under study. In the proof of our main result our arguments will be focused on the single pole contained in the interval $[-2d, 2d]$ (if it is present), and will in one instance make use of the preceding and succeeding poles. We guarantee ourselves this control on the number of poles in $[-d,d]$ by making an appropriate choice of $\dip$.
\end{remark}

The coarsened profile of $\L_n(k+1,\cdot\,)$ that will be used in defining the jump ensemble $J$ is exactly the set $\{(p, \L_n(k+1, p) \mid p\in P)\}$, which $J$ will be conditioned to \emph{jump} over. We next make precise what we mean by $J$ jumping over the poles, and also give the definition of $J$.

\subsection{Defining the jump ensemble}

Conditional on $\F$, let $B:\intint{k}\times[-2T,2T]\to \R$, with $\{B(i,\cdot\,)\}_{i=1}^k$ a collection of $k$ independent Brownian bridges on $[-2T,2T]$ and $B(i,\cdot\,)$ having endpoints $(-2T, \L_n(i,-2T))$ and $(2T,\L_n(i,2T)$ for $i=1, \ldots, k$. Note that the required information about the endpoint values of $\L_n$ is present in $\F$. The jump ensemble $J:\intint{k}\times[\mf l, \mf r]\to\R$ is the restriction to $[\mf l, \mf r]$ of $B$ conditioned on 
\begin{enumerate}[label=(\roman*)]
	\item $\overline B(x) - \gbargenmin \in (0,\infty)_>^k$ for $x\in\{\mf l, \mf r\}$; and
	\item $B(i,p) \geq \L_n(k+1, p)$ for all $p\in P$ and $i=1, \ldots, k$.
\end{enumerate}
As we saw in Lemma~\ref{l.corner}, the conditioning present in point (i) ensures that $J$ passes the side interval tests.
We will refer to the event in point (ii), namely $J(p)\geq \L_n(k+1,p)$ for $p\in P$, as \emph{jumping over the pole $p$}. 

\begin{figure}
\centering {\epsfig{file=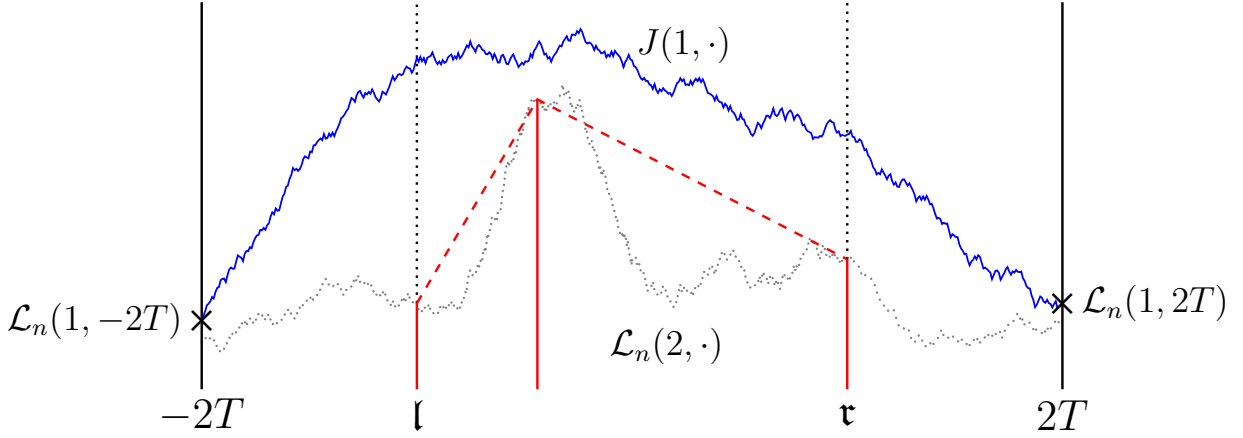, width=\textwidth}}
\caption{The jump ensemble candidate $J$ (in blue) when $k=1$. The lines in red are the \emph{poles}, i.e., the elements of the pole set $P$, which $J$ is conditioned to jump over (recall that necessarily $\mf l,\mf r\in P$). The dashed red piecewise linear function is the $\tent$ map. The lower curve is drawn dotted and in light gray to indicate that $J$ does not have access to this data, only the heights of the solid poles. Making this jump forces the candidate to avoid a coarsened version of the lower curve profile, which makes it more likely that it also avoids the full lower curve, as in this figure. Moving back to $[-2T, 2T]$ from $[\mf l, \mf r]$ gives the candidate space to make the jump. Though we have shown the blue curve on the entirety of $[-2T,2T]$, $J$ is only the restriction to $[\mf l, \mf r]$, and so does \emph{not} need to avoid the lower curve outside $[\mf l, \mf r]$ (though it does in the illustrated instance). The avoidance on the full interval $[-2T,2T]$ is a requirement imposed only on $\L^{\mathrm{re}, J}$, which is made by combining $J$ with the data in $\F$ as in Figure~\ref{f.mcm-reconstruction}.}\label{f.jump ensemble}
\end{figure}

As we noted when we stated its definition, the pole set $P$ represents the coarsened version of the lower curve $\L_n(k+1,\cdot\,)$ that $J$ has access to. By conditioning $J$ to avoid this coarsened version of $\L_n(k+1,\cdot\,)$, we increase the probability of $J$ successfully avoiding all of $\L_n(k+1,\cdot\,)$ on $[\mf l, \mf r]$, compared (heuristically) to a candidate process with no information about the underlying curve. Indeed we will see ahead in Proposition~\ref{p.T_3(J) bound} that this increased probability is high enough to be useful for our intended application.

It will be necessary in our arguments to consider how much $J$ deviates from the shape defined by the poles. To do this, define the $\F$-measurable random piecewise affine function $\tent:[\mf l,\mf r]\to \R$ which linearly interpolates between the points $(p, \L_n(k+1, p))$ for $p\in P$. Note that from the definition of $\mf c_+$ we have that $\tent$ is concave, and from the definition of $\mf l$ and $\mf r$ we have that the slope of every linear segment of $\tent$ lies in $[-4T, 4T]$, which for future reference we will express (with abuse of notation, as $\tent$ is only piecewise linear and not linear) as
\begin{equation}\label{e.tent map slope}
\mrm{slope}(\tent) \in [-4T,4T].
\end{equation}
See Figure~\ref{f.jump ensemble} for an illustration of the jump ensemble and the Tent map.

We also note here that Lemma~\ref{l.jump ensemble brownian gibbs} implies that $\L^{\mathrm{re}, J}$, conditionally on $\F$ and $\{\pass(J) = 1\}$, has the $\F$-conditional distribution of the top $k$ curves of $\L_n$ on $[-2T,2T]$. Since $\L^{\mathrm{re},J} = J$ on $\intint{k}\times[\mf l, \mf r]$, this implies that the distribution of $J$, conditionally on $\F$ and $\{\pass(J) =1\}$, is that of the top $k$ curves of $\L_n$ on $[\mf l,\mf r]$.

\subsection{The probability that $J$ passes the non-intersection test}\label{s.jump ensemble non-intersection test pass}
Now we shall address whether the jump ensemble is in fact able to pass the non-intersection test on the whole interval $[-2T,2T]$ with sufficiently high probability, the task for which we specifically defined the coarsened version of $\L_n(k+1,\cdot\,)$ that $J$ is conditioned to jump over. Recall from Section~\ref{s.jump ensemble side interval} that we have an easy criterion for $J$ passing the non-intersection test on the side intervals of $[-2T, \mf l]$ and $[\mf r, 2T]$, which we called the side-intervals test. This criterion is that $(J(i,x) - \mrm{Corner}_i^{x,\F})_{i=1}^k\in (0,\infty)^k_>$ for $x\in\{\mf l, \mf r\}$, which $J$ is in fact conditioned to satisfy above in point (i) of its definition. Thus all that remains is for $J$ to pass the non-intersection test on the middle interval $[\mf l,\mf r]$, i.e., for $J$ to satisfy
$$J(1,x) > J(2,x) >  \ldots >J(k,x) > \L_n(k+1,x) \quad \forall x\in[\mf l, \mf r].$$
In other words, the indicator of this last event is the same as the indicator $\pass(J)$.
It is important for our approach that the event $\{\pass(J) = 1\}$ being conditioned on does not have too low a probability. Unlike the side intervals test, there is no simple criterion for the middle interval test. In fact, an analysis was undertaken in \cite{hammond2017brownian} to obtain an appropriately strong lower bound on this probability, which holds on a high probability $\F$-measurable favourable event $\fav$ (that we will define shortly). As that argument does not serve out expository purpose, we do not present it here; instead we reproduce the statement from \cite{hammond2017brownian} in the next Proposition \ref{p.T_3(J) bound}. This is the statement we previously referenced as \emph{jump ensemble candidate proficiency}.

\begin{proposition}[Jump ensemble candidate proficiency, Proposition 4.2 of \cite{hammond2017brownian}]\label{p.T_3(J) bound}
We have that
$$
\PF\Big(\pass(J)=1\Big) \geq \exp \left\{-3973 k^{7/2} \dip^{2} D_{k}^{2}\left(\log \epsilon^{-1}\right)^{2 / 3}\right\} \cdot \one_{\mathrm{Fav}}
.$$
\end{proposition}

We will now define $\fav$, before returning to discuss the important role of Proposition~\ref{p.T_3(J) bound} in our approach.

\subsection{The definition of $\fav$ \& the role of Proposition~\ref{p.T_3(J) bound}} \label{s.jump ensemble fav}

%
%
%
%
%
The favourable event $\fav$ is defined as the intersection
$$\fav = \mathsf{F}_1\cap\mathsf{F}_2\cap\mathsf{F}_3,$$
where
\begin{align*}
\mathsf{F}_1 &= \left\{\L_n(i, x) \in T^{2}[-2 \sqrt{2}-1,-2 \sqrt{2}+1] \text { for }(i, x) \in \intint{k} \times\{-2 T, 2 T\} \right\}\\
\mathsf{F}_2 &= \left\{-T^2 \leq \L_n(k+1, x) \leq T^2 \text { for } x \in[-T, T]\right\},\\
\mathsf{F}_3 &= \bigcap_{i \in \intint{k}} \left\{\mrm{Corner}_{i}^{\mf{l}, \F} \in\left[-T^{2}, T^{2}\right]\right\} \cap\left\{\mrm{ Corner }_{i}^{\mf{r}, \F} \in\left[-T^{2}, T^{2}\right]\right\}.
\end{align*}
Note that $\fav$ is an $\F$-measurable event. As its name suggests, this event fixes good data in $\F$ on which we have strong enough control to make our arguments. The reader should view this data as being fixed in the arguments involving the jump ensemble, as we will be working only on this event; the bound on $\P(\fav^c)$ just ahead allows us to take this liberty.

The form of the favourable event respects the parabolic curvature possessed by $\L_n$. In particular, since we are working on the interval $[-2T,2T]$, we expect that at the endpoints the location of $\L_n$ will be $O(-T^2)$, which dictates the form of the three subevents $F_1, F_2,$ and $F_3$ above.

It is a simple calculation based on the definition of $\mf l$ and $\mf r$ that, on $\mathsf F_2$, 
$$\mf l \leq -T/2 \quad \text{and} \quad \mf r \geq T/2.$$
We need the knowledge that the favourable event occurs with sufficiently high probability; this is provided to us from \cite{hammond2017brownian}:

\begin{lemma}[High probability of favourable event, Lemma 4.1 of \cite{hammond2017brownian}]\label{l.fav complement prob}
$$\P\left(\fav^{c}\right) \leq \epsilon^{2^{-5} c_k D_k^3}.$$
\end{lemma}

In fact for our purposes it will be sufficient to note that $2^{-5}c_k D_k^3 \geq 1$ for all $k$, and so the upper bound above is further bounded by $\epsilon$.

Now we may discuss the central role of Proposition~\ref{p.T_3(J) bound} in our argument. We will use it to reduce the problem of understanding the probability of an event under the $\F$-conditional law of $\L_n$ to understanding the same under the law of $J$. For concreteness, let us illustrate this by attempting to bound the probability that the vertically shifted $k^\text{th}$ curve $\L_n(k,\cdot\,) - \L_n(k,-d)$ lies in a measurable subset $A \subseteq \mc C_{0,*}([-d,d], \R)$ of continuous functions vanishing at $-d$, where $d>0$. Recall again that the $\F$-conditional distribution of $\L_n$ on $[\mf l, \mf r]$ is the same as the distribution of $\L^{\mathrm{re}, J}$ on $[\mf l, \mf r]$ conditioned on the event that $\pass(J)=1$, and also that $\L^{\mathrm{re}, J}(i, \cdot\,) = J(i,\cdot\,)$ on $[\mf l, \mf r]$. We assume that $\epsilon$ is small enough that $[-d,d] \subseteq [-T/2,T/2]\subseteq [\mf l, \mf r]$, the last inclusion on the event $\fav$. (This assumption on $\epsilon$ is implied by the condition that $\epsilon<\exp(-(24)^{6}d^6/D_k^3)$ that is imposed in Theorem~\ref{t.airytail.ln}.) We also have to set the last parameter of the jump ensemble, the inter-pole distance $\dip$, which we set as
\begin{equation}\label{e.dip value}
\dip = 5d,
\end{equation}
which will be its value in our application of the jump ensemble. Then we see that
\begin{equation}
\begin{split}\label{e.ln prob breakup}
\P\Big(\L_n(k,\cdot\,) - \L_n(k, -d)\in A\Big) &= \E\left[\PF(\L_n(k,\cdot\,) - \L_n(k, -d)\in A)\cdot\one_\fav\right] + \P(\fav^c)\\
&= \E\left[\PF\big(J(k,\cdot\,)- J(k, -d) \in A \mid \pass(J)=1 \big)\cdot\one_\fav\right] + \P(\fav^c)\\
&\leq \E\left[\frac{\PF(J(k,\cdot\,) - J(k, -d) \in A)}{\PF\left(\pass(J)=1\right)}\cdot\one_\fav\right] + \P(\fav^c)\\
&\leq \E\left[\PF\big(J(k,\cdot\,)- J(k, -d) \in A\big)\cdot\one_\fav\right]\cdot \exp \left\{O_k(1)\left(\log \epsilon^{-1}\right)^{2 / 3}\right\}\\
&\qquad + \P(\fav^c),
\end{split}
\end{equation}
using Proposition~\ref{p.T_3(J) bound} in the last inequality.

Theorem~\ref{t.airytail.ln} asserts a bound of the form $\epsilon\cdot \exp(O_k(1)(\log\epsilon^{-1})^{5/6})$ on the left-hand side of the first line of the above display, where $\epsilon$ is the probability of $A$ under the law of Brownian motion. So in order to prove Theorem~\ref{t.airytail.ln}, the main step to be effected is to bound the first term after the last inequality by a quantity of the same form. The notational equivalence we have just made between the parameter $\epsilon$ of the jump ensemble and the Brownian motion probability of the event of interest is one we will adopt formally: in the remainder of the proof of Theorem~\ref{t.airytail.ln}, the parameter $\epsilon$ of the jump ensemble will have the value
\begin{equation}
\epsilon := \B_{0,*}^{[-d,d]}\big(A\big).
\end{equation}
(As we will record formally soon, the parabolic invariance Lemma~\ref{l.shifted is regular} allows us to reduce Theorem~\ref{t.airytail.ln} to the case where $K=0$.)
With this choice of $\epsilon$, the importance of Proposition~\ref{p.T_3(J) bound} in achieving the goal mentioned in the last paragraph is now clear, in particular that the exponent of the $\log\epsilon^{-1}$ in the exponent of the statement of Proposition~\ref{p.T_3(J) bound} is $2/3<5/6$. Looking back at \eqref{e.ln prob breakup}, to actually achieve this goal we need two bounds: that $\PF\big(J(k,\cdot\,) - J(k, -d)\in A\big)\cdot\one_\fav$ and $\P(\fav^c)$ are both bounded by $\epsilon \exp(O_k(1)(\log\epsilon^{-1})^{5/6})$. 

The second bound is implied by Lemma~\ref{l.fav complement prob}. Finally, in the following Theorem~\ref{t.prob estimate for J}, we have the first bound, proving which will be the work of the next chapter. After stating this theorem and an immediate corollary, we will end this section by giving a brief summary of the jump ensemble. Recall that $\B_{0,*}^{[-d,d]}$ is the law on $\mc C_{0,*}([-d,d], \R)$ of a Brownian motion started at coordinates $(-d,0)$.

\begin{theorem}\label{t.prob estimate for J}
Let $d\geq 1$ and $A\subseteq \mc C_{0,*}([-d,d], \R)$. Then there exist $\epsilon_0 = \epsilon_0(d, k) = \exp(-(24)^6d^6/D_k^3)$ and absolute constant $G<\infty$ such that, if $\B_{0,*}^{[-d,d]}(A) = \epsilon < \epsilon_0$,
$$\PF\Big(J(k,\cdot\,) - J(k, -d) \in A\Big)\cdot\one_\fav \leq \epsilon\cdot Gd^{\frac{1}{2}}\cdot D_k^{4} (\log \epsilon^{-1})^{\frac{4}{3}}\cdot \exp\left(792\cdot d\cdot D_k^{5/2}\cdot (\log \epsilon^{-1})^{5/6}\right).$$
\end{theorem}


Given this theorem we may prove the main Theorem~\ref{t.airytail.ln}:

\begin{proof}[Proof of Theorem~\ref{t.airytail.ln}]
	By applying the parabolic invariance Lemma~\ref{l.shifted is regular} with $y_n=-K$, proving Theorem~\ref{t.airytail.ln} reduces to the case $K=0$. The condition that $[K-d,K+d]\subset c/2\cdot[-n^{1/9},n^{1/9}]$ is exactly the one required to apply Lemma~\ref{l.shifted is regular}.

	We have that
	\begin{align*}
	\P\bigl(\L_n(k,\cdot\,)-\L_n(k, -d)\in A\bigr) &= \E\Bigl[\PF\bigl(J(k,\cdot\,)-J(k, -d)\in A \ \bigm|\  \pass(J)=1\bigr)\cdot\one_\fav\Bigr] + \P(\fav^c)\\
	&\leq \E\left[\frac{\PF\bigl(J(k,\cdot\,)-J(k, -d)\in A\bigr)}{\PF\left(\pass(J)=1\right)}\cdot\one_\fav\right] + \P(\fav^c).
	\end{align*}
	By Lemma~\ref{l.fav complement prob}, $\P(\fav^c) < \epsilon$ for the choice of $D_k$ we have assumed. This bound requires \eqref{e.alowerbound}, which we have also assumed. By Proposition~\ref{p.T_3(J) bound} and Theorem \ref{t.prob estimate for J}, we find that the last expression is bounded up to a constant factor by
	\begin{align*}
	\MoveEqLeft[10]
	\epsilon\cdot d^{\frac{1}{2}}\cdot D_k^{4} (\log \epsilon^{-1})^{\frac{4}{3}}\cdot \exp\left(792\cdot d\cdot D_k^{5/2}\cdot (\log \epsilon^{-1})^{5/6} + 3973 k^{7/2} \dip^{2} D_{k}^{2}\left(\log \epsilon^{-1}\right)^{2 / 3}\right) + \epsilon&\\
	&\leq 	2\epsilon\cdot d^{\frac{1}{2}}\cdot D_k^{4} (\log \epsilon^{-1})^{\frac{4}{3}}\cdot \exp\left(4931 \cdot d\cdot k^{7/2}\cdot D_k^{5/2}\cdot (\log \epsilon^{-1})^{5/6}\right);
	\end{align*}
	we have used that $\dip = 5d$ from \eqref{e.dip value}, $d\geq 1$, $D_k^{1/2}(\log\epsilon^{-1})^{1/6}\geq 24d$ from the assumed upper bound on $\epsilon$, and $3973\times 5^2/24 + 792 \leq 4931$. Since $x^{11/5} \leq e^x$ for all $x\geq 1$, we may absorb the factor $d^{\frac{3}{2}} D_k^{\frac{11}{2}} (\log \epsilon^{-1})^{\frac{11}{6}}$ by increasing the coefficient of the exponent by 1.	This proves Theorem~\ref{t.airytail.ln}.
\end{proof}

Hence the remaining task is to prove Theorem~\ref{t.prob estimate for J}. This is accomplished in Chapter~\ref{ch.the proof}. Before proceeding to this, we give in the next subsection an important statement about stochastic domination relations between the jump ensemble and certain Brownian bridges, and finally a concluding subsection giving a brief summary of the jump ensemble which may act as a quick reference for the reader.

\subsection{A stochastic domination property of the jump ensemble}

In our proof of Theorem~\ref{t.prob estimate for J}, at various points we will need to stochastically dominate $J(k,\cdot\,)$ by or have $J(k,\cdot\,)$ stochastically dominate certain Brownian bridges. For this we will make use of the following statement, which will be proven essentially using Lemma~\ref{l.bb independent decomposition property} to reduce stochastic domination of processes to obvious stochastic dominations of point-values.

\begin{lemma}\label{l.J stochastic dominanation}
Fix $n\in\N$ and $k\leq n$. Let $\phistart \geq \L_n(k,-2T)$ and $\phiend \geq \L_n(k, 2T)$, and let $x\in [\mf l, \mf r]$. Then, on $\fav$ and conditionally on $\F$,
\begin{enumerate}[label=(\roman*)]
	\item The law of $J(k,\cdot\,)$ (as a law on $[\mf l, \mf r]$) stochastically dominates the law of a Brownian bridge from $(\mf l, -T^2)$ to $(\mf r,-T^2)$.

	\item Conditionally on $J(k, x)$, the law of $J(k,\cdot\,)$ restricted to $[\mf l, x]$ stochastically dominates the law of a Brownian bridge from $(\mf l, -T^2)$ to $(x, J(k, x))$.
	Under the same conditioning, the law of $J(k,\cdot\,)$ restricted to $[x, \mf r]$ stochastically dominates the law of a Brownian bridge from $(x, J(k, x))$ to $(\mf r, -T^2)$.

	\item The law of $J(k,\cdot\,)$ is stochastically dominated by the law of the restriction to $[\mf l,\mf r]$ of a Brownian bridge from $(-2T, \phistart)$ to $(2T, \phiend)$ which is conditioned to be above all poles, above $(\mf l, \gxmin_k)$, and above $(\mf r,\gymin_k)$.

	\item Conditionally on $J(k,x)$, the law of $J(k,\cdot\,)$, restricted to $[x, \mf r]$, is stochastically dominated by the law of the restriction to $[x,\mf r]$ of a Brownian bridge from $(x, J(k,x))$ to $(2T, \phiend)$ which is conditioned to be above all poles in $[x, \mf r]$ and above $(\mf r,\gymin_k)$.
\end{enumerate}
\end{lemma}

\begin{proof}
Lemma~\ref{l.J stochastic dominanation}(i) follows from (ii) by taking $x=\mf r$ and averaging, and noting that, on $\fav$, $J(k,\mf r)\geq \gymin_k \geq -T^2$. So we prove (ii); in fact, we prove the first part of (ii) as the second part is analogous.

Let $B$ be a Brownian bridge from $(\mf l, J(k,\mf l))$ to $(x, J(k,x))$.  Letting $m = |P|$ and $x_1, \ldots, x_{m}$ be the elements of $P$, we may apply Lemma~\ref{l.bb independent decomposition property} to decompose $B$ at the elements of $P$. Since, conditionally on $J(k,\mf l)$ and $J(k,x)$, the distribution of $J(k, \cdot\,)$ is that of $B$ conditioned on the values of $B$ at elements of $P$ being sufficiently high, the decomposition provided by Lemma~\ref{l.bb independent decomposition property} yields that the law of $J(k,\cdot\,)$, conditionally on its values at $\mf l$ and $x$, stochastically dominates that of $B$. Since, on $\fav$, $J(k,\mf l) \geq \gxmin_k \geq -T^2$, it is clear that $B$ stochastically dominates the Brownian bridge described in Lemma~\ref{l.J stochastic dominanation}(i), yielding the claim.

Now we turn to (iii). Given two intervals $I_x$ for $x\in\{\mf l, \mf r\}$, consider the restriction to $[\mf l, \mf r]$ of a Brownian bridge $B$ with starting point $(-2T, \L_n(k,-2T))$ and ending point $(2T, \L_n(k, 2T))$, conditioned on $B(x)\geq \L_n(k+1, x)$ for all $x\in P$, and on $B(x)\in I_x$ for $x\in\{\mf l, \mf r\}$. Call this law $\B(I_{\mf l}, I_{\mf r})$. 

Conditionally on $J(k-1,\cdot\,)$ as a process on $[\mf l, \mf r]$, the law of $J(k,\cdot)$ is $\B(I_{\mf l}, I_{\mf r})$ with $I_x$ a \emph{finite} interval determined by $J(k-1,\cdot\,)$, in such a way that $\inf I_x = \ggenmin_k$, for $x\in\{\mf l, \mf r\}$. Using the same decomposition from Lemma~\ref{l.bb independent decomposition property}, it is clear that $\B(I_{\mf l}, I_{\mf r})$ is stochastically dominated by $\B(I'_{\mf l}, I'_{\mf r})$, where $I'_x = [\ggenmin_k, \infty)$ for $x\in \{\mf l, \mf r\}$. 

Averaging over $J(k-1,\cdot\,)$, we find that the law of $J(k,\cdot\,)$ is stochastically dominated by $\B(I'_{\mf l}, I'_{\mf r})$. On $\fav$, the law of the process described in Lemma~\ref{l.J stochastic dominanation}(iii) is that of the Brownian bridge $B$ in the definition of $\B(I'_{\mf l}, I'_{\mf r})$ with the endpoints shifted vertically upwards, which clearly maintains the described stochastic domination.

The proof of (iv) is along the same lines as (iii).
\end{proof}

\subsection{Summary of the jump ensemble} The definition of the jump ensemble was rather involved, and here we provide a quick summary of the main aspects of its definition which the reader should keep in mind in order to understand the arguments leading to the proof of Theorem~\ref{t.prob estimate for J}.

The jump ensemble has three parameters, $\epsilon>0$, $k\in \N$, and $\dip>0$; and is defined on an interval $[-2T,2T]$, where $T=D_k(\log\epsilon^{-1})^{1/3}$, with $D_k$ given by \eqref{e.D_k value}.  The interval $[-2T,2T]$ contains a random subinterval $[\mf l, \mf r]$. The jump ensemble $J$ is the restriction to $[\mf l,\mf r]$ of a collection of $k$ independent Brownian bridges, the $i^{\text{th}}$ from $(-2T, \L_n(i, -2T))$ to $(2T, \L_n(i,2T))$, conditioned on $\overline J(x) - \gbarxmin \in (0,\infty)^k_>$ for $x\in\{\mf l,\mf r\}$ and on $J(p) > \L_n(k+1, p)$ for all $p$ in the pole set $P$. The elements of the pole set $P$ have a minimum separation of $\dip$, and are a subset of $[\mf l, \mf r]$. This interval is defined in terms of the underlying curve $\L_n(k+1,\cdot\,)$, and the relevant consequence of its definition is that the slope (of each linear segment) of the $\tent$ map (which linearly interpolates the points $(p,\L_n(k+1,p))$) lies in $[-4T,4T]$.

In our application, $\epsilon$ is set according to the Brownian motion probability of the event under consideration, and $\dip = 5d$.

We gain control over several of the random objects present in the above paragraph on a high probability favourable event $\fav$. On this event, we have that $[-T/2, T/2]\subseteq [\mf l,\mf r]$; that $\gbarxmin, \gbarymin \in [-T^2, T^2]^k$; that $\L_n(k+1,x) \in [-T^2, T^2]$ for $x\in[-T,T]$; and that $ \L_n(i, x) \in T^2[-2\sqrt 2-1, -2\sqrt 2+1]$ for $(i,x)\in\intint{k}\times\{-2T,2T\}$.

In the next section we discuss the conceptual framework underlying the proof of Theorem~\ref{t.prob estimate for J}.





\section{A conceptual framework in terms of costs}\label{s.heuristics for proof}

\begin{notation}
We will use the notation $A\lesssim B$ to indicate that there exists a constant $G<\infty$ which is independent of $\epsilon$, $k$, and $d$ such that $A\leq GB$. The value of $G$ may however vary from line to line.
\end{notation}

We now begin discussing the ideas underlying the approach of the proof of Theorem~\ref{t.prob estimate for J}. We fix the curve index $k$ that we are studying, and for the proof of Theorem~\ref{t.prob estimate for J} in this section and in Chapter~\ref{ch.the proof} adopt the abuse of notation 
$$J(\,\cdot\,) = J(k,\cdot\,),$$
which we will refer to as the \emph{jump curve}; none of our arguments will involve the other curves of the full jump ensemble.

 The overarching conceptual framework is one of \emph{costs} that a certain joint density of the jump curve $J$ must satisfy. This section will set up the quantities the proof will be working with, introduce the costs these quantities must interact with, and conclude by showing that the proof of Theorem~\ref{t.prob estimate for J} can be reduced to showing that these costs are met.

We would like to study the probability that the jump curve $J$ lies in some $A\subseteq \mc C_{0,*}([-d,d], \R)$, where the probability of $A$ under the law of Brownian motion is $\epsilon$. As mentioned at the end of the last section, the parameters of the jump ensemble are $\epsilon$ and $k$, with $\dip = 5d$. Our aim is to get a bound of $\epsilon$ on the jump curve probability of $A$, up to a constant multiplicative factor which is subpolynomial in $\epsilon^{-1}$. 

Recall that the jump ensemble involves the notion of a pole set. We are provided control over the minimum separation of consecutive poles, and so by our choice of $\dip = 5d$ we ensure that there is at most one pole in the larger interval $[-2d,2d]$. Thus there are two possibilities: either the interval $[-2d, 2d]$ contains a single pole, or no poles. 

In the case where there are no poles in $[-2d,2d]$, there can of course be poles which are arbitrarily close to $[-2d, 2d]$, and we will see in our analysis that our bounds become too weak if a pole is too close to a point under consideration. This is why we consider the presence of poles in $[-2d, 2d]$ even though the interval of interest is $[-d,d]$: when there is no pole in $[-2d,2d]$ we can focus the analysis at the points $\pm d$, which are then ensured a distance of at least $d$ from the nearest pole. When there \emph{is} a pole, we will adopt a trick that will be described later to allow us to step back from it, again giving us an order $d$ distance from it.

In the scenario where there is no pole in $[-2d,2d]$, the jump curve is essentially just a Brownian bridge on $[-2d,2d]$, and the argument for the probability comparison we claim is much more straightforward. So for the purposes of the exposition in this section we discuss the more difficult case where there is a pole in $[-2d,2d]$. 

By the Brownian Gibbs property, we have a direct way to write the probability of $A$ for the jump curve, given its endpoint values $J(-2d)$ and $J(2d)$, in terms of the Brownian bridge probability of $A$ between those endpoint values. In some sense, our task is to show that this Brownian bridge probability becomes a Brownian motion probability when we take expectations over $J(-2d)$ and $J(2d)$.

\subsection{Translating between Brownian motion and Brownian bridge}
It is instructive to look at how Brownian motion probabilities translate to Brownian bridge probabilities. We remind the reader of a useful property of Brownian motion: if $B$ is a Brownian motion, then the distribution of $B$ on an interval $[x_1,x_2]$, conditionally on $B(x_1)$ and $B(x_2)$, is that of Brownian bridge from $(x_1, B(x_1))$ to $(x_2, B(x_2))$. 

Now suppose we are again working on $[-2d,2d]$, and suppose $B$ is a rate one Brownian motion started from $x$-coordinate $-4d$ according to some probability distribution $\mu$. Right now it is not clear why we need the device of the measure $\mu$, but we will retain it to give ourselves some freedom which, by proper choice at a later point, will make the requirements more lenient. At present it can be thought of as, and indeed will later be defined to be, the uniform measure on a very large interval centred at 0.

On a heuristic level, the probability that $B(-4d) \in [y,y+ \mathrm dy]$ and $B(4d) \in [z, z+\mathrm dz]$ is given by $(16\pi d)^{-\frac{1}{2}}\exp(-(y-z)^2/16d)\,\mathrm d\mu(y)\,\mathrm dz$. By the Markov property, $B(\,\cdot\,) - B(-d)$ is standard Brownian motion started at coordinates $(-d,0)$ when restricted to $[-d,\infty)$. Now, recall that $\B^{[-d,d]}_{0,*}$ is the law of standard Brownian motion on $[-d,d]$, started at coordinates $(-d,0)$, and that $\B^{[-4d,4d]}_{y,z}$ is the law of rate one Brownian bridge from $(-4d,y)$ to $(4d,z)$. Then for $A\subseteq \mc C_{0,*}([-d,d], \R)$, we have the calculation
\begin{align}
\B^{[-d,d]}_{0,*}\big(A\big) &= \P\big(B(\,\cdot\,)- B(-d)\in A\big) \nonumber\\
 &= \E\Big[\P\big(B(\,\cdot\,) - B(-d)\in A \mid B(-4d), B(4d)\big)\Big] \nonumber\\
&= \frac{1}{\sqrt{16\pi d}}\int_{-\infty}^{\infty}\int_{-\infty}^{\infty} \B^{[-4d, 4d]}_{y, z}\big( \tilde A\big)\cdot e^{-(y-z)^2/16d} \, \mathrm d\mu(y)\, \mathrm dz,  \label{e.bm probability expression}
\end{align}
%
where $\tilde A$ is the set of functions $f$ in $\mc C_{*,*}([-4d, 4d], \R)$ such that $f(\,\cdot\,) - f(-d)$, regarded as a function with domain $[-d,d]$, lies in $A$. As we said before, $(16\pi d)^{-\frac{1}{2}}\exp(-(y-z)^2/16d)$ is exactly the conditional density of $B(4d)$ at $z$ given that $B(-4d)$ is $y$. A very similar calculation holds for the jump curve, which points us to what we should try to prove. Let $f_J(y,z)$ be the joint density of 
\begin{align*}
Y^* &:= J(-4d) \quad \text{and} \quad Z^* := J(4d)
\end{align*}
 at the point $(y,z)\in\R^2$; the $*$ is an adornment that will be removed in the final definition of $Y$ and $Z$ that will be used in our actual arguments.

 Suppose now, for simplicity, that the pole $p$ in $[-2d,2d]$ is at zero and has height zero (i.e., $p = 0$ and $\tent(p) = 0$). Then, on $\fav \cap \{P\cap[-2d,2d] \neq \emptyset\}$,
\begin{equation}\label{e.bb to bm probability}
\begin{split}
\PF\big(J(\,\cdot\,) - J(-d) \in A\big) &= \EF\Big[\PF\Big(J(\,\cdot\,) - J(-d) \in A \, \Big|\,  J(-4d), J(4d)\Big)\Big] \\
&= \EF\Big[ \B^{[-4d,4d]}_{J(-4d),J(4d)} \big(\tilde A \mid J(0)\geq 0\big)\Big] \\
&\leq \int_{-\infty}^{\infty}\int_{-\infty}^{\infty} \frac{\B^{[-4d,4d]}_{y, z}\big( \tilde A\big)}{\B^{[-4d,4d]}_{y, z}\big( J(0) \geq 0\big)}\cdot f_J(y,z) \, \mathrm dy\, \mathrm dz. 
\end{split}
\end{equation}
In essence, our aim is to run this calculation forward, and the previous one backwards, in order to get from the jump curve probability of $A$ to the standard Brownian motion probability of $A$. Then by direct comparison of the integrands, in order for the last line of the second calculation to be roughly equal to the last line of the first calculation we would need
\begin{align*}
f_J(y, z)\  \text{``}&=\text{"}\  \B^{[-4d,4d]}_{y, z}\big( J(0) \geq 0\big) \cdot d^{-\frac{1}{2}} \exp\left(-\frac{(y-z)^2}{16d}\right) \cdot \frac{\mathrm d\mu}{dy}(y)\\
& = \P\left(N\left(\frac{y+z}{2}, 2d\right) \geq 0\right) \cdot d^{-\frac{1}{2}} \exp\left(-\frac{(y-z)^2}{16d}\right)  \cdot \frac{\mathrm d\mu}{dy}(y).
\end{align*}
%
%
Of course, we will have to provide ourselves some error margins in order to succeed. This is hiding in the ``$=$" symbol above, which means that the left side is bounded above by the right side, possibly multiplied by a constant of the form $\exp(G(\log\epsilon^{-1})^{5/6})$, which we will often refer to as the \emph{leeway factor}; the notation of ``$=$'', however, we use only in this instance. Since $T = D_k(\log\epsilon^{-1})^{1/3}$, we will also often refer to constants of the form $\exp(GT^{5/2})$ as leeway factors. 

So equivalently, what we require is that
\begin{equation}\label{e.the two costs}
f_J(y, z) \cdot \left[\P\left(N\left(\frac{y+z}{2}, 2d\right) \geq 0\right)\right]^{-1} \cdot d^{\frac{1}{2}}\exp\left(\frac{(y-z)^2}{16d}\right) \leq \exp\left(G(\log\epsilon^{-1})^{5/6}\right).
\end{equation}
Notice that we have temporarily ignored the issue of choosing $\mu$ and the slightly stronger demand (than the above inequality) that will arise from its Radon-Nikodym derivative with respect to Lebesgue measure. We will return to this in a few paragraphs, but only promise here that the cost will be polynomial in $\log\epsilon^{-1}$ instead of exponential, and so will not substantively affect the analysis. 

\begin{remark}\label{r.y z convention}
 This correspondence between the random variables, $Y$ and $Z$ (temporarily with the $*$ adornment added), and the arguments of their joint density, $y$ and $z$, is one that will be maintained throughout the proof of Theorem~\ref{t.prob estimate for J}. Sometimes we will also refer to the marginal densities of $Y$ or $Z$. By an abuse of notation (as $Y$ and $Z$ do not necessarily have the same distribution), we will refer to both densities as $f_J$, distinguishing whether we mean that of $Y$ or $Z$ based on whether the argument is $y$ or $z$. This will not cause any confusion as at no point will we refer to a marginal density at a specific value. Similarly, $f_J(z\mid y)$ will be the conditional density of $Z$ given $Y=y$, evaluated at the point $z$, and so on.
 \end{remark}

\subsection{The vault and slope costs}
With this simplification, let us focus on the two terms multiplying $f_J(y,z)$ on the left-hand side of \eqref{e.the two costs}. They arise from two conceptually distinct sources. The first of these two terms comes from the potential difficulty a Brownian bridge faces in order to jump or \emph{vault} over the pole, a task which $J$ is conditioned to accomplish. Accordingly we refer to this factor as the \emph{vault} cost, and it will be denoted by $\jcost^*$ (the $*$ being again an adornment that will be removed in the final corrected version of $\jcost$ which we will use). The second of the two terms, on the other hand, is the potential difficulty faced by a Brownian motion to attain the slope specified by $y$ and $z$, which may be equivalently be thought of as an increment, across an interval of length~$4d$. We refer to this as the \emph{slope} cost, and denote it by $\scost^*$. Thus,
\begin{equation}\label{e.V* and S*}
\begin{split}
\jcost^* &:=  \left[\P\left(N\left(\frac{y+z}{2}, 2d\right) \geq 0\right)\right]^{-1}\\
\scost^* &:= d^{\frac{1}{2}}\cdot\exp\left(\frac{(y-z)^2}{16d}\right).
\end{split}
\end{equation}
Note that $\jcost^*$ behaves differently depending on the values of $y$ and $z$; for example, if $y+z$ is  positive, the probability is bounded below by a constant, and so $\jcost^*$ is bounded above by a constant and is easily managed by the margin of error we have provided. On the other hand, if $y+z$ is negative, $\jcost^*$ can be seen to be roughly $d^{\frac{1}{2}}\exp((y+z)^2/16d)$, thus posing a much more serious demand. More precisely,
\begin{equation} \label{e.jcost* simplified form}
\jcost^* \lesssim \begin{cases}
1 & y+z > 0\\
d^{\frac{1}{2}}\cdot\exp\left(\frac{(y+z)^2}{16d}\right) & y+z \leq 0,
\end{cases}
\end{equation}
Our analysis will later break into cases based on this fact.

Thus, \eqref{e.the two costs} says that roughly what we need to show, for some $G<\infty$, is that
$$f_J(y,z)\cdot \jcost^*\cdot\scost^*\leq \exp\left(G(\log\epsilon^{-1})^{5/6}\right).$$

The above heuristic description was idealised to highlight the main features of the approach, but is in essence correct. We now discuss which aspects of the description change in the actual approach. There were three simplifying assumptions: that the pole position $p$ equals zero; that the height of the pole is zero; and the postponement of the choice of $\mu$. In addressing each of these simplifications, we will come to the final quantities $\jcost$ and $\scost$, and we will see that $\jcost$ and $\scost$ respectively equal $\jcost^*$ and $\scost^*$ up to the leeway factor.

\subsection{Addressing the simplifications}\label{s.heuristic.addrssing simplifications}
\subsubsection*{Pole position $p=0$.} The first simplifying assumption we made in the heuristic description was that the pole in $[-2d,2d]$ lies at $0$. In general, of course, we have very limited control over the pole position as it is determined by the $(k+1)^{\text{st}}$ curve of the original line ensemble. This point will be addressed essentially by changing our frame of reference horizontally.

Call the pole position $p \in [-2d,2d]$ and suppose $p\neq 0$. The coefficient of $1/2$ for each of $y$ and $z$ in $\jcost^*$ was due to the pole position of 0 being equidistant from $-4d$ and $4d$. One option would be to maintain the random variables $Y^*$ and $Z^*$ to be the values at $\pm4d$, in which case the coefficients would not be $1/2$  but $\lambda$ and $1-\lambda$ for some $\lambda \in[0,1]$. While the subsequent analysis could possibly be adapted for this case, this would introduce an undesirable level of complication to the formulas. So to maintain the symmetric coefficient of $1/2$, we instead have $Y^*$ and $Z^*$ be the values at $p-4d$ and $p+4d$, i.e., $J(p-4d)$ and $J(p+4d)$ respectively. (We will introduce a further modification shortly which will be the final definition of $Y$ and $Z$, without the $*$ adornment.)

While on some level this is merely a trick, it is one which leads to very useful simplifications. At a technical level, though $p$ is a random variable, it is an $\F$-measurable one, and so this trick is sound: the relevant point is that, given the $\F$-data, there is no obstacle to writing the law of the jump curve on $[-2d,2d]$ as the marginal of Brownian bridge on $[p-4d, p+4d]$, between correctly distributed endpoints, which is conditioned to jump over the pole at $p$.

\subsubsection*{Pole height $\text{Tent}\,(p) = 0$.} Now we address the height of the pole. Here also we essentially employ a change in the frame of reference. At this point in our discussion, we have $Y^*$ and $Z^*$ being the values of $J(p-4d)$ and $J(p+4d)$ themselves, i.e., the deviation from $0 = \tent(p)$. When $\tent(p) \neq 0$, the obvious choice is to let $Y$ and $Z$ respectively represent the deviations of $J(p-4d)$ and $J(p+4d)$ from $\tent(p)$. But to make certain derivations slightly simpler in the sequel, we instead let the final definitions of $Y$ and $Z$ be the respective deviations of $J$ from $\tent$ at $p-4d$ and $p+4d$, i.e.,
\begin{equation}\label{e.Y Z defn}
\begin{split}
Y:= J(p-4d) - \tent(p-4d),\\
Z:= J(p+4d) - \tent(p+4d).
\end{split}
\end{equation}
See Figure~\ref{f.y z defn}. This is the final definition of $Y$ and $Z$ which will be maintained for the rest of the argument, at least on the event that $P\cap[-2d,2d]\neq \emptyset$. (On the other event, where there is no pole, they will have conceptually analogous but different definitions.)

\begin{figure}
\centering {\epsfig{file=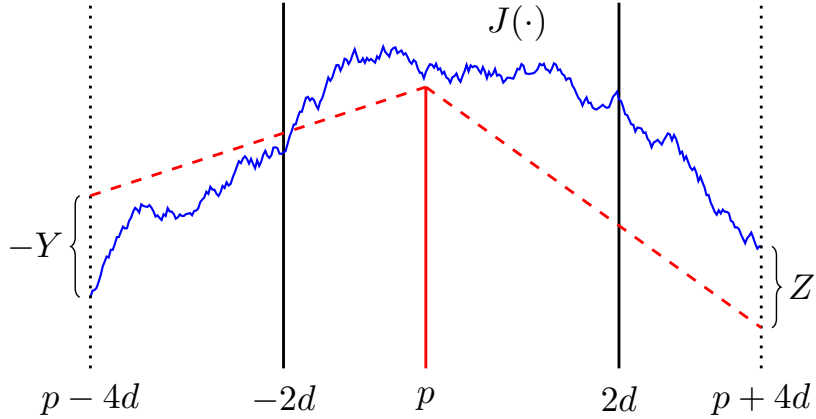, width=0.67\textwidth}}
\caption{Illustrating the definition of $Y$ and $Z$ when $P\cap[-2d,2d]\neq\emptyset$. The $\tent$ map is depicted as the piecewise linear function drawn in dashed red. The single pole $p$ in $[-2d,2d]$ is depicted as a vertical red line. The jump curve is blue. In this figure $Y$ is negative as $J$ is below $\tent$ at $p-4d$, and thus the distance indicated is $-Y$. The quantity $Z$ is positive as $J$ is above $\tent$ at $p+4d$.}\label{f.y z defn}
\end{figure}

\subsubsection*{The choice of $\mu$.} We still have to make a suitable choice for $\mu$. We start by providing some intuition as to the role of $\mu$.

Since our approach to move from jump curve probabilities \eqref{e.bb to bm probability} to Brownian motion probabilities \eqref{e.bm probability expression} is somewhat crude, in the sense that we are trying to directly compare the integrands to conclude that the integrals are comparable, in order to succeed we cannot allow the integrand in \eqref{e.bm probability expression} to be zero when the integrand in \eqref{e.bb to bm probability} is not small. This suggests that the support of $\mu$ should be the support of the law of $Y$, i.e., the support of $f_J(y)$ (recall from Remark~\ref{r.y z convention} that this is the marginal density of $Y$). However, it is reasonable to assume that $f_J(y)>0$ on the entire real line, and we run into a problem when we try to find a full support measure $\mu$ that also satisfies the constraints imposed by the costs $\jcost$ and $\scost$. This will be shown in the discussion in the following paragraphs. Heuristically, the solution will be to let $\mu$ have support only where $f_J(y)$ is not too small in some sense that we will specify.

From the above expressions for $\jcost^*$ and $\scost^*$, we see that, in the worst case, the total cost to be paid is $\exp\left(\frac{y^2}{8d}+\frac{z^2}{8d}\right).$ This tells us that we cannot afford for $\frac{\mathrm d\mu}{\mathrm dy}$ to be too small anywhere, as the additional cost corresponding to the choice of $\mu$ will be the reciprocal of this derivative. The logic behind this inference is as follows: in the worst case, we have (without loss of generality) $y<0$, as well as the condition $J(0)\geq 0$, where for the sake of simplified discussion we have again assumed $p=0$ and $\tent(p)=0$. A Brownian bridge forced to make a large jump (from $y$ to $0$) in unit order time will, because of the difficulty, make the required jump with very little extra margin. So heuristically we may think of the $J(0)\geq 0$ condition as being $J(0) = 0$. But the density of a Brownian motion which is 0 at 0 having value $y$ at $-4d$ and $z$ at $4d$ is exactly $(8\pi d)^{-1}\exp\left(-y^2/8d -z^2/8d\right)$. Thus we expect that this is the best density bound we can hope for $J$ as well, which means any extra requirement imposed by $\frac{\mathrm d\mu}{\mathrm dy}$ must be absorbable in the leeway factor of $\exp\left(G(\log \epsilon^{-1})^{5/6}\right)$.

In other words, we require 
\begin{align}
\frac{\mathrm d\mu}{\mathrm dy}(y) \geq \exp\left(-G(\log \epsilon^{-1})^{5/6}\right). \label{e.mu rn deriv reqd lower bound}
\end{align}
This bound cannot hold for all $y$ for the density of a probability measure. But a closer look at \eqref{e.bb to bm probability} shows a modification we may make to that calculation: we can look for the density bound on $f_J(y,z)$ that we have been discussing for $y$ and $z$ in a \emph{good} region, and find a separate argument for why the contribution of the integral from the bad regions is $O(\epsilon)$. More precisely, call the good region $\mc G_R^{(1)} \subseteq \R^2$ ($R$ is a parameter we will set later, and the $1$ in the superscript refers to the fact that this is the first case, when a pole is present), and let $\tilde y = y+\tent(p-4d)$, $\tilde z = z+\tent(p+4d)$; this is so that $Y=y$ implies $J(p-4d)=\tilde y$, and similarly $Z=z$ implies $J(p+4d)=\tilde z$. Then, on $\fav \cap \{P\cap[-2d,2d] \neq \emptyset\}$,
\begin{align}
\PF\big(J(\,\cdot\,) - &J(-d) \in A\big) \label{e.main calculation}\\
&= \EF\Big[\PF\big(J(\,\cdot\,) - J(-d) \in A \mid J(p-4d), J(p+4d)\big)\Big]\nonumber \\
&\leq \EF\Big[\PF\big(J(\,\cdot\,) - J(-d) \in A \mid J(p-4d), J(p+4d)\big)\cdot \one_{(Y,Z)\in \mc G_R^{(1)}}\Big] + \PF\big((Y,Z) \not \in \mc G_R^{(1)}\big)\nonumber\\
&= \EF\Big[ \B^{[p-4d, p+4d]}_{J(p-4d),J(p+4d)} \big(\widetilde A \mid J(p)\geq \tent(p)\big)\cdot \one_{(Y,Z)\in \mc G_R^{(1)}}\Big] + \PF\big((Y,Z) \not \in \mc G_R^{(1)}\big)\nonumber\\
&\leq \iint_{\mc G_R^{(1)}} \frac{\B^{[p-4d, p+4d]}_{\tilde y, \tilde z}\big(\widetilde A\big)}{\B^{[p-4d, p+4d]}_{\tilde y, \tilde z}\big( J(p) \geq \tent(p)\big)}\cdot f_J(y,z) \, \mathrm dy\, \mathrm dz  + \PF\big((Y,Z) \not \in \mc G_R^{(1)}\big),\nonumber
\end{align}
where, recall, $\tilde A$ is the set of functions $f$ in $\mc C_{*,*}([-4d, 4d], \R)$ such that $f(\,\cdot\,) - f(-d)$, regarded as a function with domain $[-d,d]$, lies in $A$. Thus we see that, if we can show the second term in the final displayed line is $O(\epsilon)$, then we only require the density bounds associated with $\jcost$ and $\scost$ for $(y,z)\in \mc G_R^{(1)}$, and the same holds true for the lower bound on $\frac{\mathrm d\mu}{\mathrm dy}(y)$. It will turn out that including the condition $(y,z) \in [-RT^{3/2}, RT^2]^2$ in the definition of $\mc G_R^{(1)}$ allows us to obtain the $O(\epsilon)$ bound on the second term. (We will remark later in Section~\ref{s.heuristics.final costs} on why we require the lower bound to be $-RT^{3/2}$ and not $-RT^2$.) So we only need to meet the condition \eqref{e.mu rn deriv reqd lower bound} for $y,z\in[-RT^{3/2}, RT^2]$. An easy choice which meets this requirement is the uniform measure on $[-RT^2, RT^2]$, and so we let $\mu$ be this measure. Thus, we set
$$\frac{\mathrm d\mu}{\mathrm dy}(y) = (2RT^2)^{-1},$$
for $y\in[-RT^2,RT^2]$. We will treat separately the corresponding cost, which is equal to $2RT^2$, and not include it in either $\jcost$ or $\scost$.

\subsubsection*{The good regions}
We now give the definition of $\mc G_R^{(1)}$ as a subset of $\R^2$.
\begin{equation}\label{e.G_R^1 defn}
\mc G^{(1)}_R = \left\{(y,z) \in \R^2 \ : \  \parbox[c]{1.3in}{\centering
                       $y \in (-RT^{3/2}, RT^{2})$,\\
                       $z\in (-RT^{3/2}, RT^{2}),$\\
                       $|y-z| < 2RT^{3/2}$}\right\}.
\end{equation}
We have included an extra condition that $|y-z| \leq 2RT^{3/2}$. To have the $O(\epsilon)$ upper bound on the probability of $(Y,Z)\in \mc G_R^{(1)}$ means that the increment of $J$ across an interval of $8d$ must be bounded by $O(T^{3/2}) = O((\log \epsilon^{-1})^{1/2})$ with probability at least $1-\epsilon$. This is plausible, on the basis that the probability that a Brownian motion has an increment of size greater than $O((\log\epsilon^{-1})^{1/2})$ over a unit order interval is polynomial in $\epsilon$. In fact, we have the following bound on the event that $(Y,Z)\in \mc G_R^{(1)}$.
\begin{lemma}\label{l.bad event 2}
We have for $R\geq 3$ and $d\geq 1$
$$\PF\Big((Y,Z)\not\in \mc G^{(1)}_R\Big)\cdot\one_{\fav, P\cap [-2d,2d]\neq \emptyset} \lesssim \left(\epsilon^{R^2D_k^3/8d} + \epsilon^{(R-3)^2D_k^3/2}\right)\cdot\exp\left(13R^2D_k^{5/2}\left(\log\epsilon^{-1}\right)^{5/6}\right).$$
\end{lemma}
To prove this, we will actually make use of a similar but weaker statement about a good region $\mc G^{(2)}_R \subseteq \R^2$, the difference being that the bound on the probability that $(Y,Z) \in \mc G_R^{(2)}$ will hold regardless of the presence or absence of a pole in $[-2d,2d]$. We define 
\begin{equation}\label{e.G_R^2 defn}
\mc G^{(2)}_R = \left\{(y,z)\in\R^2 \ : \  \parbox[c]{1.2in}{\centering
                       $y \in (-RT^{2}, RT^{2})$,\\
                       $z\in (-RT^{2}, RT^{2}),$\\
                       $|y-z| < 2RT^{3/2}$}\right\}.
\end{equation}

Our earlier definitions of $Y$ and $Z$ were on the $\F$-measurable event that $P\cap[-2d,2d]\neq \emptyset$ and do not make sense otherwise, as they are deviations of $J$ from $\tent$ at points defined relative to $p$. So on the event that $P\cap[-2d,2d]= \emptyset$, we define
\begin{equation}\label{e.Y Z defn no pole}
\begin{split}
Y &:= J(-d) -\tent(-d)\\
Z &:= J(d) - \tent(d).
\end{split}
\end{equation}
In the following lemma, $Y$ and $Z$ are defined in this case-specific manner.

\begin{lemma}\label{l.bad event 1}
We have that for $R\geq 3$
$$\PF\Big((Y,Z)\not\in \mc G^{(2)}_R\Big)\cdot\one_{\fav} \lesssim \left(\epsilon^{R^2D_k^3/4d} + \epsilon^{(R-3)^2D_k^3/2}\right)\exp\left(13R^2D_k^{5/2}\left(\log\epsilon^{-1}\right)^{5/6}\right).$$
\end{lemma}
While we defer the proofs of these lemmas to Section~\ref{s.easier cases}, let us say a few words about its approach. The statements about $Y$ and $Z$ being bounded above in absolute value by $RT^2$ are proved by stochastically dominating $J$ by an appropriate Brownian bridge on $\fav$, and similarly with the roles of $J$ and the Brownian bridge reversed for the lower bound. The slightly trickier issue is bounding the increment across the interval of length $8d$. We postpone discussing this point.

Based on the form of the statements of Lemmas~\ref{l.bad event 2} and \ref{l.bad event 1}, we set $R$ (somewhat arbitrarily) as
\begin{equation}\label{e.R value}
R= 6\sqrt d,
\end{equation}
so that, for small enough $\epsilon$, and since $d\geq 1$ and $D_k\geq 2$ (see \eqref{e.D_k value}),
$$\PF\Big((Y,Z)\not\in \mc G^{(1)}_R\Big)\cdot\one_{\fav, P\cap [-2d,2d]\neq \emptyset} \leq \epsilon \quad \text{and} \quad \PF\Big((Y,Z)\not\in \mc G^{(2)}_R\Big)\cdot\one_{\fav} \leq \epsilon.$$

\subsection{The final costs}\label{s.heuristics.final costs}
 Having addressed the three simplifications in the heuristic derivation, let us see how the costs have changed. Recall that the random variables $Y$ and $Z$ are respectively the quantities $J(p-4d)-\tent(p-4d)$ and $J(p+4d)-\tent(p+4d)$. As immediately preceding \eqref{e.main calculation}, let $\tilde y = y+\tent(p-4d)$ and $\tilde z = z+\tent(p+4d)$, so that $Y=y$ implies $J(p-4d) = \tilde y$, and similarly $Z=z$ implies $J(p+4d)=\tilde z$. The final definitions of $\jcost$ and $\scost$ are
\begin{align*}
\jcost &:= \B^{[p-4d, p+4d]}_{\tilde y, \tilde z}\Big(J(p) \geq \tent(p)\Big)^{-1}
= \P\left(N\left(\frac{y+z}{2}, 2d\right)\geq \tent(p) - \frac{\tent(p-4d)+\tent(p+4d)}{2}\right)^{-1}\\
\scost &:= d^{\frac{1}{2}}\cdot\exp\left(\frac{1}{16d}(\tilde y-\tilde z)^2\right)
= d^{\frac{1}{2}}\cdot\exp\left(\frac{1}{16d}\left(y-z +\tent(p-4d) - \tent(p+4d)\right)^2\right).
\end{align*}
For $(y,z)\in \mc G_R^{(1)}$, we have that $\scost$ and $\jcost$ differ from $\scost^*$ and $\jcost^*$ by factors which can be absorbed within the leeway factor, which we show now.
We will have, on expanding the second square in the exponent of $\scost$, some extra terms apart from the $(y-z)^2/16d$ present in $\scost^*$. The extra factor is
\begin{equation}\label{e.S vs S* extra factor}
\exp\left(\frac{1}{16d}(\tent(p-4d) - \tent(p+4d))^2 + \frac{1}{8d}(y-z)(\tent(p-4d) - \tent(p+4d))\right).
\end{equation}
In this expression, we see immediately that the first term in the exponent does not cause a problem: since $\tent$ has slope bounded in absolute value by $4T$ (recall \eqref{e.tent map slope}), it follows that $|\tent(p-4d)-\tent(p+4d)| \leq 32Td$, so that 
$$\frac{1}{16d}(\tent(p-4d)-\tent(p+4d))^2\leq 64T^2d,$$
 which is well below $O((\log \epsilon^{-1})^{5/6})$. We need the same bound to hold for the second term in the exponent in \eqref{e.S vs S* extra factor} as well. Again using that $\tent$ has absolute value of slope bounded by $4T$, and that $|y-z| \leq 2RT^{3/2}$ on $\mc G_R^{(1)}$, what we find is
$$(y-z)(\tent(p-4d) - \tent(p+4d))) \leq 2RT^{3/2}\cdot 32Td = 64RT^{5/2}d.$$
This is why we included the bound on $|y-z|$ of order $T^{3/2}$ in the definition of $\mc G_R^{(1)}$. In other words, we have that
\begin{equation}\label{e.S is approx S*}
S\lesssim S^*\cdot \exp\left(64T^2d + 8RT^{5/2}\right) \leq d^{\frac{1}{2}}\cdot\exp\left(\frac{1}{16d}(y-z)^2+10RT^{5/2}\right),
\end{equation}
the last inequality since $d\leq \sqrt{T}/24$, $R\geq 6$ from \eqref{e.R value}, and $64/24\leq 3$.

\begin{remark}
We here used that $d\leq \sqrt{T}/24$, which is equivalent to the assumption $\epsilon \leq \exp\left(-(24)^6 d^6/D_k^3\right)$ made in Theorems~\ref{t.airytail.ln} and \ref{t.prob estimate for J}; we had previously made use of this inequality in the proof of Theorem~\ref{t.airytail.ln}. We will be making use of this inequality many times in the sequel as well, as here to bound expressions of the form $T^2d$ by $T^{5/2}/24$, and also to reduce coefficients when we have a margin to convert a small power of $d$ to $T$.
\end{remark}

Let us now finally analyse $\jcost$ in comparison with $\jcost^*$, which will show us why we included in $\mc G_R^{(1)}$ that $y,z \geq -RT^{3/2}$.

We first focus on the right-hand quantity in the probability expression of $\jcost$, namely
$$\tent(p) - \frac{\tent(p-4d)+\tent(p+4d)}{2}.$$
By the fact that the slope of $\tent$ is bounded in absolute value by $4T$, this quantity is bounded above by 
$$\frac12\times 4d\times 4T + \frac12\times 4d\times 4T = 16Td.$$
Thus $\jcost$ is bounded as
$$\jcost \leq \P\left(N\left(\frac{y+z}{2}, 2d\right)\geq 16Td\right)^{-1}.$$
As in the analysis of $\jcost^*$ in \eqref{e.jcost* simplified form}, we see that this quantity behaves differently depending on the value of $y+z$:
\begin{equation} \label{e.jcost simplified form}
\jcost \lesssim \begin{cases}
\, 1 & y+z > 32Td\\
\, d^{\frac{1}{2}}\cdot\exp\left(\frac{1}{16d}(y+z - 32Td)^2\right) & y+z \leq 32Td.
\end{cases}
\end{equation}
The bound in the worse case, on expanding the exponent and including only the factors which differ from those of $\jcost^*$, is
$$\exp\left(64T^2d - 4(y+z)T\right).$$
Again the first term in the exponent, being $(\log\epsilon^{-1})^{2/3}$, does not cause a problem, and so we turn to the second term. Firstly we see that we have the trivial upper bound of 0 on the term $-4(y+z)T$ when $y+z\geq 0$. When $y+z<0$, we use that $y,z\geq -RT^{3/2}$ from the definition of $\mc G_R^{(1)}$ to see that
$$-(y+z)\cdot T \leq 2RT^{5/2}.$$
We note that any weaker lower bound on $y,z$ would not have been sufficient to obtain an upper bound of order $T^{5/2}$ on the above quantity. From the previous few paragraphs we conclude that, when $y+z\leq 32Td$,
\begin{equation}\label{e.jcost approximate exponential form}
\jcost \lesssim d^{\frac{1}{2}}\cdot \exp\left(\frac{1}{16d}(y+z)^2 +11RT^{5/2}d\right),
\end{equation}
since $d\leq \sqrt T/24$ and $R\geq 1$ imply $64T^2d + 8RT^{5/2} \leq (64/24)T^{5/2} + 8RT^{5/2} \leq 11RT^{5/2}$.

So overall what we have observed is that for $(y,z)\in \mc G_R^{(1)}$, it is true that
\begin{align*}
\jcost &\approx \jcost^*\\
\scost &\approx \scost^*,
\end{align*}
in the sense that the left sides are bounded by the right sides up to multiplication by $\exp\left(G(\log\epsilon^{-1})^{5/6}\right)$. Thus, while we will work with $\jcost$ and $\scost$, the reader is advised to keep in mind the more convenient expressions from \eqref{e.S is approx S*} and \eqref{e.jcost approximate exponential form} after ignoring the leeway factor.

\subsection{The bound to be proved}\label{s.bound to be proved}
With the quantities $\jcost$ and $\scost$ defined, the following is what we will prove for $(y,z)$ in the good region $\mc G_R^{(1)}$ and when $P\cap[-2d,2d]\neq \emptyset$:
\begin{equation}\label{e.bound to prove}
f_J(y,z)\cdot \jcost \cdot\scost \cdot\left(\frac{\mathrm d\mu}{dy}(y)\right)^{-1} 
 \leq G_1\cdot \exp\left(G_2(\log \epsilon^{-1})^{5/6}\right),
\end{equation}
for some finite constants $G_1 = G_1(\epsilon,k, d)$ and $G_2 = G_2(k,d)$, and for $\epsilon < \epsilon_0$ for some $\epsilon_0 >0$. Since the Radon-Nikodym derivative term is $2RT^2$, it is therefore sufficient to prove
\begin{equation}\label{e.sufficient bound to prove}
f_J(y,z)\cdot \jcost \cdot\scost \leq G_1'\cdot \exp\left(G_2'(\log \epsilon^{-1})^{5/6}\right),
\end{equation}
for some $G_1' = G_1'(\epsilon,k, d)<\infty$ and $G_2' = G_2'(k,d)>0$; then \eqref{e.sufficient bound to prove} implies \eqref{e.bound to prove} with $G_1 = G_1'\cdot 2RT^2$ and $G_2=G_2'$. We end this section by showing that if we have \eqref{e.bound to prove}, then we will almost have our main Theorem~\ref{t.prob estimate for J}. More precisely, we have the following lemma, which, along with a proposition about the no-pole case, will allow us to prove Theorem~\ref{t.prob estimate for J} modulo proving these input statements.

\begin{lemma}\label{l.G_R integral term bound}
Suppose for all $(y,z)\in \mc G_R^{(1)}$ we have \eqref{e.bound to prove}. Then with $G_1$ and $G_2$ as in \eqref{e.bound to prove} we have
$$\PF\big(J(\,\cdot\,) - J(-d) \in A, (Y,Z)\in \mc G_R^{(1)}\big)\cdot\one_{\fav, P\cap[-2d,2d] \neq \emptyset} \lesssim G_1\cdot\epsilon\cdot \exp\left(G_2(\log \epsilon^{-1})^{5/6}\right).$$ 
\end{lemma}

\begin{proof}
Let $B$ be a Brownian motion begun at $p-4d$ according to the distribution $\mu$. We have on $\fav \cap \{P\cap[-2d,2d] \neq \emptyset\}$,
\begin{align*}
\PF\Big(J(\,\cdot\,) - J(-d) \in A,\, (Y,Z)\in \mc G_R^{(1)}\Big)
&\leq \iint_{\mc G_R^{(1)}} \frac{\B^{[p-4d, p+4d]}_{\tilde y, \tilde z}\big( \tilde A\big)}{\B^{[p-4d, p+4d]}_{\tilde y, \tilde z}\big( J(p) \geq \tent(p)\big)}\cdot f_J(y,z) \, \mathrm dy\, \mathrm dz
\end{align*}
Using \eqref{e.bound to prove}, this integral is bounded by
\begin{align*}
\MoveEqLeft[10] G_1d^{-\frac{1}{2}}\exp\left(G_2(\log\epsilon^{-1})^{5/6}\right)\iint_{\mc G_R^{(1)}} \B^{[p-4d, p+4d]}_{\tilde y, \tilde z}\big( \tilde A\big)e^{-\frac{1}{16d}(\tilde y-\tilde z)^2} \, \mathrm d\mu(y)\, \mathrm dz\\
&\lesssim \P\bigl(B(\,\cdot\,)-B(-d) \in A\bigr)\cdot G_1\cdot\exp\left(G_2(\log\epsilon^{-1})^{5/6}\right)\\
&= \B^{[-d,d]}_{0,*}(A)\cdot G_1\cdot\exp\left(G_2(\log\epsilon^{-1})^{5/6}\right)
= G_1\cdot \epsilon\cdot \exp\left(G_2(\log\epsilon^{-1})^{5/6}\right).
\end{align*}
The second-to-last equality follows from the Markov property of Brownian motion.
\end{proof}

In fact, we will establish \eqref{e.bound to prove} with $G_2=792\cdot d\cdot D_k^{5/2}$ and $G_1= G'RT^{4}$, where $G'$ is an absolute constant independent of $d,k,\epsilon$. This is the content of Lemmas \ref{l.y<0, z<0 case}, \ref{l.y+z>0 case}, and \ref{l.yz<0 case} ahead, as $792d\geq \max(22R^2, 41R)$ from $R$'s value set in \eqref{e.R value}, and the observation above that to go from \eqref{e.sufficient bound to prove} to \eqref{e.bound to prove} we must multiply $G_1'$ by $2RT^2$. We will also establish the following proposition in the no-pole case:

\begin{proposition}\label{p.prob estimate no pole}
There exists a positive constant $\epsilon_0 = \epsilon_0(d)>0$ such that if $\epsilon<\epsilon_0$, then, on $\fav \cap \{P\cap[-2d,2d]=\emptyset\}$,
$$\PF\!\Big(J(\,\cdot\,) - J(-d) \in A,\, (Y, Z) \in \mc G^{(2)}_R\Big) \lesssim \epsilon \cdot\exp\left(756\cdot d\cdot D_k^{5/2}\cdot(\log \epsilon^{-1})^{5/6}\right).$$
\end{proposition}

Admitting these statements for now, namely Lemmas~\ref{l.bad event 2}, \ref{l.bad event 1}, Proposition~\ref{p.prob estimate no pole}, and that we have \eqref{e.bound to prove} with $G_1 = G'RT^4$ and $G_2 = 792\cdot d\cdot D_k^{5/2}$, we may complete the proof of Theorem~\ref{t.prob estimate for J}.

\begin{proof}[Proof of Theorem \ref{t.prob estimate for J}]
The quantity $\P\left(J(\,\cdot\,) - J(-d)\in A\right)\cdot\one_\fav$ satisfies the following upper bound:
\begin{equation}\label{e.final prob split up}
\begin{split}
\MoveEqLeft[6]
\P\left(J(\,\cdot\,) - J(-d)\in A\right)\cdot\one_\fav\\
&\leq \left[\PF\Big(J(\,\cdot\,) - J(-d)\in A,\, (Y,Z)\in \mc G^{(1)}_R\Big) + \PF\Big((Y,Z)\not\in \mc G^{(1)}_R\Big)\right]\one_{\fav, P\cap[-2d,2d]\neq \emptyset}\\
&\quad+\left[\PF\Big(J(\,\cdot\,) - J(-d)\in A,\, (Y,Z)\in \mc G^{(2)}_R\Big)
+ \PF\Big((Y,Z)\not\in \mc G^{(2)}_R\Big)\right]\one_{\fav,P\cap[-2d,2d]= \emptyset}.
\end{split}
\end{equation}
Focus on the first term after the inequality of \eqref{e.final prob split up}. By Lemma \ref{l.bad event 2} and our choice of $R = 6\sqrt d$ from \eqref{e.R value}, 
$$\PF\Big((Y,Z)\not\in \mc G^{(1)}_R\Big)\cdot\one_{\fav,P\cap[-2d,2d]\neq \emptyset} \lesssim \epsilon \cdot\exp\left(468\cdot d\cdot D_k^{5/2}(\log\epsilon^{-1})^{5/6}\right).$$
From our assumption that we have \eqref{e.bound to prove} with $G_2=792\cdot d\cdot D_k^{5/2}$ and $G_1 = G'RT^{4}$ for all $(y,z)\in \mc G_R^{(1)}$, we get from Lemma \ref{l.G_R integral term bound}
\begin{equation}
\begin{split}\label{e.bound on G1}
\MoveEqLeft[5]
\PF\Big(J(\,\cdot\,) - J(-d)\in A,\, (Y,Z)\in \mc G^{(1)}_R\Big)\cdot \one_{\fav, P\cap[-2d,2d]\neq \emptyset}\\
&\lesssim \epsilon\cdot d^{\frac{1}{2}}\cdot D_k^{4} (\log \epsilon^{-1})^{\frac{4}{3}}\cdot \exp\left(792\cdot d\cdot D_k^{5/2}(\log \epsilon^{-1})^{5/6}\right)\\
\end{split}
\end{equation}
Now we turn to the second line of \eqref{e.final prob split up}. From Lemma \ref{l.bad event 1} and our choice of $R=6\sqrt d$ in \eqref{e.R value}, 
$$\PF\Big((Y,Z)\not\in \mc G^{(2)}_R\Big)\cdot\one_{\fav,P\cap[-2d,2d]= \emptyset} \lesssim \epsilon\cdot\exp\left(468\cdot d\cdot D_k^{5/2}(\log\epsilon^{-1})^{5/6}\right).$$
Finally from Proposition \ref{p.prob estimate no pole} we have
\begin{equation*}\label{e.bound on G2}
\PF\Big(J(\,\cdot\,) - J(-d) \in A,\, (Y, Z) \in \mc G^{(2)}_R\Big) \lesssim \epsilon\cdot\exp\left(756\cdot d\cdot D_k^{5/2}\cdot(\log \epsilon^{-1})^{5/6}\right).
\end{equation*}
Substituting these bounds into \eqref{e.final prob split up} gives
$$\PF\Big(J(\,\cdot\,) - J(-d) \in A\Big) \lesssim \epsilon\cdot d^{\frac{1}{2}}\cdot D_k^{4} (\log \epsilon^{-1})^{\frac{4}{3}}\cdot \exp\left(792\cdot d\cdot D_k^{5/2}\cdot (\log \epsilon^{-1})^{5/6}\right),$$
completing the proof.
\end{proof}

At this point we pause to review our progress. We have defined good regions $\mc G_R^{(1)}$ and $\mc G_R^{(2)}$ to which we can restrict our analysis. We have also defined the costs that need to be met on $\mc G_R^{(1)}$ in order to prove Theorem~\ref{t.prob estimate for J}, and indeed we have given the proof of Theorem~\ref{t.prob estimate for J} modulo these costs being met and a few additional statements being proven. Concretely, our task is now to establish \eqref{e.sufficient bound to prove} with the claimed values of $G_1'$ and $G_2'$; to prove Lemmas~\ref{l.bad event 2} and \ref{l.bad event 1}; and to prove Proposition~\ref{p.prob estimate no pole}. 

Establishing \eqref{e.sufficient bound to prove} will break into separate cases that will each require different arguments which will all be handled in the next Chapter~\ref{ch.the proof}. An easy and a moderate case, respectively when $y,z<0$ or $y+z>0$, will also supply us the bounds we need to prove Lemmas~\ref{l.bad event 2} and \ref{l.bad event 1}. Chapter~\ref{ch.the proof} will also address the case where a pole is not present in $[-2d,2d]$, i.e., Proposition~\ref{p.prob estimate no pole}.

\chapter{Proving the density bounds}\label{ch.the proof}

This chapter proves the statements of Chapter~\ref{ch.proof framework} that are needed in the proof of Theorem~\ref{t.prob estimate for J} as given in Section~\ref{s.bound to be proved}; in particular, we prove here Lemmas~\ref{l.bad event 2} and \ref{l.bad event 1}, Proposition~\ref{p.prob estimate no pole}, and that equation \eqref{e.bound to prove} holds with the claimed constants. In essence, these all follow from the last item, i.e., \eqref{e.bound to prove}, which is a bound on the density $f_J(y,z)$. 

The proof of this density bound is broken up into four sections. The first three are when a pole is present in $[-2d,2d]$ and are distinguished by the values of $Y= J(p-4d)-\tent(p-4d)$ and $Z=J(p+4d)-\tent(p+4d)$, roughly corresponding to an easy case of being below the pole on both sides (Section~\ref{s.easier cases}); a moderate case of being above the pole on both sides (Section~\ref{s.moderate case}); and a difficult case of being above and below the pole on either side (Section~\ref{s.difficult case}). Lemmas~\ref{l.bad event 2} and \ref{l.bad event 1} are proved at the end of the moderate case, Section~\ref{s.moderate case}. The last section, Section~\ref{s.no pole}, addresses when there is no pole in $[-2d, 2d]$, i.e., Proposition~\ref{p.prob estimate no pole}.

\section{The easy case: Below the pole on both sides}\label{s.easier cases}

Having set up the problem and identified what bounds \eqref{e.sufficient bound to prove} we require on $f_J(y,z)$, we now prove such a bound in a simple case. This case is when $y<0$ and $z<0$, which implies, from \eqref{e.S is approx S*} and \eqref{e.jcost approximate exponential form}, that 
$$\jcost\cdot\scost \lesssim \exp\left(\frac{y^2}{8d}+\frac{z^2}{8d}+21RT^{5/2}\right).$$
So, it is sufficient to prove
\begin{equation} \label{e.bound in both negative case}
f_J(y,z) \lesssim d^{-1}\cdot\exp\left(-\frac{y^2}{8d} - \frac{z^2}{8d}\right).
\end{equation}
In this case we are aided by the presence of the pole. Essentially, the desired density of $J$ at $(y,z)$ is bounded by the density of a particular pair of independent Brownian bridges at $(y,z)$. More precisely, on the event $\fav\cap\{P\cap[-2d,2d]\neq \emptyset\}$, let $p^-$ and $p^+$ be the elements of $P$ immediately preceding and succeeding $p$; and define
\begin{align*}
\sigma_{-4d}^2 &= 4d\cdot\frac{(p-p^--4d)}{p-p^-}\\
\sigma_{4d}^2 &= 4d\cdot\frac{(p^+-p-4d)}{p^+-p}.
\end{align*}
The first quantity is the variance of a Brownian bridge defined on the interval $[p^-, p]$ at the point $p-4d$, while the second is the same for a Brownian bridge defined on $[p, p^+]$ at the point $p+4d$. 

The following proposition is exactly the case which arose in \cite{hammond2017brownian} in the analysis of the Brownian bridge regularity of regular ensembles, where it was Proposition 5.17. The proof is fairly straightforward and we will shortly reproduce it here for completeness and because it aids our exposition.

\begin{proposition} \label{p.negative density}
We have that for $y,z <0$,
\begin{enumerate}
\item Joint bound
$$f_J(y,z)\cdot\one_{\fav,P\cap[-2d,2d]\neq\emptyset} \lesssim d^{-1}\cdot\exp\left(-\frac{1}{2\sigma_{-4d}^2}y^2-\frac{1}{2\sigma_{4d}^2}z^2\right) \cdot \one_{\fav, P\cap[-2d,2d]\neq\emptyset}.$$

\item Marginal bounds
\begin{align*}
f_J(y)\cdot\one_{\fav, P\cap[-2d,2d]\neq\emptyset} &\lesssim d^{-\frac{1}{2}}\cdot\exp\left(-\frac{1}{2\sigma_{-4d}^2}y^2\right) \cdot \one_{\fav, P\cap[-2d,2d]\neq\emptyset}\\
f_J(z)\cdot \one_{\fav,P\cap[-2d,2d]\neq\emptyset} &\lesssim d^{-\frac{1}{2}}\cdot\exp\left(-\frac{1}{2\sigma_{4d}^2}z^2\right) \cdot \one_{\fav, P\cap[-2d,2d]\neq\emptyset}.
\end{align*}

\end{enumerate}
\end{proposition}

We note that we have the following simple bounds on the variances:

\begin{lemma}\label{l.variance bounds}
On $\fav\cap\{P\cap[-2d,2d]\neq \emptyset\}$, we have
$\sigma_{4d}^2, \sigma_{-4d}^2 \in [\frac{4}{5}d,4d].$
\end{lemma}

\begin{proof}
The upper bound is obvious from the defining expressions. For the lower bound, we have
\begin{align*}
\sigma_{-4d}^2 = 4d\left(1-\frac{4d}{p-p^-}\right) \geq 4d\left(1-\frac{4d}{5d}\right) = \frac{4}{5}d,
\end{align*}
since $p-p^- \geq \dip = 5d$. A similar argument proves the corresponding bound for $\sigma^2_{4d}$.
\end{proof}

With these variance bounds and the density bounds of Proposition~\ref{p.negative density}, the sufficient bound \eqref{e.sufficient bound to prove} is immediate:

\begin{lemma}\label{l.y<0, z<0 case}
When $y<0$ and $z<0$, we have \eqref{e.sufficient bound to prove} with $G_2'=21RD_k^{5/2}$ and with $G_1'$ independent of $\epsilon, k,$ and $d$.
\end{lemma}

\begin{proof}
From \eqref{e.S is approx S*} and \eqref{e.jcost approximate exponential form}, the total cost is bounded above as
$$\jcost\cdot\scost \lesssim d\cdot \exp\left(\frac{y^2}{8d}+\frac{z^2}{8d} + 21RT^{5/2}\right)$$
 Proposition~\ref{p.negative density} combined with Lemma~\ref{l.variance bounds} says that
$f_J(y,z) \lesssim d^{-1}\exp\left(-\frac{y^2}{8d}-\frac{z^2}{8d}\right)$,
so we are done.
\end{proof}

\begin{proof}[Proof of Proposition \ref{p.negative density}]
	The second statement follows from the first by integrating out one of the variables, so we prove only the first statement on the joint density bound.
	We are in the situation where $P\cap[-2d,2d]\neq \emptyset$, and $p$ is the unique element in this intersection. Let $p^-$ and $p^+$ be the adjacent elements of $\pole$. Let $\mathcal{F}[p^-,p,p^+]$
	denote the $\sigma$-algebra generated by $\mathcal{F}$ and the random variables $J(x)$ for $x \in \{ p^-,p,p^+ \}$. (These random variables provide extra information only when $\pole\cap[-2d,2d] \neq \emptyset$.) The density $f_J(y,z)$ has a counterpart $f_J^{\F[p^-, p, p^+]}$ under the augmented $\sigma$-algebra, and it is enough to show that
	$$
	 f^{\mathcal{F}[p^-,p,p^+]}_J(y,z) \cdot {\one}_{\pole \cap [-2d,2d] \neq \emptyset } \, \lesssim \, \sigma_{-4d}^{-1}\cdot\sigma_{4d}^{-1}\cdot \exp\left(-\frac{y^2}{2\sigma_{-4d}^2} - \frac{z^2}{2\sigma_{4d}^2}\right) \, ,
	$$
	since then Proposition~\ref{p.negative density}(1) will arise by averaging.

	Under the law $\P_{\mathcal{F}[p^-,p,p^+]}$, the processes $J(\,\cdot\,)$ on $[p^-,p]$ and $[p,p^+]$ are conditionally independent. Since the data in $\mathcal{F}[p^-,p,p^+]$ causes $\{ \pole \cap [-2d,2d] \neq \emptyset \}$ to occur, it is thus enough to argue that 
	\begin{itemize}
	\item the conditional density of $Y$ at $s \leq 0$ is at most a constant multiple of $\sigma_{-4d}^{-1} \exp\left(-s^2/2\sigma_{-4d}^2\right)$;
	\item and 
	the conditional density of $Z$ at $t \leq 0$ is at most $\sigma_{4d}^{-1} \exp\left(-t^2/2\sigma_{4d}^2\right)$.
	\end{itemize}
	These statements are straightforward to verify.
	Indeed, the conditional law under $\mathcal{F}[p^-,p,p^+]$ of $J(p-4d)$ is normal with mean
	$\left(1-\tfrac{4d}{p - p^-}\right) J(p) + \tfrac{4d}{p - p^-} J(p^-)$ and variance $\sigma_{-4d}^2$. Note that $J(p^-) \geq \tent(p^-)$ and $J(p) \geq \tent(p)$ since $p^-, p \in \pole$, and that $\tent$ is affine on the interval between consecutive pole set elements $p^-$ and $p$; thus, we see that this mean is at least $\tent(p-4d)$. The first bullet point statement follows from the form of the normal density since we have shown that $\E_{\F[p^-,p,p^+]}[Y] \geq 0$, and we are concerned with the density only on $(-\infty, 0]$. The second bullet point is proved in the same fashion. This proves Proposition~\ref{p.negative density}(1).
\end{proof}

\section{The moderate case: Above the pole on both sides}\label{s.moderate case}
In this section, we address the case of bounding $f_J(y,z)$ when $y+z>0$. Here is the main proposition to be proved.

\begin{proposition}[Density bound on increment]\label{p.scost density}
Let $\mf l/2<x_1 < x_2 < \mf r/2$ be $\F$-measurable, $\sigma^2 = x_2-x_1\geq d$, $R$ be as in \eqref{e.R value}, and suppose that $\sigma\leq T^2\cdot \sqrt{d/2}$. Suppose also that $[x_1-d, x_1+d]\cap P  = \emptyset$ and $[x_2-d, x_2+d]\cap P = \emptyset$. Let $f_J^{x_1,x_2}(s,t)$ be the joint density of $(J(x_1) - \tent(x_1), J(x_2)-\tent(x_2))$ at $(s,t)$. If $s, t \in [-RT^2, RT^2]$ and $|s-t|>6RT\sigma^2$, then
$$f^{x_1,x_2}_J(s, t) \cdot \one_{\fav} \lesssim \sigma^{-1}\cdot d^{-\frac{1}{2}} T^2\cdot \exp\left(-\frac{1}{2\sigma^2}\bigl(|s-t|- 6RT\sigma^2\bigr)^2 + \frac{4R^2T^2}{\sigma^2}+36R^2\sigma^2\right).$$
We also have that $f_J^{x_1,x_2}(s, t) \lesssim d^{-1}$ for all $s, t\in\R$. 
\end{proposition}

\begin{figure}
\centering {\epsfig{file=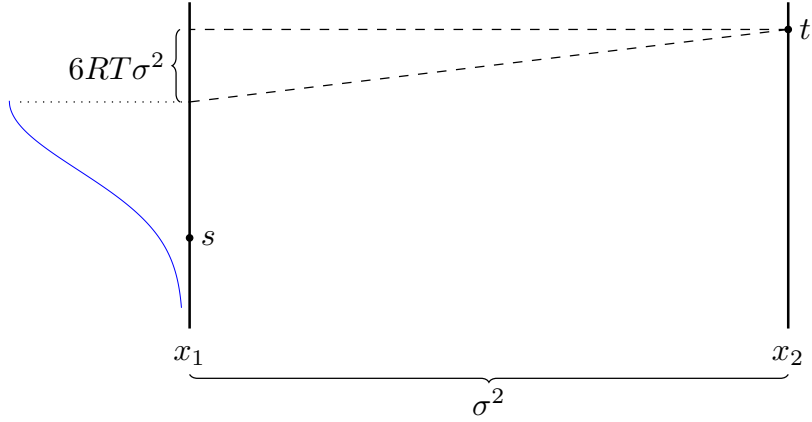, width=0.67\textwidth}}
\caption{Illustrating Proposition~\ref{p.scost density} when $s<t$. The requirement that $|s-t|\geq 6RT\sigma^2$ comes from an error in an estimate of a certain slope; the blue curve on the left shows the Gaussian-like density centred at $t-6RT\sigma^2$ that dominates the joint density $f_J^{x_1,x_2}(s,t).$ Note that the proposition does not require or use the presence of a pole in $[x_1,x_2]$.}\label{f.moderate case}
\end{figure}

\begin{remark}
The condition that $|s-t|>6RT\sigma^2$ arises in the proof of Proposition~\ref{p.scost density} from the error of an estimate on certain slopes; see Figure~\ref{f.moderate case}. For our purposes it does not cause any difficulty, as when $|y-z|<6RT\sigma^2$ the costs will be always absorbable in the leeway factor. For example, this is seen in the proof of Lemma~\ref{l.y+z>0 case} below.
\end{remark}

It may not be immediately clear what is the relation of this proposition to the case where $y+z>0$. In fact, this proposition has been carefully stated to apply to a more general situation than just the case of this section. For example, unlike Proposition~\ref{p.negative density}, this proposition does not require a pole to be present, and we will make use of it in the no-pole case addressed in Section~\ref{s.no pole} as well. We will also use Proposition~\ref{p.scost density} to prove Lemmas~\ref{l.bad event 2} and \ref{l.bad event 1} near the end of the section, as it provides a density bound in terms of the \emph{increment} $|y-z|$, and so can be easily used to bound the increment of $J$ across an interval, a requirement which was briefly discussed in Section~\ref{s.heuristics for proof}.

Before proving Proposition~\ref{p.scost density}, we apply it to show that it yields the sufficient bound \eqref{e.sufficient bound to prove} in the case that $y+z>0$. To see that Proposition~\ref{p.scost density} is sufficient for this purpose, note that by \eqref{e.jcost simplified form}, when $y+z>0$, $\jcost$ can be absorbed in the leeway factor, and so we essentially only need to consider $\scost$. This cost, being $d^{\frac{1}{2}}\exp((y+z)^2/16d)$ up to the leeway factor, is in essence met by the density bound provided by Proposition~\ref{p.scost density} by taking $x_1 = p-4d$ and $x_2=p+4d$; this is consistent with the ideas that Proposition~\ref{p.scost density} controls the increment and that the slope cost $\scost$ is a cost associated with large increments.

\begin{lemma}\label{l.y+z>0 case}
When $y+z>0$ and $(y,z)\in \mc G_R^{(1)}$, we have \eqref{e.sufficient bound to prove} with $G_2'=22R^2D_k^{5/2}$ and $G_1'=G'T^2$, where $G'$ is a constant independent of $\epsilon, k,$ and $d$.
\end{lemma}

\begin{proof}
	In this case, we see that
	$$\jcost \lesssim \begin{cases}
	\, 1 & y+z > 32Td\\
	\, d^{\frac{1}{2}}\cdot\exp\left(64T^2d\right) & 0< y+z < 32Td,
	\end{cases}$$
	i.e., in this case $\jcost$ is bounded by the leeway factor $d^{\frac{1}{2}}\cdot\exp\left(G_2T^{5/6}\right)$ with $G_2=8$, since $d\leq \sqrt T/24$ by the assumption $\epsilon \leq \exp(-(24)^{6}d^6/D_k^3)$ (note that $8 < 22R^2$ since $R\geq 1$ from \eqref{e.R value}). So we merely need to handle $\scost$, which from \eqref{e.S is approx S*} is bounded above as
	$$S\lesssim d^{\frac{1}{2}}\cdot \exp\left(\frac{1}{16d}(y-z)^2+ 10RT^{5/2}\right).$$
	Note that if we set $x_1 = p-4d$ and $x_2 = p+4d$ in Proposition~\ref{p.scost density}, then $\sigma^2 = 8d$. So with these parameters, we see from the second part of Proposition~\ref{p.scost density} that, when $|y-z|< 48RTd$,
	\begin{align*}
	f_J(y, z) \cdot \jcost \cdot \scost \lesssim \exp\left(\frac{1}{16d}|y-z|^2 + 64T^2d + 10RT^{5/2}\right) &\leq \exp\left((48\times3R^2+64)T^2d + 10RT^{5/2}\right)\\
	&\leq \exp\left(19R^2 T^{5/2}\right),
	\end{align*}
	the last inequality since $R\geq 1$ and again using that $d\leq \sqrt T/24$; we have also used that $64/24 \leq 3$.

	When $|y-z|>48RTd$, we again make use of Proposition~\ref{p.scost density} with $x_2 = p+4d$ and $x_1=p-4d$. The assumptions of Proposition~\ref{p.scost density} are satisfied since $\dip=5d$. So we obtain the bound
	\begin{align*}
	f_J(y, z)\cdot\one_\fav 
	&\lesssim d^{-1}T^2\exp\left(-\frac{1}{16d}(|y-z| - 48RTd)^2 + \frac{4R^2T^2}{8d} + 36R^2\times 8 d \right)
	\end{align*}
	when $|y-z| > 48RTd$. So, for $|y-z|>48RTd$, 
	\begin{align*}
	f_J(y, z) \cdot \jcost\cdot \scost &\lesssim T^2\cdot\exp\left(6RT|y-z| - 48\times3R^2T^2d+ \frac{1}{2}R^2T^2 +36R^2\times8d+ 64T^2d+10RT^{5/2}\right) \\
	&\leq T^2\cdot\exp\left(12R^2T^{5/2}+10RT^{5/2}\right)
	= T^2\cdot\exp\left(22R^2T^{5/2}\right),
	\end{align*}
	the last inequality since we have $|y-z| \leq 2RT^{3/2}$ on $\mc G_R^{(1)}$ and using that $1\leq d\leq \sqrt T/24$ and $R\geq 1$ from \eqref{e.R value} to see that $36R^2\times 8 d \leq 12R^2T^{1/2}$, allowing us to drop the middle four terms. This verifies \eqref{e.sufficient bound to prove} with the claimed values.
\end{proof}

We may now turn to discussing the proof strategy of Proposition~\ref{p.scost density}. In the proof of Proposition~\ref{p.negative density}, we were greatly aided by the presence of the pole and the \emph{difficulty} that a Brownian bridge faces in making large jumps while remaining negative on either side. An interesting feature of that argument is that no extra reasoning was needed to obtain a density bound from a comparison with a Brownian object. Typically, such comparisons easily yield bounds on tail probabilities, but these do not immediately imply a pointwise density bound. 

In the proof of Proposition~\ref{p.scost density}, as well as in the case addressed in Section~\ref{s.difficult case}, both of these features will be missing. Firstly, when $J$ (still thought of as essentially a Brownian bridge) is allowed to be positive on one side, the pole \emph{assists} its attainment of the values on either side, and so it is not clear why the Brownian density should bound that of the jump curve. This is especially true for the case analysed in the next Section ~\ref{s.difficult case}, where essentially the same bound as that proved in Section~\ref{s.easier cases} must hold. Secondly, we will not be able to access density bounds directly, but will need to make further technical arguments to move from tail probability bounds to density bounds. To accomplish the latter, we will make use of a technique of \emph{local randomization}, in both the case of this section as well as that of Section~\ref{s.difficult case}.

A short description of what we mean by local randomization is the following: Suppose that we wish to obtain a bound on the density of $J(x)$ for some $x$. We first obtain bounds on the tail probabilities of $J$ at certain points, say $x -\eta$ and $x+\eta$ for some $\eta>0$, with no pole contained in $[x-\eta, x+\eta]$. To convert this to a density bound at $x$, we use that $J$, conditionally on its values at $x-\eta$ and $x+\eta$, is a Brownian bridge on $[x-\eta,x+\eta]$. Then the distribution of $J$ at $x$ can be written as a convolution of the distributions at $x-\eta$ and $x+\eta$ with a normal random variable, which, when combined with the tail bounds, can be used to give a density bound. Heuristically, the tail bound is being propagated and smoothed by the Brownian bridge to a density bound.

At the level of this description, no importance is given to the exact value of $\eta$ and we have not explained what we mean by ``local" in local randomization. To aid our discussion, let us say that a random variable $X$ has a \emph{pseudo-variance at most} $\sigma^2$ if we have a tail bound of the form $\P(X<t) \leq \exp(-t^2/2\sigma^2)$. As the discussion of $\jcost$ and $\scost$ in Section~\ref{s.heuristics for proof} showed, we require the pseudo-variance we obtain in sub-Gaussian density bounds to be essentially optimal. And indeed, if we knew that the distributions of $J(x-\eta)$ and $J(x+\eta)$ were actually Gaussian, then the density bound arising from the convolution mentioned would be precisely the correct one. However, when all we have is a sub-Gaussian tail bound and not an actual Gaussian distribution, there is some extra gain in the pseudo-variance we obtain for the final density bound. This is captured in the following lemma, whose proof will be given at the end of the section.

\begin{lemma}\label{l.density bound tool lower}
Let $X$ be a random variable such that $\P(X<x)\leq A\exp(-\frac{1}{2\sigma_2^2}(x-x_0)^2)$ for $x<x_0$, and let $N$ be a normal random variable with mean 0 and variance $\sigma_1^2$ which is independent of $X$. Then the density $f$ of $X+N$ satisfies
\begin{align}\label{e.density bound tool}
f(x) \leq \frac{A+1}{\sqrt{2\pi}\sigma_1}\cdot\exp\left(-\frac{(x-x_0)^2}{2(\sigma_1+\sigma_2)^2}\right)
\end{align}
for $x<x_0$, and is bounded by $1/\sqrt{2\pi}\sigma_1$ for all $x\in\R$.
\end{lemma}

The point we were expressing is seen in this formula by the fact that the pseudo-variance guaranteed by this bound is $(\sigma_1+\sigma_2)^2$, which is greater than $\sigma_1^2+\sigma_2^2$ as it would be had we known, in the notation of the lemma, that $X$ has Gaussian distribution with variance $\sigma_2^2$.

For $\eta$ a positive constant, this gain in pseudo-variance gives a density bound that is too weak for our purposes; indeed, the bound is weaker than that claimed in Proposition~\ref{p.scost density}. The solution is, roughly, to take $\eta\to 0$. In the language of the lemma, if $\sigma_2(\eta) \to \sigma$ and $\sigma_1(\eta) \to 0$ as $\eta\to0$, then
$$\lim_{\eta\to 0} (\sigma_1 +\sigma_2)^2 = \sigma^2 = \lim_{\eta\to 0} (\sigma_1^2+\sigma_2^2)\, ;$$
i.e., there is no gain in pseudo-variance in the limit. However, taking $\eta\to 0$ leads to a blow up in the constant in front of the exponential in \eqref{e.density bound tool}, and so we actually take $\eta$ to be a small $\epsilon$-dependent quantity, small enough that the gain in pseudo-variance is manageable. This is the argument of Proposition~\ref{p.scost density}, which we turn to next. We will then give the pending proofs of Lemmas~\ref{l.bad event 2} and \ref{l.bad event 1}, and finish the section by proving the technical tool Lemma~\ref{l.density bound tool lower}.

%
%
%

The proof of Proposition~\ref{p.scost density} will actually obtain the claimed bound on the conditional density $f_J^{x_1,x_2}(s\mid t)$, so we also need that the marginal density $f_J^{x_2}(t)$ is bounded.

\begin{lemma}[Marginal density is bounded]\label{l.density is bounded}
For any $\F$-measurable $x\in[-2T,2T]$ such that $[x-d, x+d]\cap P = \emptyset$, let $f_J^x$ be the density of $J(x)-\tent(x)$ conditionally on $\F$. Then we have on $\fav$ that $f_J^x(s) \leq \pi^{-\frac{1}{2}}d^{-\frac{1}{2}}$ for all $s\in \R$.
\end{lemma}

\begin{proof}
	By assumption there is no pole in $[x-d, x+d]$. So $J(x)$, conditionally on $J(x-d)$ and $J(x+d)$, is given by
	$$J(x) = \frac12 J(x-d)+\frac12J(x+d) + N\left(0,\frac12d\right).$$
	Thus a formula for $f_J^x$ is
	$$f_J^x(s) = \frac{1}{\sqrt{\pi d}}\int_{-\infty}^\infty \exp\left(-\frac{1}{d}(t-s)^2\right)\, \mathrm d\nu(t),$$
	where $\nu$ is the law of $\frac12 J(x-d)+\frac12J(x+d)$. 
	From this formula the claim follows.
\end{proof}

We next cite a standard bound on normal probabilities before turning to the proof of Proposition~\ref{p.scost density}.

\begin{lemma}[Normal bounds] \label{l.normal bound}
Let $\sigma^2>0$. If $t>\sigma$ for the first inequality and $t>0$ for the second,
$$\frac{\sigma}{2\sqrt{2\pi}t}\exp\left(-\frac{t^2}{2\sigma^2}\right)\leq \P\Big(N(0,\sigma^2) > t\Big) \leq \exp\left(-\frac{t^2}{2\sigma^2}\right).$$
\end{lemma}

\begin{proof}
Replacing $t$ by $\sigma t$, it suffices to take $\sigma = 1$. The standard lower bound
$$
 \P(N(0,1) > t) \geq (2 \pi)^{-1 / 2} \frac{t}{t^{2}+1}\exp \left(-t^{2} / 2\right)
$$
for $t\geq 0$ may be found in \cite[Section 14.8]{williams1991probability}. Note that $\frac{t}{t^2+1}\geq (2t)^{-1}$ for $t\geq 1$.
The upper bound is simply the Chernoff bound.
\end{proof}



The proof of Proposition~\ref{p.scost density} has two steps, as described in the earlier discussion. The first is a tail bound on a quantity close to $J(x_1)$, conditionally on $J(x_2)$ (and with the roles of $x_1$ and $x_2$ reversed); the second is to convert this tail bound into a density bound using Lemma~\ref{l.density bound tool lower}. The first step is isolated in the next lemma, while the second step is performed in the immediately following proof of Proposition~\ref{p.scost density}.

\begin{lemma}\label{l.scost tail bound}
Let $\mf l/2<x_1 < x_2 < \mf r/2$ be $\F$-measurable, $\sigma^2 = x_2-x_1\geq d$, and $R$ be as in \eqref{e.R value}. Suppose also that $[x_1-d, x_1+d]\cap P  = [x_2-d, x_2+d]\cap P = \emptyset$. Then on $\fav$, for $r<t+\tent(x_2) - 4\sigma^2(R+2)T$ and any $\eta < d/2$,
\begin{align*}
\MoveEqLeft[15]
\PF\Big(\tfrac12 J(x_1+2\eta) + \tfrac12 J(x_1-2\eta) < r \ \big|\  J(x_2)+\tent(x_2) = t\Big)&\\
&\leq \exp\left(-\frac{1}{2\sigma_\eta^2}\big(r-t-\tent(x_2)+ 4\sigma^2(R+2)T\big)^2\right),
\end{align*}
where $\sigma_\eta^2 = \sigma^2\cdot\frac{x_1-\mf l}{x_2-\mf l} -\eta$.

Similarly, for $r<t+\tent(x_1)-4\sigma^2(R+2)T$ and on $\fav$,
\begin{align*}
\MoveEqLeft[15]
\PF\Big(\tfrac12 J(x_2+2\eta) + \tfrac12 J(x_2-2\eta) < r \ \big|\  J(x_1)+\tent(x_1) = t\Big)&\\
&\leq \exp\left(-\frac{1}{2\tilde\sigma_\eta^2}\big(r-t-\tent(x_1)+ 4\sigma^2(R+2)T\big)^2\right),
\end{align*}
where $\tilde\sigma_\eta^2 = \sigma^2\cdot\frac{\mf r-x_2}{\mf r- x_1}-\eta$.
\end{lemma}

\begin{proof}
We will only prove the first bound as the second bound is analogous, by repeating the below argument with the roles of $x_1$ and $x_2$ switched, and $\mf r$ in place of $\mf l$.

	By assumption, there is no pole in $[x_1 -2\eta, x_1+2\eta]$ for all $\eta \leq d/2$. For every such $\eta$, the distribution of $J(x_1)$ given $J(x_2)$, $J(x_1-2\eta)$, and $J(x_1+2\eta)$ depends on only $J(x_1-2\eta)$ and $J(x_1 +2\eta)$, and is given by
	\begin{equation}
	J(x_1) = \frac{1}{2}J(x_1-2\eta) + \frac{1}{2}J(x_1+2\eta) + N\left(0, \eta\right). \label{e.dist of lower point}
	\end{equation}
	Given $J(x_2)=t+\tent(x_2)$ and on $\fav$, Lemma~\ref{l.J stochastic dominanation}(ii) implies that $J$, restricted to $[\mf l, x_2]$, stochastically dominates the Brownian bridge with endpoints ${(\mf l, -T^2)}$ and $(x_2, t+\tent(x_2))$. We call this Brownian bridge $B$. Then the slope of the line connecting these two points is
	$$m:=\frac{t+\tent(x_2)+T^2}{x_2-\mf l},$$
	and so $\E[B(x_1+r)] = t+\tent(x_2)-(x_2-x_1-r)m$ for any $r$ such that $x_1+r\in[\mf l,x_2]$.
	Now, conditionally on $J(x_2)$, we have a coupling such that
	\begin{align}\label{e.scost tail bound stoch dom}
	\frac12 J(x_1 + 2\eta) + \frac12 J(x_1-2\eta) \geq \frac12 B(x_1+2\eta) + \frac12 B(x_1-2\eta).
	\end{align}
	%
	%
	%
	Since the covariance of $B$ is given for $r_1\leq r_2$ by
	$$\Cov(B(r_1),B(r_2)) = \frac{(r_1-\mf l)(x_2-r_2)}{x_2 -\mf l},$$
	it follows after some algebraic simplification that the variance of the right hand side of \eqref{e.scost tail bound stoch dom} is $\sigma_\eta^2$. 
	%
	The mean of $\frac12 B(x_1+2\eta) + \frac12 B(x_1-2\eta)$ is
	$$t+\tent(x_2)-\tfrac12(x_2-x_1-2\eta)m - \tfrac12(x_2-x_1+2\eta)m = t+\tent(x_2) - \sigma^2m.$$
	Thus we have that, for $r<t+\tent(x_2)-\sigma^2 m$, on $\fav\cap\{[x_1-d,x_1+d]\cap P = \emptyset\}$,
	\begin{align}
	\PF\Big(\tfrac12 J(x_1+2\eta) &+ \tfrac12 J(x_1-2\eta) < r \ \big|\  J(x_2)+\tent(x_2) = t\Big)\nonumber\\
		&\leq \PF\Big(N\big(t+\tent(x_2)-\sigma^2m, \sigma_\eta^2\big )< r\Big)
		%
		%
		\leq \exp\left(-\frac{1}{2\sigma_\eta^2}\big(r-t-\tent(x_2)+ \sigma^2 m\big)^2\right),\label{e.scost first tail bound}
	\end{align}
	the last inequality obtained for $r< t+\tent(x_2)-\sigma^2m$ via the upper bound from Lemma~\ref{l.normal bound}. 

	Now returning to the definition of $m$, on $\fav$,
	\begin{align*}
	m &= \frac{t+\tent(x_2)+T^2}{x_2-\mf l} \leq \frac{(R+2)T^2}{T/4} = 4\left(R+2\right)T,
	\end{align*}
	since we have assumed that $t\leq RT^2$; that $x_2 \geq \mf l/2$; and since, on $\fav$, $\mf l\leq -T/2$ and $\tent(x_2)\leq T^2$. 

	Using this bound on $m$ in \eqref{e.scost first tail bound} completes the proof of Lemma~\ref{l.scost tail bound}.
\end{proof}

\begin{proof}[Proof of Proposition~\ref{p.scost density}]
	We prove only the case of $s<t$; the other case is analogous, making use of the second inequality of Lemma~\ref{l.scost tail bound} instead of the first as we do in the case of $s<t$.

	We first note that
	\begin{align}\label{e.sigma_eta}
	\sigma_\eta^2 = \frac{(x_2-x_1)(x_1-\mf l)}{x_2-\mf l} - \eta \leq \sigma^2 - \eta.
	\end{align}
	We will apply Lemma~\ref{l.density bound tool lower} to Lemma~\ref{l.scost tail bound} using \eqref{e.dist of lower point}. The parameters of Lemma~\ref{l.density bound tool lower} are set as follows (the formal notational conflict between $\sigma_1$ or $\sigma_2$ and $\sigma_\eta$ should not cause confusion): $X=\frac{1}{2}J(x_1-2\eta)+\frac{1}{2}J(x_1+2\eta)-\tent(x_1)$, $\sigma_1^2 = \eta$, $\sigma_2^2 = \sigma_\eta^2$, $x_0 = t+\tent(x_2)-\tent(x_1) - \sigma^2m$, and $A$ specified by the constant represented by $\lesssim$ in the first inequality of Lemma~\ref{l.scost tail bound}. 

	This yields, on the event $\fav\cap\{[x_1-d,x_1+d]\cap P = \emptyset\}$, for each $\eta < d/2$, the following bound on the conditional density of $J(x_1)-\tent(x_1)$, conditionally on $J(x_2)-\tent(x_2) = t$:
	\begin{align*}
	f^{x_1,x_2}_J(s\mid t)  &\lesssim \eta^{-\frac{1}{2}}\cdot \exp\left(-\frac{1}{2\left(\sigma_\eta+ \eta^{1/2}\right)^2}\left(s-t+\tent(x_1)-\tent(x_2)+4\sigma^2(R+2)T\right)^2\right)\\
	&\leq \eta^{-\frac{1}{2}}\cdot \exp\left(-\frac{1}{2\left(\sigma^2 + 2\eta^{1/2}\sigma\right)}\left(s-t+\tent(x_1)-\tent(x_2)+4\sigma^2(R+2)T\right)^2\right)
	\end{align*}
	for $s<t+\tent(x_2)-\tent(x_1)-4\sigma^2(R+2)T$. We used \eqref{e.sigma_eta} when expanding the square in the denominator of the exponent in the last inequality. 
	Note that 
	$$\Big|\tent(x_1)-\tent(x_2)\Big| \leq 4T(x_2-x_1) = 4T\sigma^2.$$
	Using the previous equation, we obtain, on $\fav\cap\{[x_1-d,x_1+d]\cap P = \emptyset\}$, that
	\begin{align*}
	f^{x_1,x_2}_J(s\mid t) &\lesssim \eta^{-\frac{1}{2}}\cdot \exp\left(-\frac{1}{2\left(\sigma^2 + 2\eta^{1/2}\sigma\right)}\left(s-t+MT\sigma^2\right)^2\right),
	\end{align*}
	for $s\leq t-MT\sigma^2$, where $M= 4+ 4(R+2) = 4(R+3).$ Now using the inequality $(1+x)^{-1}\geq 1-x$ for $x = 2\eta^{1/2}\sigma^{-1}$, we find
	\begin{equation}\label{e.scost density before expanding exponent}
	f^{x_1,x_2}_J(s\mid t) \lesssim \eta^{-\frac{1}{2}}\cdot \exp\left(-\frac{1}{2\sigma^2}\big(s-t+MT\sigma^2\big)^2 +\frac{\eta^{1/2}}{ \sigma^3}\big(s-t+MT\sigma^2\big)^2\right).
	\end{equation}
	Let us focus on bounding the second term in the exponent. We expand the square and drop the cross-term, since $s-t\leq 0$, to get that the second term is bounded above by
	$$\frac{\eta^{1/2}}{\sigma^3}\bigl((s-t)^2 + M^2T^2\sigma^4\bigr)\leq \frac{\eta^{1/2}}{\sigma^3}\bigl(4R^2T^4 + M^2T^2\sigma^4\bigr),$$
	the last inequality since $s-t\in[-2RT^2, 0]$. We now use this bound in \eqref{e.scost density before expanding exponent} and set $\eta^{1/2} = T^{-2}\sigma$ (which satisfies $\eta < d/2$ by assumption), to obtain
	$$f^{x_1,x_2}_J(s\mid t) \lesssim \frac{T^2}{\sigma} \exp\left(-\frac{1}{2\sigma^2}(s-t+MT\sigma^2)^2 +\frac{4R^2T^2}{\sigma^2} + M^2\sigma^2\right).$$
	The argument is complete by noting that $f^{x_2}_J(t)\lesssim d^{-\frac{1}{2}}$ by Lemma~\ref{l.density is bounded} and that $M\leq 6R$, since $M=4R+12\leq 4R+2\times 6\sqrt d = 6R$ from \eqref{e.R value} and $d\geq 1$.
%
%
%
%

	The final statement in Proposition~\ref{p.scost density} of a constant bound on $f_J^{x_1,x_2}(s,t)$ for all values of $s$ and $t$ follows immediately from \eqref{e.dist of lower point} and the latter assertion of Lemma~\ref{l.density bound tool lower} with the parameters $\sigma_1^2 = \eta = d/4$, and again using that the marginal density satisfies $f^{x_2}_J(t)\lesssim d^{-\frac{1}{2}}$ from Lemma~\ref{l.density is bounded}.
\end{proof}

We will now move towards the proofs of Lemmas~\ref{l.bad event 2} and \ref{l.bad event 1}, which use Propositions \ref{p.negative density} and \ref{p.scost density}. Then we will conclude this section and this part of the argument by proving Lemma~\ref{l.density bound tool lower}.

For the proofs of Lemmas~\ref{l.bad event 2} and \ref{l.bad event 1} we will need two further statements, the first bounding certain Gaussian integrals, and the other a standard tail bound on the supremum of a Brownian bridge. These are the next two lemmas. We will make frequent use of Lemma~\ref{l.integral bound} in the next section as well.

\begin{lemma}\label{l.integral bound}
For $a>0$,
$$\int_0^\infty \exp\left(-a x^2 + bx\right)\, \mathrm dx \lesssim \begin{cases}
a^{-\frac{1}{2}} & b\leq 0\\
a^{-\frac{1}{2}} \exp\left(\frac{b^2}{4a}\right) & b \in \R.

\end{cases}$$
\end{lemma}

\begin{proof}
	Completing the square, we find
	\begin{align*}
	\int_0^\infty \exp\left(-ax^2 + bx\right)\, \mathrm dx &= \int_0^\infty \exp\left(-a\left(x^2 - \frac{bx}{a}+\frac{b^2}{4a^2}\right) +\frac{b^2}{4a}\right)\, \mathrm dx\\
	&=\sqrt{\frac{\pi}{a}}\exp\left(\frac{b^2}{4a}\right)\P\left(N\left(\frac{b}{2a}, \frac{1}{2a}\right)>0\right)\\
	&=\sqrt{\frac{\pi}{a}}\exp\left(\frac{b^2}{4a}\right)\P\left(N(0,1)> -\frac{1}{\sqrt{2a}}b\right).
	\end{align*}
	For all $b\in\R$ this probability factor is bounded by a constant, which yields Case 2 of the statement. If $b<0$ we may use the Chernoff bound for normal random variables to obtain Case 1:
	\begin{equation*}
	\sqrt{\frac{\pi}{a}}\exp\left(\frac{b^2}{4a}\right)\P\left(N(0,1)> -\frac{1}{\sqrt{2a}}b\right)\leq \sqrt{\frac{\pi}{a}}\exp\left(\frac{b^2}{4a} -\frac{1}{2}\cdot \frac{b^2}{2a}\right)=\sqrt{\frac{\pi}{a}}.\qedhere
	\end{equation*}
\end{proof}

\begin{lemma}\label{l.brownian bridge inf}
Let $B$ be a Brownian bridge of length $T$ from $(0,0)$ to $(T,0)$. Then we have
$$\P\left(\sup_{[0,T]} B(x) \geq r\right)=\P\left(\inf_{[0,T]} B(x) \leq -r\right) = e^{-2r^2/T}.$$
\end{lemma}

\begin{proof}
The equality of the two quantities follows from Brownian symmetry. By Brownian scaling the statement reduces to when $T=1$, which is given by equation (3.40) in \cite[Chapter 4]{karatzas1998brownian}.
\end{proof}

\begin{proof}[Proof of Lemma~\ref{l.bad event 1}]
	By Lemma~\ref{l.J stochastic dominanation}(iii) we have that $J$ is stochastically dominated by the restriction to $[\mf l, \mf r]$ of a Brownian bridge from $(-2T,2T^2)$ to $(2T, 2T^2)$ conditioned to jump over all the poles as well as $\Corner^{\mathfrak l, \F}_k$ and $\Corner^{\mathfrak r, \F}_k$. This event being conditioned on has a constant probability since, on $\fav$, the value of $\tent$ at the poles, $\gxmin_k$, and $\gymin_k$ are all below $T^2$ and $T\geq 1$. We also have on $\fav$ that $-T^2\leq \tent\leq T^2$.
	Thus, for both the pole and no-pole cases of the definition of $(Y,Z)$ from \eqref{e.Y Z defn} and \eqref{e.Y Z defn no pole}, on $\fav$,
	\begin{align*}
	\PF\Big(\max\big\{Y,Z\big\} >  RT^2\Big)
	&\lesssim \B_{2T^2, 2T^2}^{[-2T,2T]}\left(\sup_{t\in[-2T,2T]} B(t) > (R-1)T^2\right)\\
	&= \exp\left(-\frac{1}{2}(R-3)^2T^3\right) = \epsilon^{(R-3)^2 D_k^3/2},
	\end{align*}
	using Lemma~\ref{l.brownian bridge inf} in the second inequality. Similarly for the lower side, we have by Lemma~\ref{l.J stochastic dominanation}(i) that, on $\fav$, $J$ stochastically dominates a Brownian bridge from $(\mf l,-T^2)$ to $(\mf r,-T^2)$. Thus again using Lemma~\ref{l.brownian bridge inf} and that $-T^2\leq \tent\leq T^2$, and for both cases of the definition of $(Y,Z)$,
	\begin{align*}
	\PF\Big(\min\big\{Y, Z\big\} < -RT^{2}\Big)\cdot\one_\fav 
	&\leq \B_{-T^2,-T^2}^{[\mf l, \mf r]}\left(\inf_{t\in[\mf l, \mf r]} B(t) <-(R-1)T^2\right)\\
	&= \exp\left(-\frac{2}{(\mf r-\mf l)}(R-2)^2T^4\right) \leq \epsilon^{(R-2)^2D_k^3},
	\end{align*}
	since $|\mf r|, |\mf l| \leq T$.
	We note that for our range of $R$, $\epsilon^{(R-2)^2D_k^3}\leq \epsilon^{(R-3)^2D_k^3/2}$.

	We are left with bounding 
	$$\PF\left(|Y-Z| > 2RT^{3/2}, -RT^{2}<Y, Z<RT^{2}\right)$$
	on $\fav$. Since the definition of $Y$ and $Z$ depends on whether $P\cap[-2d,2d] = \emptyset$ or not (see \eqref{e.Y Z defn} and \eqref{e.Y Z defn no pole}), the bound we can obtain on the above probability depends on the same as well. However, the bound in the case where a pole is present (which is the one claimed in the statement of Lemma~\ref{l.bad event 1}) actually holds for both cases; this is because the distance between the points where $Y$ and $Z$ measure the deviation of $J$ from $\tent$ is $8d$ when a pole is present, larger than the $2d$ it is when the pole is absent in $[-2d,2d]$. So we will present the case where $P\cap[-2d,2d]\neq \emptyset$, but exactly the same argument works in the other case as well, where it yields a slightly stronger bound corresponding to $2d$ in place of $8d$. Let us define
	$$\mc G = \left\{(y,z): -RT^{2}<y, z< RT^2, |y-z|>2RT^{3/2}\right\},$$
	so that
	$$\PF\left(|Y-Z| > 2RT^{3/2}, -RT^{2}<Y, Z<RT^{2}\right) = \iint_{\mc G} f_J(y,z)\, \mathrm dy\,\mathrm dz \cdot \one_\fav.$$
	To bound this integral, we will use Proposition \ref{p.scost density} to bound the density and make the change of variables $(u,v) = (y-z,y)$. Note that the range of $y$ and $z$ satisfies the hypotheses of Proposition~\ref{p.scost density} with $x_1 = p-4d, x_2 = p+4d$, so that $\sigma^2 = 8d$. These parameter choices satisfy the hypotheses of Proposition \ref{p.scost density} since $\dip = 5d$; in particular, $|y-z|\geq 48RTd$ for all $(y,z)\in \mc G$ as $2RT^{3/2}\geq 48RTd$ since $d\leq \sqrt T/24$. Note that $36\sigma^2=36\times 8d = 288d$. So, on $\fav$,
	\begin{align*}
	\iint_{\mc G} f_J(y,z)&\,\mathrm dy\, \mathrm dz\\
	&\lesssim d^{-1}T^2\iint_{\mc G} \exp\left(-\frac{1}{16d}(|y-z| - 48RTd)^2 + \frac{R^2T^2}{2d} + 288R^2d\right)\, \mathrm dy\,\mathrm dz\\
	&\leq 2d^{-1}T^2\int_{-RT^{2}}^{RT^2}\int_{2RT^{3/2}}^{\infty} \exp\left(-\frac{1}{16d}(u - 48RTd)^2 + \frac{R^2T^2}{2d}+288R^2d\right)\, \mathrm du\,\mathrm dv\\
	&= 2d^{-1} T^2\int_{-RT^{2}}^{RT^2}\int_{0}^{\infty} \exp\left(-\frac{1}{16d}(u+2RT^{3/2}-48RTd)^2 + \frac{R^2T^2}{2d}+288R^2d\right)\, \mathrm du\,\mathrm dv\\
	&\lesssim RT^4d^{-1}\exp\left(\frac{R^2T^2}{2d}+288R^2d\right)\\
	&\quad\times\int_{0}^{\infty}\!\! \exp\left(-\frac{1}{16d}\big(u^2+2u(2RT^{3/2}-48RTd) + 4R^2T^{3} - 192R^2T^{5/2}d + 48^2 R^2T^2d^2\big)\right)\, \mathrm du\\
	&\lesssim RT^4d^{-\frac{1}{2}} \exp\left(-\frac{1}{4d}R^2T^3 + 12R^2T^{5/2} -48\times3R^2T^2d  + \frac{R^2T^2}{2d}+288R^2d\right)\\
	&\leq RT^{4}d^{-\frac{1}{2}}\exp\left(12R^2T^{5/2}\right)\epsilon^{R^2D_k^3/4d}.
	\end{align*}
	We have used Case 1 of Lemma~\ref{l.integral bound} with $a=1/(16d)$ and $b=-2(2RT^{3/2}-48RTd)$ for the integral in the second-to-last line, since $d\leq \sqrt T/24$ implies that $2RT^{3/2} - 48RTd\geq 0$, and thus that $b\leq 0$. In the last line, since $1\leq d \leq \sqrt T/24$, we see that $288R^2d \leq 12R^2T^{1/2}$, and thus the sum of the last three terms in the exponent of the penultimate line is negative and may be dropped. Finally, since $x\leq \exp(x^{5/8})$ for $x\geq 1$, and since $1\leq d\leq \sqrt{T}/24$ and $R\geq 1$ from \eqref{e.R value}, we have that $RT^4d^{-\frac{1}{2}}\leq \exp(R^2T^{5/2})$. This completes the proof of Lemma~\ref{l.bad event 1}.
\end{proof}

\begin{proof}[Proof of Lemma~\ref{l.bad event 2}]
	Using Lemma~\ref{l.bad event 1}, it is enough to show that, on $\fav$ and when $P\cap[-2d,2d]\neq \emptyset$, $\P\left(Y < -T^{3/2}\right) + \PF\left(Z<-T^{3/2}\right)$ is bounded by the right-hand side in the Lemma~\ref{l.bad event 2}'s statement.

	Since we are considering the situation where $Y$ and $Z$ are negative, we may use Proposition~\ref{p.negative density}. So from Proposition~\ref{p.negative density}(2) and Lemma~\ref{l.variance bounds} we have that, on  $\fav \cap \{P\cap [-2d,2d]\neq \emptyset\}$,
	\begin{align*}
	\PF\big(Y < -RT^{3/2}\big) \lesssim d^{-\frac{1}{2}}\int^{\infty}_{RT^{3/2}} \exp\left(-\frac{y^2}{8d}\right)\, \mathrm dy
	 &\lesssim \exp\left(-\frac{1}{8d} R^2T^3\right)
	 = \epsilon^{R^2D_k^3/8d},
	\end{align*}
	where we have performed the change of variables $y\mapsto y+RT^{3/2}$ and applied Case 1 of Lemma~\ref{l.integral bound} with $a=1/(8d)$ and $b=-RT^{3/2}/(4d)$ in the second inequality. Similarly, we have
	$$\PF\big(J(p+4d)-\tent(p+4d) < -RT^{3/2}\big) \lesssim \epsilon^{R^2D_k^3/8d}. \eqno\qedhere$$
\end{proof}

We conclude the section by providing the proof of the technical tool Lemma~\ref{l.density bound tool lower}.

\begin{proof}[Proof of Lemma~\ref{l.density bound tool lower}]
	Let $\nu$ be the law of $X$. For $x<x_0$ and $0<\delta<1$, let $\tilde x = \delta(x_0-x)$. Since $X$ and $N$ are independent, we have
	\begin{align*}
	f(x) &= \frac{1}{\sqrt{2\pi}\sigma_1}\int_{-\infty}^\infty e^{-(x-y)^2/2\sigma_1^2}\, \mathrm d\nu(y)= \frac{1}{\sqrt{2\pi}\sigma_1}\left[\int_{[x-\tilde x, x+\tilde x]} + \int_{[x-\tilde x, x+\tilde x]^c} e^{-(x-y)^2/2\sigma_1^2}\, \mathrm d\nu(y)\right].
	\end{align*}
	From the first equality we see that the density is bounded by $(\sqrt{2\pi}\sigma_1)^{-1}$ for all $x$. For the stronger bound for small enough $x$, note using the hypothesis on $\nu$ that the first integral in the right-hand side is bounded by
	\begin{align*}
	\int_{[x-\tilde x, x+\tilde x]}\,\mathrm d\nu(y) \leq \nu\big((-\infty, x+\tilde x)\big) &\leq A\cdot\exp\left(-\frac{(x+\tilde x-x_0)^2}{2\sigma_2^2}\right) = A\cdot\exp\left(-\frac{(1-\delta)^2(x-x_0)^2}{2\sigma_2^2}\right),
	\end{align*}
	where we have used that $x+\tilde x$ is less than $x_0$; this is due to $\delta<1$ and $x<x_0$.

	The second integral is bounded by
	$\exp\left(-\frac{\tilde x^2}{2\sigma_1^2}\right) = \exp\left(-\frac{\delta^2 (x-x_0)^2}{2\sigma_1^2}\right).$
	These inequalities hold for all $0<\delta<1$, and so if we set $\delta = \sigma_1/(\sigma_1+\sigma_2)$, we obtain our result.
\end{proof}


\section{The difficult case: Above and below the pole on either side}\label{s.difficult case}

At this stage we have proved the required bound on $f_J(y,z)$ in the two cases where $y,z <0$ or $y+z>0$. This leaves the case where $y<0$, $z>0$, and $y+z<0$ (the case where $y>0$ and $z<0$ is clearly symmetric). Perhaps surprisingly, this turns out to be the most difficult case. However, we now give a heuristic reason why we should expect the density $f_J(y,z)$ to be largest in this case, as a proxy for why this case is most difficult.

Recall that $f_J(y,z)$ is the density of $(Y,Z)$, which, from \eqref{e.Y Z defn}, are respectively the deviations of $J$ from $\tent$ at $p-4d$ and $p+4d$. So, the size of the density $f_J(y,z)$ essentially represents a comparison of the probability that $J$ takes the values $y + \tent(p-4d)$ and $z+\tent(p+4d)$ respectively at $p-4d$ and $p+4d$ to the probability of the same for Brownian motion. A larger value of this density is obtained if $J$ finds it easier to adopt the specified values than a Brownian motion does. This is precisely what happens when $y<0$, $z>0$, and $y+z<0$, as $J$ has the pole at $p$ which pushes it up and helps it attain the value of $z$ at $p+4d$; a Brownian motion has no such assistance. Thus the density should be highest for this case. Considering the situation in the two cases we have already analysed in Sections~\ref{s.easier cases} and \ref{s.moderate case} should convince the reader that in Section~\ref{s.easier cases} the pole actually makes $J$'s task more difficult than $B$'s, which has no pole, while in Section~\ref{s.moderate case}, the pole has essentially no effect. 

In this section, since $y+z<0$, the vault cost $\jcost$ cannot be ignored. Thus we need a stronger bound on $f_J(y,z)$ than was required in Section~\ref{s.moderate case}; in fact, we need a bound of the same basic form as that proved in Section~\ref{s.easier cases}. This is why the previous paragraph's conclusion that the density is highest in this case indicates that the required argument will be more delicate.
 
 Let $f_J(z\mid y)$ be the conditional density of $Z$ at $z$ given $Y=y$. Our aim will be the following proposition.

\begin{proposition} \label{p.density one positive one negative}
Let $R$ be as in \eqref{e.R value}. If $40Td \leq z \leq 2RT^{3/2}$ and $y+z< 0$,
$$f_J(z\mid y) \cdot \one_{\fav, P\cap[-2d,2d]\neq\emptyset} \lesssim d^{-\frac{1}{2}}\cdot \exp\left(-\frac{z^2}{8d} + 20RT^{5/2}\right).$$
Further, $f_J(z\mid y) \lesssim d^{-\frac{1}{2}}$ for all $y,z\in \R$ with $y+z<0$.
\end{proposition}

First we show that Proposition~\ref{p.density one positive one negative} implies the sufficient bound \eqref{e.sufficient bound to prove} when $y < 0, z > 0,$ and $y + z < 0$ (and symmetrically when $y > 0$ and $z < 0$).

\begin{lemma}\label{l.yz<0 case}
When $yz<0$, $y+z<0$, and $(y,z)\in \mc G_R^{(1)}$, we have \eqref{e.sufficient bound to prove} with $G_2'=41RD_k^{5/2}$ and $G_1' = G'$, where $G'$ is a constant independent of $\epsilon, k,$ and $d$. 
\end{lemma}

\begin{proof}
	We give the proof for when $y<0$ and $z>0$. From \eqref{e.S is approx S*} and \eqref{e.jcost approximate exponential form}, we see that in this case
	\begin{equation}\label{e.total cost in hard case}
	\jcost\cdot\scost \lesssim d\cdot\exp\left(\frac{y^2}{8d}+\frac{z^2}{8d} + 21RT^{5/2}\right).
	\end{equation}
	So it suffices to prove, for some $G<\infty$,
	$$f_J(y,z) \lesssim d^{-1}\cdot\exp\left(-\frac{y^2}{8d}-\frac{z^2}{8d} + GT^{5/2}\right).$$
	There are two cases to consider. If $z \leq 40Td$, then, from \eqref{e.total cost in hard case}, $\jcost\cdot\scost$ is bounded by
	\begin{equation*}
	d\cdot\exp\left(\frac{y^2}{8d}+ \frac{40^2T^{2}d^2}{8d} + 21RT^{5/2}\right) \leq d\cdot\exp\left(\frac{y^2}{8d}+30RT^{5/2}\right)
	\end{equation*}
	since $d\leq \sqrt T/24$ and $40^2/(8\times 24) \leq 9$. Noting that $30 < 41$, it suffices to prove
	$$f_J(y,z) \lesssim d^{-1}\cdot\exp\left(-\frac{y^2}{8d}\right).$$
	This is provided by Proposition~\ref{p.negative density} and Lemma~\ref{l.variance bounds} since $y<0$, and  by Proposition~\ref{p.density one positive one negative}'s latter statement that $f_J(z\mid y) \lesssim d^{-\frac{1}{2}}$. 

	Now suppose $z>40Td$. Note that we also have $z\leq 2RT^{3/2}$, since $y<0$ and $(y,z)\in \mc G_R^{(1)}$ implies that $|y-z|\leq 2RT^{3/2}$. So from Proposition~\ref{p.negative density}, Lemma~\ref{l.variance bounds}, and Proposition~\ref{p.density one positive one negative}, we obtain
	\begin{align*}
	f_J(y,z) &\lesssim d^{-1} \cdot \exp\left(-\frac{y^2}{8d}-\frac{z^2}{8d}\right)\cdot \exp\left(20RT^{5/2}\right).
	\end{align*}
	This completes the proof, after taking into account the extra factor of $\exp\left(21RT^{5/2}\right)$ which arises from the expressions for $\jcost$ and $\scost$ as in \eqref{e.total cost in hard case}.
\end{proof}

We next turn to discussing the proof ideas of Proposition~\ref{p.density one positive one negative}. The claim of the proposition may be a surprising one at first glance, for, in a slight abuse of the language of pseudo-variance used in Section~\ref{s.moderate case}, Proposition~\ref{p.density one positive one negative} says that $J(p+4d)-\tent(p+4d)$, conditionally on $J(p-4d)-\tent(p-4d)$ being negative, has pseudo-variance at most $4d$; in contrast, a Brownian motion, conditionally on its value at $p-4d$, would have a much higher variance of $8d$ at the position $p+4d$. So we must crucially use both that $J(p-4d)-\tent(p-4d)<0$ (i.e., $y<0$) and that $J$ must satisfy $J(p)\geq \tent(p)$. 

Heuristically, because $J$ is jumping over $\tent(p)$ from a negative value at $p-4d$, it will make the jump with a very low margin. Thus the variance at $p$ is not $4d$ as it would be for a Brownian motion, but essentially~0. This explains how we can get a pseudo-variance of at most $4d$ at $p+4d$ for $J$. This intuition is captured in Lemma~\ref{l.limit u}, which says that we may safely restrict our analysis to the case where $J$ jumps over the pole at $p$ by at most 1. To prove Lemma~\ref{l.limit u}, we need a technical lemma about the monotonicity of conditional probabilities of Gaussians. The result is identical to \cite[Lemma 2.21]{hammond2017brownian}, but we include its short proof here for completeness.

\begin{lemma}\label{l.monotonicity of conditional gaussian prob}
Fix $r>0$, $m \in \R$, and $\sigma^2 >0$, and let $X$ be distributed as $N(m, \sigma^2)$. Then the quantity $\P(X\geq r+s \mid X\geq s)$ is a strictly decreasing function of $s\in \R$.
\end{lemma}

\begin{proof}
Note that
$$
\log \P(X \geq s+r \mid X \geq s)=\log \int_{s+r}^{\infty} \exp \left\{-\frac{(x-m)^{2}}{2 \sigma^{2}}\right\} \mathrm{d} x-\log \int_{s}^{\infty} \exp \left\{-\frac{(x-m)^{2}}{2 \sigma^{2}}\right\} \mathrm{d} x
$$
has derivative in $s$ given by
$$\frac{\exp\left\{-\frac{(s-m)^2}{2\sigma^2}\right\}\int_{s+r}^\infty \exp \left\{-\tfrac{(x-m)^2}{2\sigma^2} \right\} \,\mathrm dx - \exp\left\{-\frac{(s+r-m)^2}{2\sigma^2}\right\}\int_{s}^\infty \exp \left\{-\tfrac{(x-m)^2}{2\sigma^2} \right\} \,\mathrm dx}{\int_{s + r}^\infty \exp \left\{- \tfrac{(x-m)^2}{2\sigma^2} \right\} \mathrm{d}x \,\cdot\, \int_{s}^\infty \exp\left\{- \tfrac{(x-m)^2}{2\sigma^2} \right\} \mathrm dx}.$$
The denominator is clearly positive. Performing the change of variable $x\mapsto x+r$ in the first integral of the numerator and manipulating the exponents show that the numerator equals
$$
 \int_{s}^\infty \exp \left\{-\tfrac{(x-m)^2+(s-m)^2+r^2}{2\sigma^2} \right\} \left(\exp \left\{ -\tfrac{(x-m)r}{\sigma^2} \right\}  -  \exp \left\{- \tfrac{(s-m)r}{\sigma^2} \right\} \right) \mathrm{d}x\,.
%
$$
The proof is complete by noting that this integrand is strictly negative for all $x>s$.
\end{proof}

In order to state Lemma~\ref{l.limit u}, we define the random variable $U$ to be the deviation of the jump ensemble from the $\tent$ map at the pole $p$. For later use, we also take this opportunity to define $W_\eta$ to be the same at $p+4d+\eta$ for $\eta<d$. So, we define
\begin{equation}\label{e.U V definition}
\begin{split}
U &:= J(p) - \tent(p)\\
W_{\eta} &:= J(p+4d+\eta) - \tent(p+4d+\eta).
\end{split}
\end{equation}
The parameter $\eta$, as in the previous section, will be set to a specific small value in a local randomization argument later. Recall also from \eqref{e.Y Z defn} that $Y$ and $Z$ are respectively the deviation of $J$ from $\tent$ at $p-4d$ and $p+4d$. See Figure~\ref{f.u w y z defn}. We now turn to our assertion that $J$ typically makes a narrow jump over $p$.

\bigskip

\begin{figure}[h]
\centering {\epsfig{file=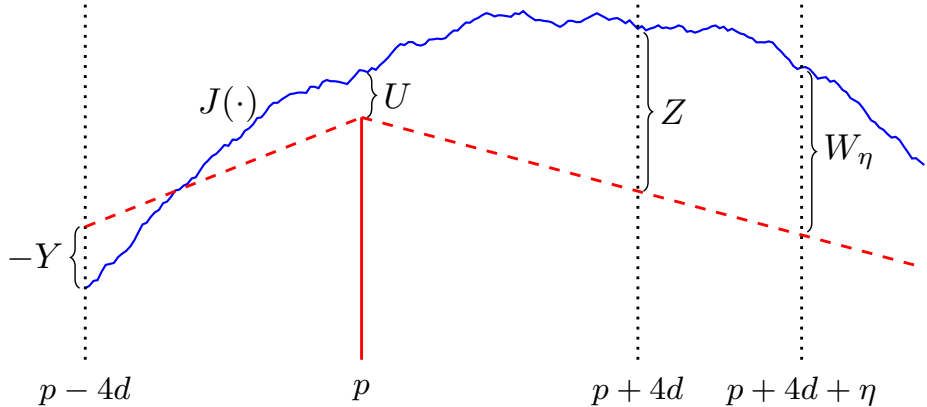, width=0.75\textwidth}}
\caption{Illustrating the definitions of $U$, $W$, $Y$ and $Z$ in the subcase being addressed in this section; $Y$ is negative, and so the length being shown is $-Y$. The blue curve is $J$, while the red dashed function is $\tent$. The red vertical line emphasises the height of the pole at $p$.}\label{f.u w y z defn}
\end{figure}

\begin{lemma}[Narrow jump over $p$]\label{l.limit u}
Let $f_J(z,u \mid y)$ be the joint conditional density of $Z$ and $U$ given $Y$. For $y,z\in \R$ such that $y+z<0$,
$$f_J(z\mid y) \lesssim d^{\frac{1}{2}}\int_0^1 f_J(z,u\mid y) \, \mathrm du.$$
\end{lemma}

\begin{proof}
	We have
	\begin{align}\label{e.narrow jump}
	\int_0^1 f_J(z,u \mid y) \, \mathrm du &= f_J(z\mid y) \int_0^1 f_J(u\mid y,z)\, \mathrm du.
	\end{align}
	Let $X$ be distributed as $N(0, 2d)$, which is the distribution of a Brownian bridge of duration $8d$ from 0 to 0 at its midpoint, and let $\Delta = \tfrac{1}{2}(\tent(p-4d)+\tent(p+4d)) -\tent(p)$. Note that $\Delta\leq 0$ by the concavity of the $\tent$ map. Then the second factor on the right hand side is 
	\begin{align*}
	\PF\bigl( U \in [0,1] \ \big|\  Y=y,Z=z\big) &= \P\left(X + \frac{y+z}{2} + \Delta \in [0,1] \Bigm| X + \frac{y+z}{2} + \Delta\geq 0\right)\\
	&= 1-\P\left(X + \frac{y+z}{2} + \Delta\geq 1 \Bigm| X+\frac{y+z}{2} + \Delta\geq 0\right).
	\end{align*}
	Now by Lemma~\ref{l.monotonicity of conditional gaussian prob} with $r=1$ and $s=-(y+z)/2-\Delta$, we have that
	\begin{align*}
	\P\left(X \geq 1-\frac{y+z}{2}-\Delta \Bigm| X\geq -\frac{y+z}{2}-\Delta\right)
	\end{align*}
	is an increasing function of $y+z$, i.e., $\PF(U\in[0,1]\mid Y=y, Z=z)$ is a decreasing function of $y+z$. So for $y+z<0$, we obtain that
	$$\PF\big(U \in [0,1] \mid Y=y,Z=z\big) \geq \P\bigl(X\in[0,1]\mid X\geq 0\bigr) = \P\bigl(|X| \leq 1\bigr);$$
	the second quantity is the value of the first at $y+z = -2\Delta\geq 0$. 
	An easy bound using the standard normal density gives that $\P(|X|\leq 1)^{-1} \lesssim d^{\frac{1}{2}}$ for $d\geq 1$; tracing back the relations and using this last bound in \eqref{e.narrow jump} yields
	\begin{equation*}
	f_J(z\mid y) \lesssim d^{\frac{1}{2}}\int_0^1 f_J(z,u \mid y) \, \mathrm du,
	\end{equation*}
	completing the proof of Lemma~\ref{l.limit u}.
\end{proof}

 In the previous case in Section~\ref{s.easier cases} we made use of the tool Lemma~\ref{l.density bound tool lower} to convert a tail bound to a density bound, and we will make use of essentially the same tool in this section; it is simply restated in a form involving the upper tail. 

\begin{lemma}\label{l.density bound tool upper}
Let $X$ be a random variable such that $\P(X>x)\leq A\exp(-\frac{1}{2\sigma_2^2}(x-x_0)^2)$ for $x>x_0$, and let $N$ be a normal random variable with mean 0 and variance $\sigma_1^2$ which is independent of $X$. Then the density $f$ of $X+N$ satisfies
\begin{align*}
f(x) \leq \frac{A+1}{\sqrt{2\pi}\sigma_1}\cdot\exp\left(-\frac{(x-x_0)^2}{2(\sigma_1+\sigma_2)^2}\right)
\end{align*}
for $x>x_0$, and is bounded by $1/\sqrt{2\pi}\sigma_1$ for all $x\in\R$.
\end{lemma}

\begin{proof}
This follows by applying Lemma~\ref{l.density bound tool lower} to $\tilde X := -X$, since mean zero normal distributions are symmetric, and where $X$ is as in Lemma~\ref{l.density bound tool upper}.
\end{proof}

The broad idea of the proof of Proposition~\ref{p.density one positive one negative} is to write the distribution of $Z$, conditionally on $U = J(p) -\tent(p)$ and $W_\eta = J(p+4d+\eta)-\tent(p+4d+\eta)$, in terms of a Brownian bridge using the definition of $J$, and then use this Brownian structure to obtain a bound on the conditional density of $Z$ given $Y$ via Lemma~\ref{l.density bound tool upper} and Lemma~\ref{l.limit u}. This is the same strategy of local randomization that was used in Section~\ref{s.moderate case} to prove Proposition~\ref{p.scost density}.

So, as before, we will need tail probabilities on the distribution of $W_{\eta}$ given $U$ and $Y$. By the Markov property of Brownian bridges, $W_{\eta}$ is conditionally independent of $Y$ given $U$, and so in fact we need a tail probability for $W_{\eta}$ given only $U$. Such a tail bound is the content of Lemma~\ref{l.p+ tail bound}, whose argument is essentially again stochastically dominating $J$ by a Brownian bridge.

\begin{lemma}\label{l.p+ tail bound}
We have for $0<\eta\leq d$, $u\geq 0$, $w>u + 9T(4d+\eta)$ and on the event $\fav \cap \{P\cap[-2d,2d]\neq\emptyset\}$,
\begin{align*}
\PF\Big(W_\eta > w \ \big|\  U = u&\Big) \lesssim \exp\left(-\frac{(w-u - 9T(4d+\eta))^2}{2(4d+\eta)}+7T\right).
\end{align*}
\end{lemma}

To prove Lemma~\ref{l.p+ tail bound} we will need a lower bound on the probability of a Brownian bridge jumping over poles. This is recorded in the next lemma, whose straightforward proof is deferred to the end of the section to permit the flow of the overall argument.

\begin{lemma}\label{l.prob bb above poles}
Let $N\geq 1$, $x_0 \in [\mathfrak l, \mathfrak r]$, and $x_N = \mathfrak r$. Let $x_1<\ldots <x_{N-1}$ be $\F$-measurable points in $(x_0, x_N)$, and $B$ be a Brownian bridge with law $\B^{[x_0, 2T]}_{0,0}$. Then there exists a $G<\infty$ such that
$$\PF\Big( B(x_i) > 0, i=1,\ldots, N-1\Big) > G^{-1}N^{-1/2}\exp(-3N).$$
\end{lemma}


\begin{proof}[Proof of Lemma~\ref{l.p+ tail bound}]
	Define the coordinate $\phiend$ by
	$$\phiend := u+\tent(p)+5T(2T-p),$$
	which is the $y$-coordinate at $x=2T$ of a line with slope $5T$ started at $(p, u+\tent(p))$. Since on $\fav$ we have $|\mathrm{slope}(\tent)| \leq 4T$, it follows that $\phiend > \L(2T)$. Lemma~\ref{l.J stochastic dominanation}(iv) tells us that the Brownian bridge $B$ from $(p, u+\tent(p))$ to $(2T, \phiend)$  conditioned to jump over all poles in $[p,2T]$ and to be above $\Corner^{\mathfrak r, \F}_k$ stochastically dominates $J$ on $[p,2T]$. Let $A$ be the conditioning event just mentioned. Then we have, on $\fav\cap\{P\cap[-2d,2d]\neq \emptyset\}$,
	$$\PF\Bigl(W_\eta > w \ \bigm|\  U = u\Bigr)
	\leq \B^{[p,2T]}_{u+\tent(p),\phiend} \Bigl(B(p+4d+\eta) > w+\tent(p+4d+\eta) \ \bigm|\  A\Bigr).$$
	We need to lower bound $\PF(A)$ on $\fav$. Note that on $\fav$ we have $\Corner_k^{\mathfrak r, \F} \leq T^2$, while $B(\mathfrak r)$ has mean bounded below by
	$$J(p) + 5T\cdot(\mathfrak r-p) \geq -T^2 + 5T\cdot(T/2-d) = \frac32T^2 - 5Td;$$
  here we used $J(p)\geq \tent(p) = \L_n(k+1,p)\geq -T^2$ on $\fav$. These bounds, along with the concavity of $\tent$ and Lemma \ref{l.prob bb above poles} (with $N=|P| \leq 2T$, as $\dip\geq 1$ implies this bound), implies that $\PF(A)\cdot \one_\fav$ is bounded below, up to an absolute constant, by $T^{-\frac{1}{2}}\exp(-6T) \geq \exp(-7T)$ as $T\geq 1$. So, on $\fav \cap \{P\cap[-2d,2d]\neq\emptyset\}$,
	\begin{align*}
	\PF\Big(W_\eta > w \ \bigm|\  U = u\Big)
	&\lesssim \exp(7T)\cdot\B^{[p,2T]}_{u+\tent(p),\phiend} \Big(B(p+4d+\eta) > w+\tent(p+4d+\eta)\Big).
	\end{align*}
	Let $\rho = \frac{4d+\eta}{2T-p}$. Then $\EF[B(p+4d+\eta)] = u+\tent(p) + 5T(4d+\eta)$ and $\Var_\F(B(p+4d+\eta)) = (1-\rho)(4d+\eta)$. So, for $w>u + 9T(4d+\eta)$ and on $\fav$,
	\begin{align*}
	\MoveEqLeft[4]
	\B^{[p,2T]}_{u+\tent(p), \phiend} \Big(B(p+4d+\eta) > w+\tent(p+4d+\eta)\Big)\\
	&= \PF\Big(N(u+\tent(p)+5T(4d+\eta), (1-\rho)(4d+\eta))> w+\tent(p+4d+\eta)\Big)\\
	&\leq \PF\Big(N(0, (1-\rho)(4d+\eta))> w- u - 9T(4d+\eta)\Big)\\
	&\leq \exp\left(-\frac{(w-u - 9T(4d+\eta))^2}{2(4d+\eta)}\right)
	\end{align*}
	The third line uses that $\tent(p+4d+\eta) - \tent(p) \geq -4T(4d+\eta)$ from \eqref{e.tent map slope}; the final inequality uses the upper bound of Lemma~\ref{l.normal bound} with $\sigma^2 = (1-\rho)(4d+\eta)$ and $t = w-u-9T(4d+\eta)$.
\end{proof}

With this tail bound we may turn to the proof of Proposition~\ref{p.density one positive one negative}.

\begin{proof}[Proof of Proposition~\ref{p.density one positive one negative}]
	Let $\lambda$ be defined as
	$$\lambda = \frac{\eta}{4d+\eta}$$
	so that, conditionally on $Y$, $U$, and $W_\eta$, the distribution of $Z$ is
	\begin{equation} \label{e.J(p+2d) expression}
	Z = \lambda U + (1-\lambda)W_\eta + N(0,4\lambda d).
	\end{equation}
	This relation holds because, by the definition of $J$, the distribution of $J$ on $[p,p+4d+\eta]$, conditionally on $J(p)$ and $J(p+4d+\eta)$, is a Brownian bridge with the prescribed endpoint values, and because $\tent$ is affine on $[p,p+4d+\eta]$ for $\eta<d$ as $\dip = 5d$. Note that the variance of this Brownian bridge at $p+4d$ is $\frac{4d\eta}{4d+\eta} = 4\lambda d$.  

	Let $\nu_p(\,\cdot \mid y)$ be the conditional law of $U$ given that $Y = y$.
	From the narrow jump Lemma~\ref{l.limit u} we see
	\begin{align}\label{e.u can be limited}
	f_J(z\mid y) &\lesssim d^{\frac{1}{2}}\cdot\int_0^1 f_J(z,u\mid y)\, \mathrm du
	= d^{\frac{1}{2}}\cdot\int_0^1 f_J(z\mid u,y) \, \mathrm d\nu_p(u\mid y),
	\end{align}
	where $f_J(z\mid u, y)$ is the conditional density of $Z$ given $U$ and $Y$. So our task is to bound $f_J(z\mid u, y)$ when $u\in[0,1]$. By the Markov property, this object does not depend on $y$. We will obtain this bound by converting the conditional tail bound of Lemma~\ref{l.p+ tail bound} to a conditional density bound using Lemma~\ref{l.density bound tool upper}.

  We have from Lemma~\ref{l.p+ tail bound}, by taking $w = \frac{t-\lambda u}{1-\lambda}$ and simplifying the resulting expression, that, for $t>u+36Td$,
  \begin{equation}\label{e.difficult case conditional tail bound}
  \PF\Bigl(\lambda U + (1-\lambda)W_{\eta}>t \mid U=u, Y=y\Bigr) \lesssim \exp\left(-\frac{1}{8d(1-\lambda)}\left(t-u-36Td\right)^2 + 7T\right).
  \end{equation}
  While simplifying we used that $(4d+\eta)(1-\lambda) = 4d$. 
  Now using \eqref{e.J(p+2d) expression} and \eqref{e.difficult case conditional tail bound}, we apply Lemma~\ref{l.density bound tool upper}. The parameters of this application are $\sigma_1^2 = 4\lambda d$, $\sigma_2^2 = 4(1-\lambda)d$, $x_0 = u+36Td$, and $A$ equal to the constant specified by $\lesssim$ in \eqref{e.difficult case conditional tail bound} multiplied by $\exp(7T)$. Thus we obtain, for every $\eta<d$ and $z>u+36Td$,
  $$f_J(z\mid u, y) \lesssim \eta^{-\frac{1}{2}}\cdot\exp\left(-\frac{1}{2\sigma^2}\left(z-u-36Td\right)^2 + 7T\right),$$
  where $\sigma^2 = 4d(\lambda^{1/2}+(1-\lambda)^{1/2})^2 \leq 4d(1+2\lambda^{1/2})\leq 4d(1+\eta^{1/2})$ since $0\leq\lambda\leq \eta/4$. Setting $\eta^{1/2} = 8T^{-2}d$ (which satisfies $\eta<d$ as $d\leq \sqrt T/24$ and $T\geq 1$) and using that $(1+x)^{-1}\geq 1-x$ for $x=\eta^{1/2}$ yields
  %
  %
  %
  $$f_J(z\mid u, y) \lesssim d^{-1}T^2\cdot\exp\left(-\frac{1}{8d}\left(z-u-36Td\right)^2+ T^{-2}\left(z-u-36Td\right)^2 + 7T\right).$$
  Using this in \eqref{e.u can be limited} and bounding the integral by a trivial bound on the integrand then shows
  $$f_J(z\mid y) \lesssim d^{-\frac{1}{2}}T^2\cdot\exp\left(-\frac{1}{8d}\left(z-1-36Td\right)^2+ T^{-2}z^2 + 7T\right)$$
  for $z>36Td + 1$. Since $4Td>1$, and since $T^2\leq e^T$ as $T\geq 1$, it also holds that 
  $$f_J(z\mid y) \lesssim d^{-\frac{1}{2}}\cdot \exp\left(-\frac{1}{8d}\left(z-40Td\right)^2+ T^{-2}z^2 + 8T\right)$$
  for $z>40Td$. Expanding the expression in the exponent gives the expression
  $$-\frac{z^2}{8d} +\frac{z(40Td)}{4d}-\frac{40^2T^2d^2}{8d}+z^2T^{-2}+8T \leq -\frac{z^2}{8d}+20RT^{5/2} -200T^2d+4R^2T+8T,$$
  using that $z\leq 2RT^{3/2}$. Since $R=6\sqrt d$ from \eqref{e.R value} and since $T,d\geq 1$, the last three terms in the previous display are collectively negative and may be dropped. This proves the first part of Proposition~\ref{p.density one positive one negative}.

  Now we turn to the latter claim of Proposition~\ref{p.density one positive one negative}. The tail bound \eqref{e.difficult case conditional tail bound}, combined with the latter part of Lemma~\ref{l.density bound tool upper}, says that $f_J(z\mid y,u) \lesssim d^{-\frac{1}{2}}$. Then \eqref{e.u can be limited} gives the latter claim of Proposition~\ref{p.density one positive one negative}. This completes the proof of Proposition~\ref{p.density one positive one negative}.
\end{proof}

Finally, we provide the last piece of the proof of Proposition~\ref{p.density one positive one negative} by proving Lemma~\ref{l.prob bb above poles}.

\begin{proof}[Proof of Lemma~\ref{l.prob bb above poles}]
Let $A$ denote the event that a Brownian motion $B'$ with law $\B^{[x_0,2T]}_{0,*}$ is negative at $2T$; write $\widetilde B$ for the process $B'$ conditioned on $A$. Since Brownian motion on $[x_0, 2T]$ conditioned on its value at $2T$ is a Brownian bridge on $[x_0, 2T]$ with appropriate endpoints, Lemma~\ref{l.bb independent decomposition property} implies that the Brownian bridge $B$ in the statement of Lemma~\ref{l.prob bb above poles} stochastically dominates $\widetilde B$. Thus,
\begin{align}
\PF\bigg( \bigcap_{i=1}^{N-1}\{B(x_i) > 0\}\bigg) \geq \P\Big(\bigcap_{i=1}^{N-1}\left\{\widetilde B(x_i) > 0\right\} \ \big|\  A\Big)\nonumber
&= \frac{\P\Big(\bigcap_{i=1}^{N}\left\{\widetilde B(x_i) > 0\right\}, \widetilde B(2T) < 0\Big)}{\P\Big(\widetilde B(2T)<0\Big)}.\nonumber
\end{align}
The denominator is equal to $\frac{1}{2}$. We may lower bound the numerator by
\begin{align}
\P\left(\bigcap_{i=1}^{N}\left\{\widetilde B(x_i)-\widetilde B(x_{i-1})\in[0, \sqrt{x_{i}-x_{i-1}}]\right\}, \widetilde B(2T)- \widetilde B(x_N) < -\sum_{i=1}^N \sqrt{x_i-x_{i-1}}\right)\nonumber\\
=p^N\cdot\P\left(N(0,2T-x_N) < -\sum_{i=1}^N \sqrt{x_i-x_{i-1}}\right),\label{e.bb above poles prob 1}
\end{align}
using the independence of Brownian motion increments, where $p = \P\big(N(0,1)\in [0,1]\big)>(2\pi)^{-1/2}e^{-1/2}>e^{-2}$; the first inequality is by lower bounding the standard normal density on $[0,1]$ by $(2\pi)^{-1/2}e^{-1/2}$, and the last inequality is by numerical calculation. Now,
\begin{align*}
\P\left(N(0,2T-x_N) < -\sum_{i=1}^N \sqrt{x_i-x_{i-1}}\right)\geq \P\left(N(0,1) > \frac{\sqrt N \sqrt x_N}{\sqrt{2T-x_N}}\right)\geq \P\left(N(0,1) > \sqrt N\right)
&\gtrsim \frac{e^{-N/2}}{\sqrt N}.
\end{align*}
We have used the Cauchy-Schwarz inequality in the first inequality; that $x_N = \mathfrak r \leq T$ in the second; and Lemma \ref{l.normal bound} in the last. Combining with \eqref{e.bb above poles prob 1} and that $\P\bigl(\tilde B(2T)<0\bigr) = \frac{1}{2}$ completes the proof.
\end{proof}

\section{When no pole is present}\label{s.no pole}
The aim of this section is to perform the final task in the proof of Theorem~\ref{t.prob estimate for J}, namely to provide the proof of Proposition~\ref{p.prob estimate no pole}. The analysis follows similar lines to the case where a pole was present, but here we are aided by the fact that there is no vault cost $\jcost$. Because there is no pole, we cannot make the choice of stepping back to $p-4d$ and $p+4d$ in decomposing the jump curve probability. Since we know only that there is no pole in $[-2d,2d]$, and we need to give ourselves some distance from the nearest pole, we make the choice to not step back and instead decompose at $-d$ and $d$. More precisely we recall, as mentioned before Lemma~\ref{l.bad event 1}, that $Y$ and $Z$ are defined on the event $P\cap [-2d,2d] = \emptyset$ as
\begin{align*}
Y &:= J(-d) -\tent(-d)\\
Z &:= J(d) - \tent(d).
\end{align*}
Correspondingly, on $P\cap [-2d,2d] = \emptyset$, $f_J(y,z)$ is the joint density of $(Y,Z)$ under this definition.

As there is no vault cost, we only need to consider the slope cost $\scost$. In this context, since the interval $[-d,d]$ has length $2d$, $\scost$ is 
$$S=d^{\frac{1}{2}}\cdot \exp\left(-\frac{1}{4d}(\tilde y-\tilde z)^2\right);$$
here, $\tilde y = y + \tent(-d)$ and $\tilde z=z+\tent(d)$, so that $Y=y$ implies $J(-d) = \tilde y$ and similarly for $Z=z$ and $J(d) = \tilde z$; thus the tilde plays the same role as it did in Section~\ref{s.heuristics for proof} of going from the value of $Y$ or $Z$ to the value of $J$ at the corresponding points.

The bound we will obtain from Proposition~\ref{p.scost density} is in terms of $y$ and $z$, and here we see the exact form of the leeway factor needed to write $S$ in terms of $y$ and $z$, as was done in \eqref{e.S is approx S*} in the pole case; this bound will hold for $(y,z)\in\mc G_R^{(2)}$, where we recall the definition of $\mc G_R^{(2)}$ from \eqref{e.G_R^2 defn}. Using from \eqref{e.tent map slope} that the slope of the Tent map is bounded in absolute value by $4T$, so that $|\tent(d) - \tent(-d)|\leq 8Td$, we see that, for $(y,z)\in\mc G_R^{(2)}$,
\begin{equation}\label{e.y to tilde y}
\begin{split}
S=d^{\frac{1}{2}}\cdot \exp\left(\frac{1}{4d}(\tilde y - \tilde z)^2\right) 
&\leq d^{\frac{1}{2}}\cdot \exp\left(\frac{1}{4d}(y - z)^2+ 16T^2d + \frac{1}{2d}(y-z)(\tent(d) - \tent(-d)) \right)\\
&\leq d^{\frac{1}{2}}\cdot \exp\left(-\frac{1}{4d}(y - z)^2 +T^{5/2}+ 8RT^{5/2}\right)\\
&\leq d^{\frac{1}{2}}\cdot \exp\left(-\frac{1}{4d}(y - z)^2 +9RT^{5/2}\right);
\end{split}
\end{equation}
the first inequality using the mentioned bound on $|\tent(d)-\tent(-d)|$; the second since $(y,z)\in \mc G_R^{(2)}$ implies $|y-z|\leq 2RT^{3/2}$, and since $d\leq \sqrt T/24$ and $16/24\leq 1$; and third since $R\geq 1$ from \eqref{e.R value}.

Apart from this, we will also need to make a suitable choice for the distribution $\mu$, discussed in Section~\ref{s.heuristic.addrssing simplifications}, which is here the distribution of the Brownian motion we start at $-d$ to which we compare $J$. We select $\mu$ to be uniform on $[-RT^2, RT^2]$ as this interval contains all values $y$ may take when $(y,z)\in \mc G_R^{(2)}$. As before, the cost for this choice of $\mu$ is only polynomial in~$T$, and so does not affect the bound we need to prove on $f_J(y,z)$.

So overall our aim is to get a bound on the joint density of the form
$$f_J(y,z) \leq \exp\left(-\frac{1}{4d}(y-z)^2\right)\cdot \exp\left(GT^{5/2}\right)$$
for some $G<\infty$. This of course is essentially immediate from Proposition~\ref{p.scost density} when $|y-z|$ is sufficiently large; recall that Proposition~\ref{p.scost density} was stated carefully to not assume the presence of a pole in the interval of consideration.

\begin{proof}[Proof of Proposition \ref{p.prob estimate no pole}]
	On $\fav \cap \{P\cap[-2d,2d]=\emptyset\}$,
	\begin{align*}
	\PF\!\Big(J(\,\cdot\,) - J(-d)\in A, (Y, Z) \in \mc G^{(2)}_R\Big)
	&= \EF\!\!\left[\PF\Big(J(\,\cdot\,) - J(-d)\in A\ \bigm|\  J(d), J(-d)\Big)\one_{(Y, Z) \in \mc G^{(2)}_R}\right]\\
	&= \EF\left[\B^{[-d,d]}_{J(-d),J(d)}\big(\tilde A\big)\one_{(Y, Z) \in \mc G^{(2)}_R}\right],
	\end{align*}
	where $\tilde A$ is the set of functions $f$ in $\mc C_{*,*}([-d,d], \R)$ such that $f(\,\cdot\,) - f(-d) \in A$. 
	Recall the notation $\tilde y = y + \tent(-d)$ and $\tilde z = z + \tent(d)$. We may write the last displayed expression as an integral on the event $\fav \cap \{P\cap[-2d,2d]=\emptyset\}$:
	\begin{align}
	\EF\left[\B^{[-d,d]}_{J(-d),J(d)}\big(\tilde A\big)\one_{(Y, Z) \in \mc G^{(2)}_R}\right]
	&= \iint_{\mc G^{(2)}_R} \B^{[-d,d]}_{\tilde y, \tilde z}\big(\tilde A\big) f_J(y,z)\,\mathrm dy\,\mathrm dz\nonumber\\ 
	&=d^{-\frac{1}{2}}\iint_{\mc G^{(2)}_R} \B^{[-d,d]}_{\tilde y, \tilde z} \big(\tilde A\big)e^{-\frac{1}{4d}(\tilde y-\tilde z)^2}\cdot \bigl(f_J(y,z)\cdot S\bigr)\,\mathrm dy\,\mathrm dz. 
	\label{e.no pole integral}
	\end{align}
	From Proposition~\ref{p.scost density} with parameters $x_1=-d$ and $x_2=d$, we know that on $\fav$ and when $|y-z|>12RTd$ and $(y,z)\in \mc G_R^{(2)}$,
	\begin{align*}
	f_J(y,z) &\lesssim d^{-1}T^2\cdot \exp\left(-\frac{1}{4d}(|y-z| - 12RTd)^2 + \frac{4R^2T^2}{2d}+72R^2 d\right)\\
	 &= d^{-1}T^2\cdot \exp\left(-\frac{1}{4d}\left((y-z)^2  - 24RTd|y-z|+ 144R^2T^2d^2\right) + \frac{4R^2T^2}{2d}+72R^2 d\right)\\
	 &\leq d^{-1}\cdot \exp\left(-\frac{1}{4d}(y-z)^2 + 6RT|y-z| -36R^2T^2d+ \frac{4R^2T^2}{2d}+ 72R^2 d+T^2\right)\\
	 &\leq d^{-1} 
  \cdot\exp\left(-\frac{1}{4d}(y-z)^2 + 12R^2T^{5/2}\right),
	\end{align*}
	the penultimate inequality since $T^2\leq \exp(T^2)$ as $T\geq 1$; and the last inequality since on $\mc G^{(2)}_R$ we have $|y-z|<2RT^{3/2}$ and since $1\leq d\leq \sqrt T/24$, which implies that the sum of the last four terms in the exponent in the penultimate line is negative and may be dropped.

	For when $|y-z|<12RTd$, we also use Proposition~\ref{p.scost density} to say that $f_J(y,z) \lesssim d^{-1}$; and so we see that
	\begin{align*}
	f_J(y,z) \lesssim d^{-1}\cdot\exp\left(-\frac{(y-z)^2}{4d} + \frac{(y-z)^2}{4d}\right) &\leq d^{-1}\cdot\exp\left(-\frac{1}{4d}(y-z)^2 + \frac{144R^2T^2d^2}{4d}\right)\\
	&\leq d^{-1}\cdot\exp\left(-\frac{1}{4d}(y-z)^2 + 2R^2T^{5/2}\right),
	\end{align*}
	the last inequality since $d\leq \sqrt T/24$ and $144/(4\times24)\leq 2$.

	So we see from \eqref{e.y to tilde y} and the above two bounds on $f_J(y,z)$ depending on the size of $|y-z|$, that
	$$f_J(y,z)\cdot S \leq \exp\left(21R^2T^{5/2}\right),$$
	using that $R\leq R^2$ and $d\geq 1$.

	 Thus, we obtain that \eqref{e.no pole integral} is bounded up to an absolute multiplicative constant by
	\begin{align}\label{e.no pole simplification}
	 \exp\left(21R^2T^{5/2}\right)\cdot \frac{1}{\sqrt{4\pi d}}\iint_{\R^2} \B^{[-d,d]}_{\tilde y, \tilde z}\Big(\tilde A\Big)e^{-\frac{1}{4d}(\tilde y-\tilde z)^2}\,\mathrm d\mu(y)\,\mathrm dz.
	 \end{align}
	Let $B$ be a Brownian motion started with distribution $\mu$ at $-d$. Focusing on the integral,
	\begin{align*}
	 \frac{1}{\sqrt{4\pi d}}\iint_{\R^2} \B^{[-d,d]}_{\tilde y, \tilde z}\Big(\tilde A\Big)e^{-\frac{1}{4d}(\tilde y-\tilde z)^2}\,\mathrm d\mu(y)\,\mathrm dz 
	&= \P\Big(B(\,\cdot\,) - B(-d)\in A\Big)
	= \B^{[-d,d]}_{0,*}\big(A\big) =\epsilon,
	\end{align*}
	using the Markov property of Brownian motion for the penultimate equality.
	
	Combining this with the ignored factor in \eqref{e.no pole simplification} and using that $R^2 = 36 d$ from \eqref{e.R value} gives that \eqref{e.no pole integral} is bounded above, up to an absolute multiplicative constant, by 
	$$\epsilon \cdot\exp\left(756\cdot D_k^{5/2}\cdot d\cdot(\log\epsilon^{-1})^{5/6}\right),$$
	since $21\times 36 = 756$.
\end{proof}

\chapter{The patchwork quilt resewn}\label{ch.patchwork quilt}

In this chapter we make precise the notion of Brownian regularity for weight profiles from general initial condition discussed in Section~\ref{s.intro general init condition Brownian regularity}. Then we prove that the general pre-limiting weight profiles in Brownian LPP enjoy this regularity; and, using the results of \cite{dauvergne2018directed}, that this inference passes over to the limiting profiles. Finally we state and prove in Corollary~\ref{c.precise increment moment bound} the precise version of Theorem~\ref{t.weight profile increment moment bound}, giving a bound, for any $\eta>0$, on the $(2-\eta)$ moment of the increment of the limiting weight profile.

\section{Brownian motion regularity of weight profile with general initial condition}\label{s.applications brownain motion}

In \cite{hammond2017patchwork}, a certain form of Brownian \emph{bridge} regularity is proved for the weight profile started from a very general initial condition.  With our Theorem~\ref{t.airytail.ln}, we are able to upgrade this result to be a form of Brownian \emph{motion} regularity, and here we establish the notation to state this result.

We need to introduce the notion of a \emph{quilt}, which is constructed from a \emph{fabric sequence} and a set of \emph{stitch points}. These are the same as the definitions from \cite{hammond2017patchwork}.

\begin{definition}
Let $\overline F = \{F_i \in \mc C_{*,*}([a,b], \R) : i\in \N\}$ be a sequence of random continuous functions defined on $[a,b]$, which we will call the \emph{fabric sequence}. Also let $S = \{s_1 < s_2< \ldots < s_N \}\subseteq [a,b]$ be an almost surely finite set, called the \emph{stitch point set}.

The \emph{quilt} constructed from $\overline F$ and $S$ is denoted ${\rm Quilt}[\overline F, S]$, and is a random continuous function on $[a,b]$ defined by
$${\rm Quilt}[\overline F, S](x) = F_i(x) + v_i, \qquad x\in[s_{i-1},s_i],$$
for $i=1, \ldots, N+1$, where $s_0 =a$ and $s_{N+1} = b$. We set $v_0=0$, and specify $v_i$ for $i\geq 1$ so that ${\rm Quilt}[\overline F, S]$ is continuous.
\end{definition}

We modify the definition of a collection of stochastic processes $\{X_{n,\alpha}\}$ being Brownian patchwork quiltable from \cite{hammond2017patchwork} so that the reference process is Brownian motion rather than Brownian bridge.

\begin{definition}\label{d.bm comparison}
Let $\beta \in (1,\infty)$. A random continuous function $X:[a,b]\to \R$ is said to \emph{withstand $L^{\beta-}$-comparison with Brownian motion} if, for any $\eta\in(0,1-\beta^{-1})$ and any measurable $A\subseteq \mc C_{0,*}([a,b], \R)$,
$$\P\bigl(X(\,\cdot\,) - X(a) \in A\bigr) \leq C_0\cdot \B_{0,*}^{[a,b]}(A)^{1-\beta^{-1}-\eta},$$
where $C_0$ is a finite constant which may depend upon $\eta$.


Suppose that $\mc I$ is an arbitrary index set and that we have a collection of random continuous functions indexed by $\N\times \mc I$. For a function $h:\N\to (0,1]$, the collection is said to \emph{uniformly} withstand $L^{\beta-}$-comparison to Brownian motion \emph{above scale $h$} if
$$\P\bigl(X_{n,\alpha}(\,\cdot\,) - X_{n,\alpha}(a) \in A\bigr) \leq C_0\cdot \B_{0,*}^{[a,b]}(A)^{1-\beta^{-1}-\eta},$$
for all $(n,\alpha)\in \N\times \mc I$ and all measurable $A\subseteq \mc C_{0,*}([a,b], \R)$ with $\B_{0,*}^{[a,b]}(A) \geq h(n)$; here, $C_0$ is a constant that may depend upon $\eta$ but not on $n$ nor on $\alpha$.
\end{definition}

\begin{definition}\label{d.patchwork quilt}
Let $\mc I$ be an arbitrary index set. 
Suppose we are given a collection of random continuous functions $X_{n,\alpha}:[a,b] \to \R$, indexed by $(n,\alpha) \in \N \times \mc I$, defined under the law~$\PP$. 
Let $\beta_1 > 0$, $\beta_2 \geq 1$, $\beta_3 > 0$ and $\beta_4 > 0$. This collection is said to be uniformly Brownian motion patchwork $(\beta_1,\beta_2,\beta_3,\beta_4)$-quiltable if there exist
 \begin{itemize}
 \item sequences $p,q :\N \to[0,1]$ verifying $p_j \leq j^{-\beta_1 + \e}$ and $q_j \leq j^{-\beta_3 + \e}$ for each $\e > 0$ and all $j$ sufficiently high; and

\item a constant $g > 0$  
 \end{itemize}
such that, for each $(n,\alpha) \in \N \times \mc I$, we may construct under the law $\PP$,
\begin{enumerate}
\item an error event $E_{n,\alpha}$ that satisfies $\PP \big( E_{n,\alpha} \big) \leq q_n$;
\item a {\em fabric} sequence $\overline{F}_{n,\alpha} = \big\{ F_{n,\alpha;i} : i \in \N \big\}$ (consisting of continuous random functions on $[a,b]$), where the 
collection $\big\{ F_{n,\alpha;i} : \big(n,(\alpha,i) \big) \in \N \times \mc{K} \big\}$, with $\mc{K} = \mc I \times \N$,   uniformly withstands  $L^{\beta_2-}$-comparison to Brownian motion above scale $\exp \big\{ - g n^{\beta_4} \big\}$; 
\item a {\em stitch points} set $S_{n,\alpha} \subset [a,b]$ whose cardinality verifies
$\PP \big( \vert S_{n,\alpha} \vert \geq \ell \big) \leq p_\ell$ for each $\ell \in \N$; and
\item all this in such a way that, for every $(n,\alpha) \in \N \times \mc I$, the random function  $X_{n,\alpha}$ is equal to the patchwork quilt ${\rm Quilt}[\overline{F}_{n,\alpha},S_{n,\alpha}]$ throughout the interval $[a,b]$, whenever the error event $E_{n,\alpha}$ does {\em not} occur.
\end{enumerate}
\end{definition}

We may now state the theorem establishing this form of Brownian motion regularity for the pre-limiting weight
 profiles from general initial conditions.

\begin{theorem}\label{t.patchwork quilt}
Let $\bar\Psi \in(0,\infty)^3$ satisfy $\Psi_2\geq 1$. The collection of random continuous functions 
$$[-1,1] \to \R : y \mapsto \weight_{n}^f[(*,0)\to (y,1)]$$
indexed by $(n,f)\in \N\times \mc I_{\bar \Psi}$ is uniformly Brownian motion patchwork $(2,3,1/252,1/12)$-quiltable.
\end{theorem}

\cite[Theorem 1.2]{hammond2017patchwork} is exactly the same as Theorem~\ref{t.patchwork quilt}, except with Brownian bridge comparison, with the obvious change in Definition~\ref{d.bm comparison}.

\begin{proof}[Proof of Theorem~\ref{t.patchwork quilt}]
Our aim is to replace the use of \cite[Theorem 4.3]{hammond2017patchwork} in the proof of \cite[Theorem 1.2]{hammond2017patchwork} with Theorem~\ref{t.airytail.ln}. To begin, we simply change $A$ from being a subset of $\mc C_{0,0}([-1,1], \R)$ to be a subset of $\mc C_{0,*}([-1,1], \R)$, and replace all the occurrences of affinely shifted processes $Z_{n,j}^{[-1,1]}$ and $Y_{n,i}^{[-1,1]}$ with vertically shifted processes $Z_{n,j}(\,\cdot\,) - Z_{n,j}(-1)$ and $Y_{n,i}(\,\cdot\,) - Y_{n,i}(-1)$. Theorem 4.3 of \cite{hammond2017patchwork} is used only on page~59 of \cite{hammond2017patchwork} in the last display, and the proof works with our definition of uniformly Brownian motion patchwork quiltable by replacing this usage of \cite[Theorem 4.3]{hammond2017patchwork} with Theorem~\ref{t.airytail.ln}.
\end{proof}

\section{The limiting patchwork quilt}\label{s.applications limit of patchwork quilt}

In the recent work of Dauvergne, Ortmann, and Vir\'ag \cite{dauvergne2018directed}, the existence of the space-time Airy sheet is established. With some straightforward arguments that we will attend to shortly, this establishes the unambiguous existence of the random function $y\mapsto \weight_{\infty}^f[(*,0)\to (y,1)]$, and it is easy to believe, as remarked in \cite{hammond2017patchwork}, that a theorem analogous to Theorem~\ref{t.patchwork quilt} should hold for the limiting process with a suitable analysis of weak convergence of the relevant objects. This is indeed the case, and in this subsection we carry out the needed task in Proposition~\ref{p.limit is quiltable}.

The following is the result from \cite{dauvergne2018directed} establishing the joint convergence of the weight profiles, stated in our notation:

\begin{theorem}[Theorem 1.3 of \cite{dauvergne2018directed}]\label{t.dov}
The process $(x,y)\mapsto \weight_{n}[(x,0)\to(y,1)]$ for $(x,y)\in\R^2$ has a unique limit in distribution as $n\to\infty$, denoted by $\weight_{\infty}[(x,0)\to(y,1)]$. The convergence is on the space of continuous functions $f: \R^2\to \R$ endowed with the topology of locally uniform convergence.
\end{theorem}

This theorem establishes the existence of the limiting weight profile started from the narrow-wedge initial condition (i.e., $f(x) = 0$ and $f(z) = -\infty$ for $z\neq x$) jointly as $x$ varies, and first we must perform the straightforward task of using this to obtain the existence of limiting weight profiles for general initial conditions. 

\begin{proposition}\label{p.limit of general weight profile}
Let $f\in \mc I_{\bar \Psi}$. Then the process $y\mapsto \weight^f_{n}[(*,0)\to (y,1)]$ for $y\in [-1,1]$ has a unique limit in distribution---to be denoted by $\weight_{\infty}^f[(*,0)\to (y,1)]$---on $\mc C_{*,*}([-1,1], \R)$ endowed with the topology of uniform convergence as $n\to\infty$.
\end{proposition}

Before giving the proof, we record some statements from \cite{hammond2017patchwork} which give us some control over the maximiser  in the definition of $\weight_{n}^f[(*,0)\to (y,1)]$. Let $x^n_y = \mathrm{argmax}\left\{\weight_{n}[(x,0)\to(y,1)] + f(x) \ :\  x\in \R\right\}$ for $y\in\{-1,1\}$ and for a fixed $f\in \mc I_{\bar \Psi}$. Note that $x^n_1$ and $x^n_{-1}$ are well defined since this maximiser is unique a.s. for every fixed $y\in[-1,1]$.

\begin{lemma}[Lemma 4.10 of \cite{hammond2017patchwork}]\label{l.maximiser unique}
Fix $f\in \mc I_{\bar \Psi}$, and let $x^n_y$ be as above. Then $\{x^n_1, x^n_{-1}\}_{n\in\N}$ is a tight sequence of random variables. 
\end{lemma}

\begin{lemma}\label{l.sandwich}
In the notation of Lemma~\ref{l.maximiser unique}, $y\mapsto x^n_y$ is a.s. a  non-decreasing function for each $n$.
\end{lemma}

\begin{proof}
This is an immediate consequence of a simple sandwiching property of polymers recorded in \cite[Lemma 4.4]{hammond2017patchwork}. 
\end{proof}

This sandwiching property just used is a common tool in arguments involving polymer geometry in other LPP models as well, and has also been called ``polymer ordering''.

\begin{proof}[Proof of Proposition \ref{p.limit of general weight profile}]
By Theorem \ref{t.dov} and the Skorokhod representation theorem, we have on a common probability space that
$\weight_{n}[(x,0)\to(y,1)] \to \weight_{\infty}[(x,0)\to(y,1)]$
uniformly on compact sets almost surely. Define 
$$\weight^f_{\infty}[(*,0)\to (y,1)] := \sup\Big\{\weight_{\infty}[(x,0)\to(y,1)] + f(x) \ :\  x\in(-\infty, \infty)\Big\}.$$
As in Lemma~\ref{l.maximiser unique}, let 
\begin{align*}
x_1^{n} &= \mathrm{argmax}\Bigl\{\weight_{n}[(x,0)\to(1,1)] + f(x) \ :\  x\in \R\Bigr\} \quad \text{and}\\
x_{-1}^n &=\mathrm{argmax}\Bigl\{\weight_{n}[(x,0)\to(-1,1)] + f(x) \ :\  x\in \R\Bigr\}.
\end{align*}
Since $\{x_1^n\}_{n\in\N}$ and $\{x_{-1}^n\}_{n\in\N}$ are tight, let $\{x_{1}^{n_k}\}_{k\in\N}$ and $\{x_{-1}^{n_k}\}_{k\in\N}$ be convergent subsequences.

Let $K$ be a random compact interval such that, for all $k$, $x^{n_k}_{1}, x^{n_k}_{-1}\in K$; such a $K$ exists because the sequences are convergent. Note that Lemma~\ref{l.sandwich} implies that $x^{n_k}_y\in K$ for all $y\in[-1,1]$ and large enough $k$ almost surely. From uniform convergence on compact sets, it follows that $x^{n_k}_1 \to x^{\infty}_1\in K$, where $x^\infty_1$ is a maximiser of the supremum in the definition of $\weight^f_{\infty}[(*,0)\to (1,1)]$, and similarly for $x^{n_k}_{-1}$. Again by uniform convergence on compact sets, we obtain
\begin{align*}
\MoveEqLeft
\left|\,\weight^f_{n_k}[(*,0)\to (y,1)] - \weight^f_{\infty}[(*,0)\to (y,1)]\right| \\
&\leq \sup_{(x,y)\in K\times[-1,1]} \Big|\,\weight_{n_k}[(x,0)\to(y,1)] - \weight_{\infty}[(x,0)\to(y,1)]\Big| \to 0,
\end{align*}
as $k\to\infty$. This was for the subsequence $n_k$ along which $x_{1}^{n_k}$ and $x_{-1}^{n_k}$ converge, but observe that given any subsequence of $\N$, we can find a further subsequence $n_k$ along which $x_1^{n_k}$ and $x_{-1}^{n_k}$ converge, and the argument goes through. It is standard that this implies the claim of Proposition~\ref{p.limit of general weight profile}, completing the proof.
\end{proof}

Now we turn to the result that the patchwork quilt description is preserved in the limit. Note that, when considering the Brownian motion patchwork quiltability of a limiting family of processes $\{X_{\infty,\alpha}\}_{\alpha\in \mc I}$ for some index set $\mc I$, i.e., $n=\infty$, the constants $\beta_3$ and $\beta_4$ do not play any role; this is because $\beta_3$ controls the probability of the error event $E_n$, which is zero when $n=\infty$ for any $\beta_3>0$, while $\beta_4$ controls the scale above which the comparison to Brownian motion is made, which is also zero when $n=\infty$ for any $\beta_4>0$. So as a formal device, we may set both the parameters $\beta_3$ and $\beta_4$ to $\infty$ for the limiting family of processes, regardless of their values in the prelimiting family.

\begin{proposition}\label{p.limit is quiltable}
Suppose the collection $\{X_{n,\alpha}:[a,b]\to\R : (n,\alpha)\in\N\times \mc I\}$ is uniformly Brownian motion patchwork $(\beta_1,\beta_2,\beta_3,\beta_4)$-quiltable, and that $\{X_{n,\alpha}\}_{n\in \N}$ is a tight sequence of random functions for each $\alpha\in\mc I$. Let $\{X_{\infty,\alpha}\}_{\alpha\in\mc I}$ be a collection of weak limit points of this collection. Then $\{X_{\infty,\alpha}\}_{\alpha\in\mc I}$ is uniformly Brownian motion patchwork $(\beta_1,\beta_2,\infty, \infty)$-quiltable.
\end{proposition}

As an immediate implication of this proposition and Theorem~\ref{t.patchwork quilt}, we have the following theorem.

\begin{theorem}\label{t.limiting weight profile is quiltable}
Let $\bar\Psi \in(0,\infty)^3$ satisfy $\Psi_2\geq 1$. The collection of random continuous functions 
$$[-1,1] \to \R : y \mapsto \weight_{\infty}^f[(*,0)\to (y,1)]$$
indexed by $f\in \mc I_{\bar \Psi}$ is uniformly Brownian motion patchwork $(2,3,\infty, \infty)$-quiltable.
\end{theorem}

\begin{proof}[Proof of Proposition~\ref{p.limit is quiltable}]
Looking back at the definition of a uniformly Brownian motion patchwork $(\beta_1, \beta_2, \infty,\infty)$-quiltable collection when $n=\infty$, to prove Proposition~\ref{p.limit is quiltable} we must find, for each $\alpha\in\mc I$, a collection of random continuous functions $\overline F_{\alpha} = \{F_{\alpha;i} : i\in \N\}$ defined on $[a,b]$ and a stitch point set $S_\alpha$ such that
\begin{itemize}
	\item the collection $\{F_{\alpha;i} : (\alpha,i)\in \mc I\times \N\}$ uniformly withstands $L^{\beta_2-}$-comparison to Brownian motion above scale zero;
	\item it holds that $\P(|S_{\alpha}| \geq \ell) \leq p_{\ell}$ for each $\ell\in \N$, with $p_{j} \leq j^{-\beta_1+\epsilon}$ for each $\epsilon>0$ and $j$ sufficiently large; and
	\item in such a way that $X_{\infty, \alpha} = {\rm Quilt}[\overline F_{\alpha}, S_{\alpha}]$ on $[a,b]$ almost surely.
\end{itemize}
To do this we take appropriate weak limits of the collections $\overline F_{n,\alpha}$ and sets $S_{n,\alpha}$ for every $\alpha$, which we now fix for the remainder of the proof. We assume without loss of generality that $X_n\stackrel{d}{\to} X_\infty$, focusing our attention, if needed, on the subsequence along which this convergence occurs.

We start by fixing $\alpha$ and showing that there exists a subsequence of $n$ along which $F_{n,\alpha, i} \stackrel{d}{\to} F_{\infty,\alpha,i}$ for all $i$ simultaneously, where $\{F_{\infty,\alpha,i} : i\in \N\}$ is some collection of random continuous functions. Since the construction of the patchwork quilt entails a vertical shift of all $F_{n, \alpha, i}$ for $i\geq 2$, we may assume that $F_{n,\alpha, i}(a) = 0$ for all $i\geq 2$ and all $n\in \N$. The tightness of this collection follows from a standard diagonalisation argument if we can show the subsequential convergence for each individual $i$, i.e., we must show tightness of $\{F_{n,\alpha, i} : n\in\N\}$ for each fixed $i\in\N$. This is a consequence of the uniform $L^{\beta_2-}$ comparison to Brownian motion that $\overline F_{n,\alpha}$ enjoys. So let $K\subseteq \mc C_{0,*}([a,b], \R)$ be a compact set such that $\B_{0,*}^{[a,b]}(K^c) \leq \epsilon$. Then using the $L^{\beta_2-}$-comparison with $\eta = (1-\beta_2^{-1})/2$, we get, for $i\geq 2$,
$$\P\bigl(F_{n, \alpha, i} \not\in K\bigr) \leq C_0\cdot \B_{0,*}^{[a,b]}\bigl(K^c\bigr)^{(1-\beta_2^{-1})/2} \leq C_0\cdot \epsilon^{(1-\beta_2^{-1})/2}.$$
Since $\epsilon>0$ is in our control, this establishes tightness of $\{F_{n,\alpha,i} : n\in \N\}$ for each $i\geq 2$. 

For $i=1$ we cannot assume $F_{n,\alpha,1}(a) = 0$, and so we must use the tightness of $\{X_{n,\alpha}\}_{n\in\N}$ as well. Recall that to establish tightness of $\{F_{n,\alpha,1} : n\in \N\}$, it is sufficient to show that $\{F_{n,\alpha,1}(a) : n\in \N\}$ is tight, and to show a uniform modulus of continuity; i.e., that, given $\rho>0$ and $\epsilon>0$, we can find $r$ and $n_0$ such that for $n>n_0$ we have
$$\P\big(\omega(F_{n,\alpha,1}, r) \leq \rho\big) \geq 1-\epsilon,$$
where for a continuous function $f:[a,b]\to\R$,
$$\omega(f,r) = \sup_{\substack{|x-y|<r\\ x,y\in[a,b]}} |f(x) - f(y)|.$$
That $\{F_{n,\alpha,1}(a) : n\in \N\}$ is tight follows from the fact that $F_{n,\alpha,1}(a) = X_n(a)$ on the event $E_{n,\alpha}^c$, which has probability less than $q_n$, and the assumed tightness of $\{X_n\}_{n\in\N}$. The uniform modulus of continuity follows again from $L^{\beta_2-}$-comparison with Brownian motion, and the fact that the modulus of continuity event does not change by considering $F_{n,\alpha,1}(\,\cdot\,) - F_{n,\alpha,1}(a)$ instead of $F_{n,\alpha,1}(\,\cdot\,)$:
$$\P\big(\omega(F_{n,\alpha,1}, r) \geq \rho\big) \leq C_0\cdot \P\big(\omega(B, r) \geq \rho\big)^{(1-\beta^{-1})/2},$$
where $B$ is a Brownian motion on $[a,b]$ and $C_0$ is a constant independent of $n$ and $\alpha$. Since this right-hand side can be made less than $\epsilon$ by taking $r$ sufficiently small, and since this inequality holds for all large enough~$n$, the tightness of $\{F_{n,\alpha,1} : n\in\N\}$ follows.

Having established convergence of $F_{n,\alpha,i}$ along a subsequence simultaneously for all $i$, we may assume without loss of generality that this subsequence is the entire sequence. We now turn to the collection $\{S_{n,\alpha}\}$. This is a collection of random finite sets, and regarding them as $\{0,1\}$-valued random measures (i.e., as simple point processes), we see that $\{S_{n,\alpha}\}_{n\in\N}$ is a tight sequence of random measures, as these laws are all defined on the compact set $[a,b]$.

At this point we have assumed $X_{n,\alpha} \stackrel{d}{\to} X_{\infty,\alpha}$ as $n\to \infty$ and have established tightness of $\{S_{n,\alpha} : n\in \N\}$ as well as that $\overline F_{n,\alpha} \stackrel{d}{\to} \overline F_{\infty,\alpha}$ as $n\to\infty$. It is an immediate consequence that there exists a subsequence $n_k$ such that $(X_{n_k,\alpha}, \overline F_{n_k,\alpha}, S_{n_k,\alpha}) \stackrel{d}{\to} (X_{\infty,\alpha}, \overline F_{\infty,\alpha}, S_{\infty,\alpha})$, where $S_{\infty,\alpha}$ is some integer-valued random measure; we make no statement about the joint distribution of the limiting triple. We have to establish the properties of these objects contained in the three bullet points at the beginning of the proof, but to do so we will need to slightly modify $\overline F_{\infty,\alpha,i}$ and $S_{\infty,\alpha}$.

Note that it may not be the case that we can identify $S_{\infty,\alpha}$ with a finite set, as $S_{\infty,\alpha}$ may assign integer measure greater than 1 to singletons. So we define $S_\alpha$ to be the finite set support of $S_{\infty,\alpha}$. Now since $|S_{\alpha}|$ and $|S_{\infty,\alpha}|$ are integer-valued random variables (where the $|S_{\infty,\alpha}|$ is the measure assigned to $[a,b]$ by the associated integer-valued measure), we have, by the Portmanteau theorem,
\begin{equation}\label{e.S_alpha bound}
\P\left(|S_\alpha| \geq \ell\right) \leq \P\left(|S_{\infty, \alpha}| \geq \ell\right) =\lim_{n\to\infty} \P\left(|S_{n, \alpha}| \geq \ell\right) \leq p_{\ell},
\end{equation}
thus verifying the second bullet point from the beginning of the proof.


Because of the possibility that points in $S_{\infty,\alpha}$ have multiplicity greater than 1, we will later remove some entries of $\overline F_{\infty,\alpha}$ to obtain the final $\overline F_\alpha$.

Now we verify the first bullet point, namely that $\overline F_{\infty, \alpha}$ is uniformly $L^{\beta_2-}$-comparable to Brownian motion. Fix $i\in\N$, $\eta \in (0, 1-\beta_2^{-1})$, and let $A$ be a given measurable subset of $\mc C_{0,*}([a,b], \R)$ such that $\B_{0,*}^{[a,b]}(A) > 0$. Since all Borel probability measures on Polish spaces are regular, in particular outer regular, we have that
$$ \P\left(F_{\infty, \alpha, i} \in A\right) = \inf\big\{\P\left(F_{\infty, \alpha, i} \in O\right) : A\subseteq O \text{ open}\big\}.$$
Since the $O$ are open, we have by the Portmanteau theorem
\begin{align}
\inf\Big\{\P\left(F_{\infty, \alpha, i} \in O\right) : A\subseteq O \text{ open}\Big\} 
&\leq \inf\left\{\lim_{n\to\infty}\P\left(F_{n, \alpha, i} \in O\right) : A\subseteq O \text{ open}\right\} \nonumber\\
&\leq \inf\left\{C_0\cdot\left(\B_{0,*}^{[a,b]}\left(O\right)\right)^{1-\beta_2^{-1}-\eta} : A\subseteq O \text{ open}\right\} \nonumber\\
&= C_0\cdot\left(\inf\left\{\left(\B_{0,*}^{[a,b]}\left(O\right)\right) : A\subseteq O \text{ open}\right\}\right)^{1-\beta_2^{-1}-\eta} \nonumber\\
&= C_0\cdot\left(\B_{0,*}^{[a,b]}\left(A\right)\right)^{1-\beta_2^{-1}-\eta}. \label{e.uniform brownian comparison}
\end{align}
Here $C_0$ is an $\eta$-dependent constant independent of $n,\alpha, i$ which comes (for large enough $n$) from the uniform $L^{\beta_2-}$-comparison to Brownian motion enjoyed by the collection $\{F_{n,\alpha,i} : n\in \N, \alpha\in \mathcal I, i\in \N\}$. In the last equality we have used that $\B_{0,*}^{[a,b]}$ is an outer regular measure. This establishes the uniform $L^{\beta_2-}$-comparison of $F_{n,\alpha,i}$ to Brownian motion above scale zero for every $i$.

Finally we define $\overline F_{\alpha}$ so that the third bullet point, i.e., $X_{\infty, \alpha} = {\rm Quilt}[\overline F_\alpha, S_\alpha]$ on $[a,b]$, holds almost surely. Label the points in $S_\alpha$ as $\{s_1, \ldots, s_N\}$; let $m(s_i)$ be the multiplicity of $s_i$ in $S_{\infty, \alpha}$; and let $r(i) = \sum_{1\leq j\leq i} m(s_i)$. Then we let
$$F_{\alpha, i} = F_{\infty, \alpha, r(i)}$$
for $i=1, \ldots, N$. For $i>N$, $F_{\alpha, i} = F_{\infty,\alpha,i}$. In words, if the multiplicity of $s_i$ is greater than 1, this means that a corresponding number of entries of $F_{\infty, \alpha, i}$ were squished together into an interval of size zero in the stitching, and so must be removed from the fabric collection.

As we have $(X_{n_k,\alpha}, \overline F_{n_k,\alpha}, S_{n_k,\alpha}) \stackrel{d}{\to} (X_{\infty,\alpha}, \overline F_{\infty,\alpha}, S_{\infty,\alpha})$, by the Skorokhod representation theorem we may assume we are working on a probability space where this convergence happens almost surely. Then it is immediate from the fact that $X_n = {\rm Quilt}[\overline F_{n,\alpha}, S_{n,\alpha}]$ on $E_n^c$ and from our definitions of $\overline F_{\alpha}$ and $S_{\alpha}$ that $X_{\infty}={\rm Quilt}[\overline F_{\alpha}, S_\alpha]$ almost surely.
 Since the right-hand sides of equations \eqref{e.S_alpha bound} and \eqref{e.uniform brownian comparison}  are independent of $\alpha\in \mc I$, this establishes Proposition~\ref{p.limit is quiltable}.
\end{proof}

With the existence of the limiting weight profile and the knowledge that it enjoys a description as a Brownian motion patchwork quilt, we may prove the next corollary. It is a stronger version of Theorem~\ref{t.weight profile increment moment bound}, stated with the constants that are uniform once the space of admissible initial conditions has been fixed.

\begin{corollary}\label{c.precise increment moment bound}
Let $f\in \mc I_{\bar \Psi}$ and $0<\eta<\frac{1}{2}$. Then there exist constants $G = G(\bar\Psi, \eta)<\infty$ and $y_0>0$ such that, for $|y|<y_0$,
$$\E\left[\left|\weight^f_{\infty}[(*,0)\to (y,1)] - \weight^f_{\infty}[(*,0)\to (y,1)]\right|^{2-\eta}\right] \leq G |y|^{1-\eta/2}.$$
\end{corollary}

In the case where $f\equiv 0$, corresponding to flat initial conditions, $\weight^f_{\infty}[(*,0)\to (y,1)]$ is strongly believed to be the Airy$_1$ process (in analogy with the exponential and Poissonian LPP models where such a result has been proven), though we have been unable to find a proof in the literature.

As was qualitatively described in Section~\ref{s.intro general init condition Brownian regularity}, to upgrade this result from the $2-\eta$ moment to the second moment would require the knowledge that the first parameter of $2$ in the Brownian motion patchwork quilt parameters in Theorem~\ref{t.patchwork quilt} may be increased to $2+\rho$ for some $\rho>0$.

\begin{proof}[Proof of Corollary~\ref{c.precise increment moment bound}]
We assume without loss of generality that $y>0$, as the same argument works for the other case. The constant $G$ may vary from line to line, and may depend only on $\eta$ and $\bar\Psi$. For ease of notation, we will write $\weight(y)$ for $\weight^f_{\infty}[(*,0)\to (y,1)]$. Now, by Theorem~\ref{t.patchwork quilt} and Propositions~\ref{p.limit of general weight profile} and \ref{p.limit is quiltable}, we may write
$$\weight(y) = \mathrm{Quilt}[\overline F, S](y),$$
and we label the points of the stitch point set $S$ which are strictly smaller than $y$ as $s_1<s_2< \ldots <s_{N-1}  < y$. We set $s_0 = 0$ and $s_{N} = y$ by convention, though they need not lie in $S$. Then, using that the convexity of the mapping $x\mapsto x^{2-\eta}$ implies $\left(\sum_{i=1}^n a_i\right)^{2-\eta}\leq n^{1-\eta} \sum_{i=1}^na_i^{2-\eta}$, we get
\begin{align*}
\E\left[\bigl|\weight(y) - \weight(0)\bigr|^{2-\eta}\right] = \E\left[\left|\sum_{i=1}^{N} \bigl(F_i(s_i) - F_i(s_{i-1})\bigr)\right|^{2-\eta}\right]
&\leq \E\left[\left(\sum_{i=1}^{N} \big|F_i(s_i) - F_i(s_{i-1})\big|\right)^{2-\eta}\right]\\
&\leq \E\left[N^{1-\eta} \sum_{i=1}^{N}\big|F_i(s_i) - F_i(s_{i-1})\big|^{2-\eta}\right].
\end{align*}
We write the last sum as $\sum_{i=1}^\infty \E\left[N^{1-\eta}\big|F_i(s_i) - F_i(s_{i-1})\big|^{2-\eta} \one_{i\leq N}\right]$. We will bound each individual summand using the generalised H\"older's inequality, with $p=2+2\eta, q=2-\eta,$ and $r$ defined by 
$$\frac{1}{r} = 1-\frac{1}{2+2\eta} - \frac{1}{2-\eta},$$
which is positive for $0<\eta<1/2$. Doing so, we get that $\E\left[N^{1-\eta}\left|F_i(s_i) - F_i(s_{i-1})\right|^{2-\eta} \one_{i\leq N}\right]$ is bounded by
\begin{equation}\label{e.generalized holder result}
\E\left[N^{(1-\eta)p}\right]^{1/p}\cdot\P\left(N\geq i\right)^{1/q}\cdot \E\left[|F_i(s_i) - F_i(s_{i-1})|^{(2-\eta)r}\right]^{1/r}.
\end{equation}
We require this to be summable and bounded by $Gy^{1-\eta/2}$. Note that $(1-\eta) p = 2(1-\eta^2)<2$; since the tail bound on $N$ ensures it has finite $2-\epsilon$ moment for any $\epsilon>0$, the first factor in \eqref{e.generalized holder result} is finite and bounded by an $\eta$-dependent quantity.

Now we look at the third factor of \eqref{e.generalized holder result}. This is clearly bounded by
$$G\cdot \E\left[\sup_{s\in[0,y]} |F_i(s)-F_i(0)|^{(2-\eta)r}\right]^{1/r},$$
as $s_i\leq y$. Let $Z$ be the Radon-Nikodym derivative of $F_i$ with respect to standard Brownian motion on $[0,1]$, denoted $B$, so that we know that $Z$ has finite $3-\epsilon$ moment for any $\epsilon>0$. Thus we get
\begin{align*}
\E\left[\sup_{s\in[0,y]} |F_i(s)-F_i(0)|^{(2-\eta)r}\right] &= \E\left[\sup_{s\in[0,t]} |B(s)|^{(2-\eta)r}Z\right]
\leq \E\left[\sup_{s\in[0,y]} |B(s)|^{3(2-\eta)r}\right]^{1/3}\cdot \E[Z^{3/2}]^{2/3},
\end{align*}
again using H\"older's inequality with H\"older conjugates $3$ and $3/2$. Using the symmetry of Brownian motion and the reflection principle, we obtain 
$$\E\left[\sup_{s\in[0,y]} |B(s)|^{3(2-\eta)r}\right] \leq G\cdot\E\left[|X|^{3(2-\eta)r}\right] \leq Gy^{3(1-\eta/2)r},$$
where $X$ is a normal random variable with mean 0 and variance $y$, and the last inequality is by a standard expression for the moments of the normal distribution. Tracing the steps back, we see that the third factor of \eqref{e.generalized holder result} is bounded by $Gy^{(1-\eta/2)}$, since $\E[Z^{3/2}]$ is finite and depends only on $\bar\Psi$.

Thus all that remains to be shown is that the second factor of \eqref{e.generalized holder result} is summable in $i$. Taking the $\epsilon$ in Definition~\ref{d.patchwork quilt} to be $\eta/2$, we see that for large enough $i$, we have
$$\P\left(N\geq i\right)^{1/(2-\eta)} \leq i^{-(2-\eta/2)/(2-\eta)},$$
which is summable. This completes the proof of Corollary~\ref{c.precise increment moment bound}.
\end{proof}

\bibliographystyle{alphaurl}
\bibliography{bmreg}

\appendix

\chapter{Brownian meander calculations}\label{a.brownian meander calculations}

In this appendix we carry out the calculations concerning Brownian meander whose conclusions were used in Chapter~\ref{ch.application proofs}; these calculations were recorded as Lemmas~\ref{l.near zero meander}, \ref{l.brownian meander from 0 estimate}, \ref{l.brownian meander increment estimate}, and \ref{l.prob of return from above}. We adopt the notation for the density and modified distribution function of mean zero normal random variables with variance $\sigma^2$:
$$\varphi_{\sigma^2}(x) = \frac{1}{\sqrt{2\pi}\sigma}\exp\left(-\frac{x^2}{2\sigma^2}\right) \quad \text{and} \quad \tilde\Phi_{\sigma^2}(x) = \int_0^x \varphi_{\sigma^2}(w)\, \mathrm dw.$$

We recall that the standard Brownian meander $\bme$ is a non-homogeneous Markov process on $[0,1]$, started at 0, which can be obtained as a weak limit, as $\epsilon \searrow 0$, of standard Brownian motion on $[0,1]$ conditioned to stay above $-\epsilon$ \cite{durrett1977weak}. Its transition probabilities are given---see, for example, \cite{durrett1977weak}---for $x,y>0$ and $0<s<t\leq 1$ by
\begin{align*}
\P\Big(\bme(t) \in \mathrm dy \mid \bme(s) = x\Big) &= \Big(\varphi_{t-s}(y-x) - \varphi_{t-s}(y+x)\Big)\frac{\tilde\Phi_{1-t}(y)}{\tilde\Phi_{1-s}(x)}\,\mathrm dy,\\
\P\Big(\bme(t)\in \mathrm dy\Big) &= 2\sqrt{2\pi}\cdot\frac{y}{t}\cdot \varphi_t(y)\tilde\Phi_{1-t}(y)\, \mathrm dy.
\end{align*}

Before turning to the proof of Lemma~\ref{l.near zero meander}, we state and prove the following simple estimate.

\begin{lemma}\label{l.conditioned inf distribution bound}
Let $X\geq 0$ be a random variable with density $f$, and suppose that $f$ is non-increasing on $[0,\infty)$. Then we have, for all $x,y\in\R$ with $0\leq x<y$,
$$\P\bigl(X\geq x\mid X\leq y\bigr) \leq \frac{y-x}{y}.$$
\end{lemma}

\begin{proof}
Let $F$ be the distribution function of $X$. We have that
$$\P\bigl(X\geq x\mid X\leq y\bigr) = \frac{\P\left(x\leq X\leq y\right)}{\P\left(X\leq y\right)} = \frac{F(y)-F(x)}{F(y)} = 1-\frac{F(x)}{F(y)}.$$
We wish to show that $F(x)/F(y)\geq x/y$ for all $x\leq y$, which, by rearranging, is equivalent to showing that the function $x\mapsto F(x)/x$ is non-increasing. By considering the derivative of this function, we see that the desired bound is equivalent to showing that for all $x\geq 0,$ $xf(x)\leq F(x).$

It is easy to see that our assumptions imply this inequality holds, for we have
$$F(x)\geq F(x)-F(0) = \int_0^x f(x)\,\mathrm dx \geq xf(x),$$
as $f$ is non-increasing on $[0,\infty)$.
\end{proof}

\begin{proof}[Proof of Lemma~\ref{l.near zero meander}]
We may assume $\eta\leq 1$ as we are working on $[0,1]$. We use the well-known fact that the distribution of $\bme$ on an interval $[r,1]$, conditionally on $\bme(r) = x$, is that of a Brownian motion on $[r,1]$ started at $x$ and conditioned to stay positive which is independent of $\bme$ on $[0,r]$. Thus, we have that
\begin{align}
\P\big(\nz(\bme, \eta) \mid \bme(\eta)\big) &= 
\P\left(\inf_{x\in[\eta,1]} B(x) < a\eta^{1/2}-\bme(\eta)\ \Big| \  \bme(\eta), \inf_{[\eta,1]}B(x) \geq -\bme(\eta)\right)\cdot\one_{\bme(\eta)\geq a\eta^{1/2}}\nonumber\\
&\qquad +\one_{\bme(\eta)\leq a\eta^{1/2}}, \label{e.nz breakup}
\end{align}
where $B$ is a standard Brownian motion started at 0. The reflection principle for Brownian motion asserts that $\inf_{x\in[\eta,1]} B(x) \stackrel{d}{=} -|N(0,1-\eta)|$, where $N(0,\sigma^2)$ is a normal random variable with mean zero and variance $\sigma^2$. Thus, letting $Z$ be distributed as $|N(0,1-\eta)|$, we may write
\begin{align*}
\MoveEqLeft[10]
\P\left(\inf_{x\in[\eta,1]} B(x) < a\eta^{1/2}-\bme(\eta) \ \big| \  \bme(\eta), \inf_{[\eta,1]}B(x) \geq -\bme(\eta)\right)\\
&= \P\Big(Z > \bme(\eta)-a\eta^{1/2} \ \big| \  \bme(\eta), Z \leq \bme(\eta)\Big)
\leq \frac{a\eta^{1/2}}{\bme(\eta)};
\end{align*}
here we were able to use Lemma~\ref{l.conditioned inf distribution bound}, as the law of $Z$ clearly satisfies this lemma's hypotheses.
We now take expectations in \eqref{e.nz breakup}, using that $a\eta^{1/2}/\bme(\eta)\geq 1$ when $\bme(\eta)\leq a\eta^{1/2}$ and that the density of $\bme(\eta)$ at $x$ is $2\sqrt{2\pi}\eta^{-1} x\varphi_\eta(x)\tilde\Phi_{1-\eta}(x)$, to obtain 
\begin{align*}
\P\left(\nz(\bme, \eta)\right) \leq
\int_0^{\infty}2\sqrt{2\pi}\frac{x}{\eta}\varphi_\eta(x)\tilde\Phi_{1-\eta}(x) \cdot \frac{a\eta^{1/2}}{x} \, \mathrm dx
&= 2\sqrt{2\pi} a\eta^{-1/2}\int_0^{\infty} \varphi_\eta(x)\tilde\Phi_{1-\eta}(x) \, \mathrm dx \\
&\leq \sqrt{2\pi} a\eta^{-1/2} \E[|N(0,\eta)|] = 2a, 
\end{align*}
where we used that $\tilde\Phi_{1-\eta}(x)\leq 1/2$ for all $x$ and also that $\E|N(0,\eta)| = \eta^{1/2}\E|N(0,1)| = \sqrt{2/\pi}\eta^{1/2}$.
\end{proof}

\begin{proof}[Proof of Lemma~\ref{l.brownian meander from 0 estimate}]
\begin{align*}
\P\big(\bme(\eta) < 1.1\eta^{1/2}\big) =\frac{2\sqrt{2\pi}}{\eta}\int_0^{1.1\eta^{1/2}} y\varphi_{\eta}(y)\tilde\Phi_{1-\eta}(y)\, \mathrm dy.
\end{align*}
Note that 
$$\tilde\Phi_{1-\eta}(y) \leq (2\pi(1-\eta))^{-1/2}y \quad \text{and} \quad \varphi_\eta(y) \leq (2\pi \eta)^{-1/2}.$$
Substituting these in, we get
\begin{align*}
\P\big(\bme(\eta) < 1.1\eta^{1/2}\big)\leq \frac{2\sqrt{2\pi}}{2\pi\eta^{3/2}(1-\eta)^{1/2}}\int_0^{1.1\eta^{1/2}}y^2\, \mathrm dy = \frac{(1.1)^3}{3\sqrt\pi}< \frac{1}{2},
\end{align*}
the last inequality by numerical evaluation. We have used that $\eta\leq 1/2$.
\end{proof}

\begin{proof}[Proof of Lemma~\ref{l.brownian meander increment estimate}]
We need to bound $\P\big(\bme(t) < 1.1\eta^{1/2} \mid \bme(t-\eta) = x\big)$ for all values of $x\leq 2\eta^{1/2}$ and $t\in(\eta,1-4\eta)$. Let $s=t-\eta$ and note that $s\leq 1-5\eta$.

We have
\begin{align}
\P\Big(\bme(t) < 1.1\eta^{1/2} \ \Big| \  \bme\left(s\right)=x\Big)
& = \int_0^{1.1\eta^{1/2}} (\varphi_\eta(y-x) - \varphi_\eta(y+x))\cdot \frac{\tilde\Phi_{1-t}(y)}{\tilde\Phi_{1-s}(x)}\, \mathrm dy\nonumber\\
&= \int_0^{1.1\eta^{1/2}} \varphi_\eta(y-x)\left(1 - \exp(-2xy/\eta)\right)\cdot \frac{\tilde\Phi_{1-t}(y)}{\tilde\Phi_{1-s}(x)}\, \mathrm dy\nonumber\\
&\leq \frac{2}{\eta}\cdot\int_0^{1.1\eta^{1/2}} \varphi_\eta(y-x)\cdot xy\cdot \frac{\tilde\Phi_{1-t}(y)}{\tilde\Phi_{1-s}(x)}\, \mathrm dy,\label{e.conditional Brown meander integral}
\end{align}
using in the last line that $1-e^{-r}\leq r$ for any $r\geq 0$.
Now we bound the fraction involving $\tilde\Phi$. Note that the numerator satisfies
$$\tilde\Phi_{1-t}(y) \leq \frac{y}{\sqrt{2\pi(1-t)}},$$
while for the denominator we have, using that the standard normal density is decreasing on the positive real line,
\begin{align*}
\tilde\Phi_{1-s}(x) = \P\left(N(0,1)\in (1-s)^{-1/2}[0,x]\right) \geq \frac{x}{\sqrt{2\pi(1-s)}}\exp\left(-\frac{x^2}{2(1-s)}\right)\geq \frac{x}{\sqrt{2\pi(1-s)}}\cdot e^{-2/5},
\end{align*}
as $x\leq 2\eta^{1/2}$ and $1-s\geq 5\eta$.
We also have the bound $\varphi_\eta(y-x) \leq (2\pi\eta)^{-1/2}$. Substituting these inequalities in\eqref{e.conditional Brown meander integral}, we get
\begin{align*}
\P\Big(\bme(t) < 2\eta^{1/2} \mid \bme\left((k-1)\eta\right)=x\Big)
&\leq \frac{2e^{2/5}}{\eta^{3/2}}\cdot\int_0^{1.1\eta^{1/2}} y^2(1-s)^{1/2}(2\pi(1-t))^{-1/2}\, \mathrm dy.
\end{align*}
Since $s=t-\eta$, we have $\sqrt{(1-s)/(1-t)}\leq \sqrt 2$.
Thus we may bound the last expression as
\begin{align*}
\frac{2e^{2/5}}{\eta^{3/2}\pi^{1/2}}\cdot\int_0^{1.1\eta^{1/2}} y^2\, \mathrm dy = \frac{2e^{2/5}(1.1)^3\eta^{3/2}}{3\eta^{3/2}\pi^{1/2}}.
\end{align*}
Overall, this gives that
$$\P\Big(\bme(t) < 1.1\eta^{1/2} \ \big| \  \bme\left(s\right)=x\Big)\leq \frac{2\times(1.1)^3e^{2/5}}{3\sqrt\pi} \leq \frac{3}{4},$$
the last inequality by numerical evaluation.
\end{proof}

\begin{proof}[Proof of Lemma~\ref{l.prob of return from above}]
We again use that, conditionally on $\bme\left(t-\eta\right) = x$, the distribution of $\bme$ on $[t-\eta, 1]$ is that of a Brownian motion $B$ started at $x$ conditioned to stay positive. This implies
\begin{align*}
\P\left(\inf_{s\in[t-\eta, 1]} \bme(s) < \eta^{1/2} \ \Big| \  \bme\left(t-\eta\right) =x\right)
%
%
&\leq \P\left(\inf_{s\in[t-\eta,1]} B(s) < \eta^{1/2}-x \ \Big| \  \inf_{s\in[t-\eta,1]} B(s)>-x\right),
\end{align*}
where $B$ is a standard Brownian motion on $[t-\eta,1]$ started at 0.
Now by the reflection principle and the fact that the normal density decreases on the positive real line, we know that
$$\P\left(\inf_{s\in[t-\eta,1]} B(s)>-x\right) = \P\big(|N(0,1-(t-\eta))| < x\big) \geq \frac{1.1\eta^{1/2}}{\sqrt{2\pi(1-(t-\eta))}} \exp\left(-\frac{(1.1)^2\eta}{2(1-(t-\eta))}\right)$$
since $x\geq 1.1\eta^{1/2}$; and also
\begin{align*}
\P\left(\inf_{s\in[t-\eta,1]} B(s)\in[-x, -x+\eta^{1/2}]\right) &= \P\Big(\big|N(0,1-(t-\eta))\big|\in [x-\eta^{1/2},x]\Big)\\
&\leq \P\Big(\big|N(0,1-(t-\eta))\big|\in [0.1\eta^{1/2}, 1.1\eta^{1/2}]\Big)\\
&\leq \frac{\eta^{1/2}}{\sqrt{2\pi(1-(t-\eta))}} \exp\left(-\frac{(0.1)^2\eta}{2(1-(t-\eta))}\right),
\end{align*}
again using the decreasing property of the normal density. So we obtain that the probability in the statement of Lemma~\ref{l.prob of return from above} is bounded above by 
$$\frac{1}{1.1}\exp\left(\frac{(1.1)^2\eta-(0.1)^2\eta}{2(1-(t-\eta))}\right) \leq \frac{1}{1.1}\exp\left(\frac{1.2}{20}\right)< \frac{1+0.09}{1.1} = 1-\delta$$
for a $\delta>0$ (where we have used that $e^x\leq 1+\frac{3}{2}x$ for $x<1/2$, as can be routinely checked). We have also used that $t<1-10\eta$.
\end{proof}

\end{document}